\documentclass[11pt]{article}
\usepackage[a4paper,margin=0.9in]{geometry}
\usepackage[T1]{fontenc}
\usepackage{lmodern}
\usepackage{amsmath,amssymb,amsthm,mathtools,mathrsfs,array,longtable}
\usepackage{microtype}
\usepackage[hidelinks]{hyperref}
\hypersetup{
  pdftitle={The Critical Semilinear Elliptic Equation with Isolated Boundary Singularities II},
  pdfauthor={Hua-Yang Wang and Jingang Xiong}
}

\allowdisplaybreaks[2]
\newtheorem{theorem}{Theorem}[section]
\newtheorem{lemma}[theorem]{Lemma}
\newtheorem{proposition}[theorem]{Proposition}
\newtheorem{corollary}[theorem]{Corollary}

\theoremstyle{remark}

\numberwithin{equation}{section}
\mathtoolsset{showonlyrefs}
\newcommand{\dd}{\mathop{}\!\mathrm{d}}

\title{\bfseries The Critical Semilinear Elliptic Equation with Isolated
Boundary Singularities II}
\author{Hua-Yang Wang\thanks{H.-Y. Wang was partially supported by NSFC
grant 12601183.}\quad
Jingang Xiong\thanks{J. Xiong was partially supported by NSFC grant 12325104.}}
\date{}

\begin{document}
\maketitle

\begin{abstract}
Continuing the work of the second author (2017), we study the Sobolev
critical semilinear elliptic equation in the half-space with an isolated
boundary singularity and zero Dirichlet boundary condition.
This paper addresses two open questions in this setting: the existence
of Delaunay type log-periodic solutions posed by del Pino--Musso--Pacard
(2007), and the asymptotic classification of singular solutions posed by
Bidaut-V\'eron--Ponce--V\'eron (2007).
We construct a global continuum of positive log-periodic solutions containing
the local bifurcation branch and prove that blow-up along this continuum occurs
at a uniquely determined period.
We also construct the corresponding concentrating family and prove its local uniqueness.
Consequently, the expected stationary asymptotic
classification fails, and no universal critical
scaling-invariant upper bound can hold throughout the half-space.
This behavior contrasts sharply with the classical interior singularity
theory of Caffarelli--Gidas--Spruck (1989).
\end{abstract}



\section{Introduction}

Considerable attention has been devoted to the precise asymptotic behavior of positive smooth solutions to the semilinear elliptic equation
\[
-\Delta u=u^q \quad \text{in } \Omega\setminus\{0\},
\]
where $\Omega\subset\mathbb R^N$ is an open set with $0\in\Omega$, $N\ge2$, and $q>1$. The interior singularity problem is now well understood in dimension $N=2$ and, for $N\ge3$, in the range $q\le\frac{N+2}{N-2}$; see the classical works of Lions \cite{Lions1980}, Aviles \cite{Aviles1987}, Gidas--Spruck \cite{GidasSpruck1981}, and Caffarelli--Gidas--Spruck \cite{CGS1989}. In particular, when $N\ge3$ and $q>\frac{N}{N-2}$, every nonremovable isolated singularity satisfies the scaling-invariant estimate
with some $C>0$:
\begin{equation} \label{eq:scl-inv}
u(x)\le C|x|^{-\frac{2}{q-1}}.
\end{equation}
In the subcritical range $q<\frac{N+2}{N-2}$, there exists $c>0$ such that
\[
\lim_{|x|\to0}|x|^{\frac{2}{q-1}}u(x)=c.
\]
At the critical exponent $q=\frac{N+2}{N-2}$, the asymptotic behavior becomes substantially richer: there exist $\tau\in\mathbb R$ and a positive Fowler profile $V_\varepsilon$ for which
\[
|x|^{\frac{2}{q-1}}u(x)
=
V_\varepsilon(-\log|x|+\tau)\bigl(1+o(1)\bigr)
\qquad\text{as }x\to0.
\]
Here $V_\varepsilon$ is either constant or periodic with period $T_\varepsilon$. This description was refined by Korevaar--Mazzeo--Pacard--Schoen \cite{KMPS1999}, while expansions of higher order were obtained by Han--Li--Li \cite{HanLiLi}. After a suitable translation, every nonconstant Fowler profile is \emph{bumpy}, in the sense that
\[
V_\varepsilon(-t)=V_\varepsilon(t),\qquad
\partial_tV_\varepsilon(t)<0
\quad\text{for }0<t<T_\varepsilon/2.
\]
The critical equation also has a natural geometric origin: it arises in conformal geometry as the equation for metrics of constant positive scalar curvature. In dimension four, it is also closely related to the Euclidean $SU(2)$ Yang--Mills equations; see Corrigan--Fairlie \cite{CorriganFairlie1977}, Gu \cite{Gu1980}, and Gidas \cite{Gidas1980}.

The corresponding problem for isolated singularities located on the boundary is considerably subtler. Let $\Omega\subset\mathbb R^N$ be a smooth domain and let $\Gamma\subset\partial\Omega$ be relatively open with $0\in\Gamma$. We consider positive solutions of
\[
-\Delta u=u^q\quad\text{in }\Omega,\qquad
u=0\quad\text{on }\Gamma\setminus\{0\}.
\]
For $1<q<\frac{N+2}{(N-2)_+}$, the boundary singularity problem has been investigated by Bidaut-Véron--Vivier \cite{BidautVeronVivier2000} and Bidaut-Véron--Ponce--Véron \cite{BVPV2007,BVPV2011}, who established boundary analogues of the results of Lions \cite{Lions1980} and Gidas--Spruck \cite{GidasSpruck1981}, respectively. 
In this setting, the dividing exponent $\frac{N}{N-2}$ for interior singularities is replaced by the boundary exponent $\frac{N+1}{N-1}$, first identified by Brézis--Turner \cite{BrezisTurner1977}. 
More precisely, when
\[
\frac{N+1}{N-1}<q<\frac{N+2}{(N-2)_+},
\]
a nonremovable isolated boundary singularity has precise leading asymptotics: the rescaled limit now is given by the unique positive solution $\phi$ of
\begin{equation}\label{eq:stationary-q}
\Delta_{\mathbb S^{N-1}}\phi-b_{N,q}\phi+\phi^q=0
 \quad \text{in } \mathbb S^{N-1}_+,\qquad
\phi=0
 \quad \text{on } \partial\mathbb S^{N-1}_+,
\end{equation}
where
\[
b_{N,q}
=\frac{2}{q-1}
\Bigl(
N-2-\frac{2}{q-1}
\Bigr),
\]
$\mathbb S^{N-1}$ denotes the unit sphere in $\mathbb R^N$,
\[
\mathbb S^{N-1}_+
:=
\{z\in\mathbb S^{N-1}:z_N>0\},
\]
and $\Delta_{\mathbb S^{N-1}}$ is the Laplace--Beltrami operator. The existence of boundary singular solutions in the range $q\ge\frac{N+1}{N-1}$, with $q$ sufficiently close to $\frac{N+1}{N-1}$, was established by del Pino--Musso--Pacard \cite{DMP2007}.

In contrast to the interior case, the asymptotic behavior of isolated boundary singularities in the Sobolev critical regime remains largely unresolved. 
Bidaut-Véron--Ponce--Véron \cite{BVPV2007} asked whether the above subcritical characterization remains valid at the critical exponent. 
In particular, they asked whether, for a nonremovable isolated boundary singularity, the rescaled solution still converges to the unique positive solution of \eqref{eq:stationary-q}. 
They also proved that \eqref{eq:stationary-q} continues to admit a unique positive solution at the critical exponent.  
Xiong \cite{Xiong2017} established the upper bound \eqref{eq:scl-inv} away from any fixed interior cone about the inward normal to $\partial\Omega$ at $0$. 
However, it has been unknown whether a priori estimate \eqref{eq:scl-inv} continues to hold uniformly inside such cones. 
By blow-up analysis, one immediately obtains the half-space model problem
\begin{equation}\label{eq:half-space}
-\Delta u=u^{\frac{N+2}{N-2}} \quad  \text{in } \mathbb R^N_+,\qquad
u=0 \quad \text{on } \partial\mathbb R^N_+\setminus\{0\},
\end{equation}
where
\(
\mathbb R^N_+:=\{x=(x',x_N)\in\mathbb R^N:x_N>0\}.
\)
The Emden--Fowler transformation is given by
\begin{equation}\label{eq:EF-transform}
t=-\log|x|,\qquad z=\frac{x}{|x|},\qquad
u(x)=|x|^{-\frac{N-2}{2}}v(t,z).
\end{equation}
It transforms \eqref{eq:half-space} into
\begin{equation}\label{eq:EF-critical}
 \begin{cases}
 -\partial_{tt}v-\Delta_{\mathbb S^{N-1}}v
 +\frac{(N-2)^2}{4}v
 =v^{\frac{N+2}{N-2}}
 & \text{in }\mathbb R\times\mathbb S^{N-1}_+,\\
 v=0
 & \text{on }\mathbb R\times\partial\mathbb S^{N-1}_+.
 \end{cases}
\end{equation}
We call a solution of \eqref{eq:EF-critical} stationary if it is
independent of $t$, and nonstationary otherwise.
Thus, a fundamental step toward understanding critical boundary singularities is to analyze the positive solutions of \eqref{eq:EF-critical}. Del Pino--Musso--Pacard
\cite{DMP2007} posed the following problem.

\medskip 

\noindent {\bf Open Problem 2 of \cite{DMP2007}}: 
\textit{Does \eqref{eq:EF-critical} admit a family of positive nonstationary solutions that are periodic in $t$ and depend on one parameter?}

\medskip

We shall  answer this problem  affirmatively, and thus negatively to the above mentioned classification problem of Bidaut-V\'eron--Ponce--V\'eron \cite{BVPV2007}. 
Moreover, we prove that there exist blow-up sequences of solutions showing that the critical upper bound of Xiong \cite{Xiong2017}, established outside interior cones, cannot be extended across these cones with a uniform constant.

To state our main theorem, we introduce some notations. 
We identify $T$-periodic functions with functions on $\Sigma_T=(\mathbb R/T\mathbb Z)\times\mathbb S^{N-1}_+$.
A $T$-periodic solution $v$ of \eqref{eq:EF-critical} is called
\emph{bumpy} if, after a translation in $t$,  
\begin{equation}\label{eq:bumpy}\tag{Bumpy}
v(-t,z)=v(t,z),\qquad
\partial_t v<0\quad \text{in}~~(0,T/2)\times\mathbb S^{N-1}_+.
\end{equation} 
\noeqref{eq:bumpy}
Let $\phi$ be the unique positive solution of
\eqref{eq:stationary-q} for $q=(N+2)/(N-2)$, and let
$(\psi_1,\lambda_1) \in H_0^1(\mathbb S^{N-1}_+)\times \mathbb R$ be the first eigenpair of
\[
 \mathscr A_\phi
 =-\Delta_{\mathbb S^{N-1}}+\frac{(N-2)^2}{4}
  -\frac{N+2}{N-2}\phi^{\frac{4}{N-2}},
\]
normalized by $\psi_1>0$ and
$\|\psi_1\|_{L^2(\mathbb S^{N-1}_+)}=1$.  
We introduce two important periods:
let
$T_0=2\pi/\sqrt{|\lambda_1|}$,
and let $T^*>0$ be the unique zero of
\begin{equation}\label{eq:intro-regular-part}
\mathscr M_N(T)
=
-2^{2-N}
+
2\sum_{\ell=1}^{\infty}
e^{-\frac{N-2}{2}\ell T}
\left[
(1-e^{-\ell T})^{2-N}
-
(1+e^{-\ell T})^{2-N}
\right].
\end{equation}
To study periodic solutions with different periods, we equip the space of pairs $(T,v)$ with the topology characterized by
$
(T_n,v_n)\to (T,v)
$ 
if and only if
\begin{equation}
\begin{aligned}
T_n&\longrightarrow T\quad &&\text{in}~~\mathbb R,\\  
v_n(T_n\,\cdot,\cdot)&\longrightarrow v(T\,\cdot,\cdot)\quad  &&\text{in}~~
C_{\mathrm{ax}}^{2,\alpha}(\overline{\Sigma}_1).
\end{aligned}
\end{equation}

\begin{theorem}\label{thm:local-bifurcation}
Let $N\geq3$ and $0<\alpha<\min\{1,4/(N-2)\}$.
Then the following hold.
\begin{enumerate}
\renewcommand{\labelenumi}{\textup{(\roman{enumi})}}

\item There exists, up to translation in $t$, a
unique smooth local curve of positive periodic solutions of \eqref{eq:EF-critical} bifurcating
from $(T_0,\phi)$. 

\item There exists a closed, connected, noncompact set
$
\mathscr C
\subset
(0,\infty)\times
C_{\mathrm{ax}}^{2,\alpha}
\bigl(\overline{\Sigma}_1\bigr)
$
consisting of positive periodic solutions to \eqref{eq:EF-critical}
satisfying $(T_0,\phi)\in\mathscr C$ and
\[
0<
\inf_{(T,v)\in\mathscr C}T
\leq
\sup_{(T,v)\in\mathscr C}T
<\infty,
\qquad
\sup_{(T,v)\in\mathscr C}
\int_{\Sigma_T}
v^{\frac{2N}{N-2}}\,\dd V
<\infty.
\]
Moreover, if $(T_n,v_n)\in\mathscr C$ and
$\|v_n\|_{L^\infty}\to\infty$, then $T_n\to T^*$.

\item The local curve in (i) is contained in the set $\mathscr C$, and every $(T,v)\in\mathscr C\setminus\{(T_0,\phi)\}$ is bumpy and has least period $T$.

\end{enumerate}
\end{theorem}

For part (i), the Crandall--Rabinowitz theorem \cite{CR1971} yields the local bifurcation once the simple kernel and the transversality condition are verified. 
The main issue is to obtain sufficiently precise spectral information on the linearized operator at $\phi$. 
This identifies $T_0$ as the first bifurcation period and provides the control needed to rule out smaller periods, prove half-period monotonicity, and obtain local uniqueness up to translation.
For parts (ii) and (iii), a Leray--Schauder index jump and a variant of Rabinowitz's global bifurcation theorem \cite{Rabinowitz1971} yield a noncompact continuum $\mathscr C$, but do not describe how noncompactness occurs.
Using the blow-up analysis in the spirit of Schoen \cite{Schoen1991}, Li \cite{Li1995}, Chen--Lin \cite{ChenLin1997}, and Han \cite{Han1991}, we first obtain uniform bounds for the period and energy on this global continuum.  
We then prove that every blow-up sequence concentrates only at $P$ and, after rescaling, converges to the standard bubble.
The limiting Pohozaev identity implies that any possible limiting period must be the unique positive zero of $\mathscr M_N$. 
Here $\mathscr M_N$, defined in \eqref{eq:intro-regular-part}, is the normalized regular part of the periodic Green function at $P=(0,e_N)$ and $e_N=(0,\ldots,0,1)$.

Set
\[
T_+:=\sup\{T>0:(T,v)\in\mathscr C\text{ for some }v\}.
\]
The exact value of $T_+$ remains unknown and we conjecture that
$T_+=T^*$.
The next theorem gives a complete description of the concentrating
branch near $T^*$.

\begin{theorem}\label{thm:main-resonant-family}
Let $N\geq3$. There exist $\mu_0>0$ and a family
\(
\{(T(\mu),v_\mu):0<\mu<\mu_0\}\subset\mathscr C
\)
such that, for every $\mu\in(0,\mu_0)$, $v_\mu$ is a positive bumpy
solution of \eqref{eq:EF-critical}, has least period $T(\mu)$, and
attains its maximum at $P$.
For $\alpha$ as in Theorem \ref{thm:local-bifurcation}, the map
\[
(0,\mu_0)\ni\mu\longmapsto
\bigl(T(\mu),v_\mu(T(\mu)\,\cdot,\cdot)\bigr)
\]
is smooth with values in $
(0,\infty)\times
C^{2,\alpha}_{\mathrm{ax}}(\overline{\Sigma_1}).
$
Moreover,
\[
T(\mu)\to T^*,
\qquad
\|v_\mu\|_{L^\infty}\to\infty
\quad\text{as }\mu\downarrow0.
\]
Finally, for $T$ sufficiently close to $T^*$, every positive bumpy
solution $v$ of \eqref{eq:EF-critical} with least period $T$, maximum
attained at $P$, and sufficiently large $v(P)$ satisfies
\(
(T,v)=(T(\mu),v_\mu)
\)
for a unique $\mu\in(0,\mu_0)$.
\end{theorem}

The construction is based on the Lyapunov--Schmidt reduction; see, among many others, Rey \cite{Rey1990}, del Pino--Dolbeault--Musso \cite{delPinoDolbeaultMusso2004}, Wei--Yan \cite{WeiYan2011CriticalGrowth}, Premoselli--V\'etois \cite{PremoselliVetois2022}, and Li--Vaira--Wei--Wu \cite{LiVairaWeiWu2025}. 
In the periodic setting, our construction is also related to Guo--Yan \cite{GuoYan2024}. 
The main difference is that the period $T$ is not prescribed, but enters the reduced problem as an additional parameter. 
After projecting a periodic bubble and solving the projected equation, the normalized reduced function converges in $C^1$ to a nonzero multiple of $\mathscr M_N(T)$. 
The simple zero at $T^*$ then determines a unique smooth curve $T=T(\mu)$. 
For local uniqueness, blow-up analysis and modulation reduce every solution in the stated regime to the same reduced problem, showing that it belongs to the constructed family, which forms the corresponding end of $\mathscr C$. 
It remains open whether every positive nonstationary periodic solution of \eqref{eq:EF-critical} is bumpy and belongs, after a translation in $t$, to $\mathscr C$.

\medskip
\noindent\textit{Organization of the paper.}
In Section \ref{sec:bifurcation-branches}, we prove
Theorem \ref{thm:local-bifurcation}(i), construct the continuum in
part (ii), and establish the bumpiness and least period assertions in
part (iii).
In Section \ref{sec:finite-period-compactification}, we analyze
blow-up sequences whose periods converge to a positive finite limit and
identify $T^*$ as the only possible limiting period.
The uniform energy and period bounds for bumpy solutions are proved in
Section \ref{sec:uniform-bounds}.
We then construct the concentrating family by Lyapunov--Schmidt reduction
in Section \ref{sec:resonant-end} and establish its local uniqueness and
connection to the global continuum in Section \ref{sec:local-uniqueness}.
The appendix contains the technical estimates needed for the construction.

\medskip
\noindent\textit{Declaration on the use of AI.}
The research direction, overall strategy, and proof methods were developed by the authors. Different AI models were used only to assist with technical details and language polishing. All AI generated suggestions were critically reviewed and independently verified by the authors before inclusion. The authors take full responsibility for the paper.

\medskip
\begingroup
\setlength{\LTpre}{0pt}
\setlength{\LTpost}{0pt}
\renewcommand{\arraystretch}{1.16}
\begin{longtable}{
@{}
p{0.20\linewidth}
@{\hspace{0.5em}}
>{\raggedright\arraybackslash}p{0.77\linewidth}
@{}
}
\multicolumn{2}{@{}l}{\textit{Notation.}}\\*[4pt]
\endfirsthead
\endhead
\textbf{$\mathbb S^{N-1}_+$} &
$\mathbb S^{N-1}_+=\{z\in\mathbb S^{N-1}:z_N>0\}$.\\[2pt]

\textbf{$\mathbb T,\,\Sigma_T$} &
$\mathbb T=\mathbb R/\mathbb Z$,
$\Sigma_T=(\mathbb R/T\mathbb Z)\times\mathbb S^{N-1}_+$.\\[2pt]

\textbf{$X_{\mathrm{ax}}$} &
For a function space $X$ on $\mathbb S^{N-1}_+$ or $\Sigma_T$,
$X_{\mathrm{ax}}$ is the closed subspace of functions invariant under
$z'\mapsto Rz'$ for every $R\in O(N-1)$, where $z=(z',z_N)$.\\[2pt]

\textbf{$T_0$} &
$T_0=2\pi/\sqrt{|\lambda_1|}$.\\[2pt]

\textbf{$T^*$} &
The unique positive zero of $\mathscr M_N$ in
\eqref{eq:intro-regular-part}.\\[2pt]

\textbf{$g_{\mathrm{cyl}}$} &
$g_{\mathrm{cyl}}=\dd t^2+g_{\mathbb S^{N-1}}$.\\[2pt]

\textbf{$L_g$} &
$L_g:=\Delta_g-\frac{N-2}{4(N-1)}R_g$ is the conformal
Laplacian, where $\Delta_g$ is the Laplace--Beltrami operator and $R_g$
is the scalar curvature. In particular,
$L_{g_{\mathrm{cyl}}}
=\partial_{tt}+\Delta_{\mathbb S^{N-1}}-\frac{(N-2)^2}{4}$.\\[2pt]

\textbf{$P,\,e_N$} &
$P=(0,e_N)\in\Sigma_T$, $e_N=(0,\ldots,0,1) \in \mathbb R^N$.\\[2pt]

\textbf{$d_T,\,r(X)$} &
$d_T$ is the distance on $\Sigma_T$ induced by $g_{\mathrm{cyl}}$;
$r(X):=d_T(X,P)$.\\[2pt]

\textbf{$\delta(X)$} &
$\delta(X):=\min\{1,\operatorname{dist}_{d_T}(X,\partial\Sigma_T)\}$.\\[2pt]

\textbf{$\mathcal B_r(X)$} &
The geodesic ball $\{Y\in\Sigma_T:d_T(X,Y)<r\}$.\\[2pt]

\textbf{$B_r(x)$} &
The Euclidean ball of radius $r$ centered at $x$.\\[2pt]

\textbf{$\langle\cdot,\cdot\rangle_T,\,\|\cdot\|_T$} &
$\langle u,v\rangle_T
:=\int_{\Sigma_T}
\bigl(\nabla u\cdot\nabla v+\frac{(N-2)^2}{4}uv\bigr)\,\dd V,~ \|u\|_T^2:=\langle u,u\rangle_T,\forall u,v\in H^1_0(\Sigma_T).$\\[2pt]

\textbf{$U_{\lambda,\xi},\,U_{\lambda}$} &
\(
 U_{\lambda,\xi}(y)
 :=\lambda^{\frac{N-2}{2}}
 \left(1+\frac{\lambda^2|y-\xi|^2}{N(N-2)}\right)^{-\frac{N-2}{2}}\) 
 for \(\lambda>0,  \xi\in\mathbb R^N;\)
 $U_\lambda:=U_{\lambda,0}$.\\[2pt]

\textbf{$\lambda_*,\,\alpha_N$} &
$\lambda_*:=\sqrt{N(N-2)}$,
$\alpha_N:=\lambda_*^{\frac{N-2}{2}}$; thus
$U_{\lambda_*}(y)
=\alpha_N(1+|y|^2)^{-\frac{N-2}{2}}$.\\[2pt]

\textbf{$A\asymp B$} &
For $A,B\ge0$, $C^{-1}B\le A\le CB$ for some $C\ge1$ independent of parameters.\\
\end{longtable}
\endgroup

\begin{samepage}
\section{Local bifurcation and global continuation}
\label{sec:bifurcation-branches}

The purpose of this section is to prove
Theorem \ref{thm:local-bifurcation}(i) and construct the continuum in
Theorem \ref{thm:local-bifurcation}(ii). 
For a $T$-periodic solution
$v$ of \eqref{eq:EF-critical}, define
\[
 w(s,z)=v(Ts,z).
\]
Then $w$ is $1$-periodic in $s$ and satisfies
\begin{equation}\tag{$P_T$}\label{eq:fixed-cylinder}
 \begin{cases}
 -T^{-2}\partial_{ss}w
 -\Delta_{\mathbb S^{N-1}}w+\frac{(N-2)^2}{4}w=w^{\frac{N+2}{N-2}}
 &\text{in }\mathbb T\times\mathbb S^{N-1}_+,\\
 w=0&\text{on }\mathbb T\times\partial\mathbb S^{N-1}_+.
 \end{cases}
\end{equation}
Conversely, if $w$ solves \eqref{eq:fixed-cylinder}, then
$v(t,z)=w(t/T,z)$ solves \eqref{eq:EF-critical}.  By Xiong \cite[Proposition 1.1]{Xiong2017}, every positive
periodic solution is axially symmetric.

\end{samepage}

\subsection{Local bifurcation}
\label{sec:local-branch}

Recall that $\phi$ is the unique positive solution of
\eqref{eq:stationary-q} for $q=(N+2)/(N-2)$. 
The linearized operator is defined by
\[
 \mathscr A_\phi=-\Delta_{\mathbb S^{N-1}}+\frac{(N-2)^2}{4}-\frac{N+2}{N-2}\phi^{\frac{4}{N-2}},\qquad
 \operatorname{Dom}(\mathscr A_\phi)
 =H^2_{\mathrm{ax}}(\mathbb S^{N-1}_+)\cap
 H^1_{0,\mathrm{ax}}(\mathbb S^{N-1}_+).
\]
The corresponding quadratic form on $H^1_{0,\mathrm{ax}}(\mathbb S^{N-1}_+)$ is
\[
 Q_\phi[h]:=\int_{\mathbb S^{N-1}_+}
 \bigl(|\nabla h|^2+\frac{(N-2)^2}{4}h^2-\frac{N+2}{N-2}\phi^{\frac{4}{N-2}}h^2\bigr)\dd\sigma.
\]
Recall that the Morse index is the number of negative Dirichlet eigenvalues, counted with multiplicity.
Let
$\lambda_1\leq\lambda_2\leq\cdots$ be the eigenvalues of
$\mathscr A_\phi$ in
$H^1_{0,\mathrm{ax}}(\mathbb S^{N-1}_+)$, and choose corresponding orthonormal eigenfunctions $\{\psi_j\}_{j\geq1}$ in $L^2_{\mathrm{ax}}(\mathbb S^{N-1}_+)$:
\[
 \mathscr A_\phi\psi_j=\lambda_j\psi_j
 \quad\text{in }\mathbb S^{N-1}_+,\qquad
 \psi_j=0\quad\text{on }\partial\mathbb S^{N-1}_+.
\]
The nondegeneracy of $\mathscr A_{\phi}$ in $H^1_{0,\mathrm{ax}}(\mathbb S^{N-1}_+)$ was proved in Shioji--Watanabe \cite[Theorem 12]{SW2016}.  
However, we give a different proof that also determines the Morse index.

\begin{lemma}\label{lem:spectrum}
The eigenvalues $\lambda_j$, $j\geq1$, are simple, and
\begin{equation}\label{eq:spectral-gap}
\lambda_1<0<\lambda_2.
\end{equation}
Consequently, $\mathscr A_\phi$ has Morse index one and trivial kernel.
\end{lemma}

\begin{proof}
We first prove that $\mathscr A_\phi$ has Morse index one.
Consider the quotient
\[
 \mathcal R(u):=
 \frac{\displaystyle\int_{\mathbb S^{N-1}_+}
 (|\nabla u|^2+\frac{(N-2)^2}{4}u^2)\dd\sigma}
 {\displaystyle\left(\int_{\mathbb S^{N-1}_+}
 |u|^{\frac{2N}{N-2}}\dd\sigma\right)^{\frac{N-2}{N}}},
 \qquad 0\ne u\in H^1_{0,\mathrm{ax}}(\mathbb S^{N-1}_+).
\]
The embedding
$H^1_0(\mathbb S^{N-1}_+)\hookrightarrow
L^{\frac{2N}{N-2}}(\mathbb S^{N-1}_+)$
is compact, so the direct method gives a minimizer of $\mathcal R$.
Replacing it by its absolute value and applying the strong maximum
principle, we may assume that it is positive. 
The uniqueness theorem of
Bidaut-Véron--Ponce--Véron \cite{BVPV2007} therefore identifies the
minimizer with $\phi$.
If
\[
 \int_{\mathbb S^{N-1}_+}
 \phi^{\frac{N+2}{N-2}}h\dd\sigma=0,
\]
then 
\[
 0\leq
 \left.\frac{\dd^2}{\dd\tau^2}
 \mathcal R(\phi+\tau h)\right|_{\tau=0}
 =
 2\left(
 \int_{\mathbb S^{N-1}_+}\phi^{\frac{2N}{N-2}}\dd\sigma
 \right)^{-\frac{N-2}{N}}Q_\phi[h].
\]
The above condition on $h$ defines a subspace of codimension one, while
\[
 Q_\phi[\phi]
 =-\frac{4}{N-2}\int_{\mathbb S^{N-1}_+}\phi^{\frac{2N}{N-2}}\dd\sigma<0.
\]
Hence the Morse index is one, and
$\lambda_1<0\leq\lambda_2$.

To prove simplicity of the spectrum, let
$\theta=d_{\mathbb S^{N-1}}(z,e_N)$, and use primes
for derivatives of $\theta$.  An eigenfunction $h$ with
eigenvalue $\lambda$ satisfies
\[
 -h''-(N-2)\cot\theta\,h'
 +\left(\frac{(N-2)^2}{4}
 -\frac{N+2}{N-2}\phi^{\frac{4}{N-2}}\right)h
 =\lambda h,\qquad h'(0)=h(\pi/2)=0.
\]
If $h_1,h_2$ are regular eigenfunctions for the same $\lambda$, then
\[
 \bigl(\sin^{N-2}\theta\,(h_1h_2'-h_1'h_2)\bigr)'=0.
\]
Regularity at $0$ makes the weighted Wronskian vanish, so
$h_1,h_2$ are linearly dependent.  Thus the eigenvalues are
simple.

It remains to prove that $0$ is not an eigenvalue.  If $\mathscr A_\phi h=0$ for an axial
$h\ne0$, then $0=\lambda_2$.  By the
Sturm oscillation theorem
(cf. Zettl \cite[Theorem 10.12.1]{Zettl2005}), $h$ has a unique zero
$\theta_0\in(0,\pi/2)$.
After changing its sign if necessary, $h>0$ on $(0,\theta_0)$ and $h<0$ on
$(\theta_0,\pi/2)$.
By the method of moving planes, $\phi=\phi(\theta)$ and
$\phi'<0$ on $(0,\pi/2)$.
Set $Y=-\phi'/\phi>0$ and $S=Y'$.  
Using the stationary equation \eqref{eq:stationary-q}, we obtain
\[
 S=Y^2-(N-2)\cot\theta\,Y+\phi^{\frac{4}{N-2}}-\frac{(N-2)^2}{4},
 \qquad
 S(0)=-\frac{\phi''(0)}{\phi(0)}
 =\frac{4\phi(0)^{\frac{4}{N-2}}-{(N-2)^2}}{4N-4}>0.
\]
The inequality is strict because $\phi$ attains its maximum at $0$:
equality would give the same regular Cauchy data as the constant
solution $\left(\frac{N-2}{2}\right)^{\frac{N-2}{2}}$, contradicting $\phi(\pi/2)=0$.  Direct
differentiation gives
\[
 S'=\Bigl[\Bigl(2-\frac4{N-2}\Bigr)Y
 -(N-2)\cot\theta\Bigr]S
 +\Bigl[\frac{2Y}{\sqrt{N-2}}-\sqrt{N-2}\cot\theta\Bigr]^2Y.
\]
Since $S(\varepsilon)>0$ for small $\varepsilon>0$, the integrating
factor formula gives $S=Y'>0$ on $(0,\pi/2)$.  Thus $Y^2$ is strictly
increasing.  Testing $\mathscr A_\phi h=0$ with $\phi$ and
$\phi^{\frac{N+2}{N-2}}-\frac{(N-2)^2}{4}\phi$, and using $\phi^{\frac{6-N}{N-2}}(\phi')^2=\phi^{\frac{N+2}{N-2}}Y^2$, we obtain
\[
 \int_0^{\pi/2}\sin^{N-2}\theta\,\phi^{\frac{N+2}{N-2}}h\dd\theta=0,
 \qquad
 \int_0^{\pi/2}\sin^{N-2}\theta\,\phi^{\frac{N+2}{N-2}}hY^2\dd\theta=0.
\]
Then we have
\[
 0=\int_0^{\pi/2}\sin^{N-2}\theta\,\phi^{\frac{N+2}{N-2}}h
 \bigl(Y^2-Y(\theta_0)^2\bigr)\dd\theta<0.
\]
This contradiction proves $\lambda_1<0<\lambda_2$.
\end{proof}

Fix $0<\alpha<\min\{1,4/(N-2)\}$ and set
\[
\begin{aligned}
 \mathscr X&:=\Bigl\{w\in C^{2,\alpha}_{\mathrm{ax}}(\overline{\Sigma_1}):
 w(-s,z)=w(s,z),\ w|_{\partial\Sigma_1}=0\Bigr\},\\
 \mathscr Y&:=\Bigl\{f\in C^{0,\alpha}_{\mathrm{ax}}(\overline{\Sigma_1}):
 f(-s,z)=f(s,z)\Bigr\},\\
 \mathscr U&:=\Bigl\{w\in\mathscr X:
 \inf_{\mathbb T\times\overline{\mathbb S^{N-1}_+}}
 \frac{w}{z_N}>0\Bigr\}.
\end{aligned}
\]
Here and below, $w/z_N$ denotes its continuous extension to $z_N=0$.

\begin{lemma}\label{lem:power-map}
The set $\mathscr U$ is open in $\mathscr X$, contains $\phi$, and the
following map is smooth:
\[
 \mathscr U\longrightarrow\mathscr Y,\qquad w\longmapsto w^{\frac{N+2}{N-2}}.
\]
\end{lemma}

\begin{proof}
In a boundary chart $(y',z_N)$, we use the zero trace to write
\begin{equation}\label{eq:boundary-division}
 \frac{w(s,y',z_N)}{z_N}
 =\int_0^1\partial_{z_N}w(s,y',\tau z_N)\dd\tau,\qquad
 \Bigl\|\frac w{z_N}\Bigr\|_{C^{1,\alpha}}
 \leq C\|w\|_{C^{2,\alpha}}.
\end{equation}
By the maximum principle, the Hopf lemma, and
\eqref{eq:boundary-division}, we have $\phi/z_N>0$ on
$\overline{\mathbb S^{N-1}_+}$; hence $\phi\in\mathscr U$, and the first
two assertions follow.  For $m\geq1$,
\begin{equation}
 D^m(w^{\frac{N+2}{N-2}})[h_1,\ldots,h_m]
 =\prod_{j=0}^{m-1}\left(\frac{N+2}{N-2}-j\right) z_N^{\frac{N+2}{N-2}}
   \Bigl(\frac w{z_N}\Bigr)^{\frac{N+2}{N-2}-m}
   \prod_{j=1}^m\frac{h_j}{z_N}.
 \label{eq:power-derivative}
\end{equation}
The middle factor in
\eqref{eq:power-derivative} is a smooth composition
on the positive cone, \eqref{eq:boundary-division} controls all remaining
quotients, and $z_N^{{(N+2)}/{(N-2)}}\in C^{1,\alpha}$ by the choice of $\alpha$.  Hence every
derivative in \eqref{eq:power-derivative} is a bounded multilinear map
from $\mathscr X^m$ to $\mathscr Y$.
\end{proof}

For $(T,w)\in(0,\infty)\times\mathscr U$, define
\begin{equation}\label{eq:local-residual}
 \mathscr F(T,w):=-T^{-2}\partial_s^2w-\Delta_{\mathbb S^{N-1}}w
 +\frac{(N-2)^2}{4}w-w^{\frac{N+2}{N-2}}.
\end{equation}
By Lemma \ref{lem:power-map}, $\mathscr F$ is smooth.  Moreover,
$\mathscr F(T,\phi)=0$ for every $T>0$.  
Recall that $T_0=2\pi/\sqrt{-\lambda_1}$ and set
\begin{equation*}
 \mathscr L_{T_0,\phi}=D_w\mathscr F(T_0,\phi)
 =-T_0^{-2}\partial_{ss}+\mathscr A_\phi.
\end{equation*}
\begin{lemma}\label{lem:simple-kernel}
The operator $\mathscr L_{T_0,\phi}:\mathscr X\to\mathscr Y$ is Fredholm
of index zero and
\begin{align}
 \ker \mathscr L_{T_0,\phi}&=\operatorname{span}\{\psi_1(z)\cos(2\pi s)\},
 \label{eq:kernel}\\
 \operatorname{Ran}\mathscr L_{T_0,\phi}
 &=\biggl\{f\in\mathscr Y:
 \int_{\Sigma_1}
 f(s,z)\psi_1(z)\cos(2\pi s)\dd s\dd\sigma=0
 \biggr\}.
 \label{eq:range}
\end{align}
Moreover,
\begin{equation}\label{eq:transversality}
 D_{Tw}\mathscr F(T_0,\phi)[\psi_1(z)\cos(2\pi s)]
 \notin\operatorname{Ran}\mathscr L_{T_0,\phi}.
\end{equation}
\end{lemma}

\begin{proof}
Let $\{\psi_j\}_{j\geq1}$ be the eigenfunctions fixed
before Lemma \ref{lem:spectrum}.  
By separation of variables, the spectrum of
$\mathscr L_{T_0,\phi}$ on the even, axially symmetric subspace of
$H_0^1(\Sigma_1)$ consists of the eigenvalues
$4\pi^2m^2/T_0^2+\lambda_j$, with corresponding eigenfunctions
$\psi_j(z)\cos(2\pi ms)$, where $j\geq1$ and $m\geq0$.  
In particular,
\[
 \mathscr L_{T_0,\phi}\bigl(\psi_j(z)\cos(2\pi ms)\bigr)
 =\biggl(\frac{4\pi^2m^2}{T_0^2}+\lambda_j\biggr)\psi_j(z)\cos(2\pi ms).
\]
Since $4\pi^2m^2/T_0^2+\lambda_1=(4\pi^2/T_0^2)(m^2-1)$, while
$4\pi^2m^2/T_0^2+\lambda_j\geq\lambda_2>0$ for $j\geq2$, the only
eigenfunction with eigenvalue zero is
$\psi_1(z)\cos(2\pi s)$.  
By the spectral theorem for self-adjoint operators and the preceding spectral gap,
we have
\[
 \operatorname{Ran}\mathscr L_{T_0,\phi}
 =\bigl(\ker\mathscr L_{T_0,\phi}\bigr)^\perp
 \quad\text{in }L^2(\Sigma_1).
\]
If $f\in\mathscr Y$ satisfies the condition in \eqref{eq:range},
the standard elliptic regularity shows that the weak solution of
$\mathscr L_{T_0,\phi}w=f$ belongs to $\mathscr X$. 
Hence
\eqref{eq:range} holds, and the range is closed and of codimension one
in $\mathscr Y$.  
Together with \eqref{eq:kernel}, this proves that
$\mathscr L_{T_0,\phi}:\mathscr X\to\mathscr Y$ is Fredholm of index
zero.  
Finally,
\[
 D_{Tw}\mathscr F(T_0,\phi)[\psi_1(z)\cos(2\pi s)]
 =-\frac{8\pi^2}{T_0^3}\psi_1(z)\cos(2\pi s),
\]
which proves \eqref{eq:transversality}.
\end{proof}

The following theorem strengthens
Theorem \ref{thm:local-bifurcation}(i) by including the expansion of
the local curve, its qualitative properties, and the complete local
classification up to translation.
\begin{theorem}\label{thm:local-curve}
There exist $\varepsilon_0>0$ and a smooth curve
\[
 (-\varepsilon_0,\varepsilon_0)\ni\varepsilon
 \longmapsto (T_\varepsilon,w_\varepsilon)
 \in(0,\infty)\times\mathscr U
\]
of solutions of \eqref{eq:fixed-cylinder}, with
$(T_0,w_0)=(T_0,\phi)$, such that
\begin{equation}\label{eq:CR-curve}
 T_\varepsilon=T_0+O(\varepsilon),
 \qquad
 w_\varepsilon=\phi+\varepsilon\psi_1(z)\cos(2\pi s)+O(\varepsilon^2)
 \quad\text{in }C^{2,\alpha}.
\end{equation}
Set $v_\varepsilon(t,z)=w_\varepsilon(t/T_\varepsilon,z)$.  For every
sufficiently small $\varepsilon\ne0$, the function $v_\varepsilon$ is
a positive bumpy solution of \eqref{eq:EF-critical} with least period
$T_\varepsilon$.  Moreover, every positive nonstationary periodic
solution sufficiently close to $(T_0,\phi)$ is, after a translation in
$t$, a member of this curve.
\end{theorem}

\begin{proof}
Let
\[
 \widehat w_{\varepsilon,1}(z)
 =2\int_0^1w_\varepsilon(s,z)\cos(2\pi s)\dd s.
\]
By \eqref{eq:CR-curve}, for $0<|\varepsilon| \ll 1 $,
\begin{equation}\label{eq:first-mode}
 \widehat w_{\varepsilon,1}
 =\varepsilon\psi_1+O(\varepsilon^2) \ne 0.
\end{equation}
If $\tau>0$ is a period, then
\begin{equation*}
 0=\int_0^1
 \bigl(w_\varepsilon(s+\tau,z)-w_\varepsilon(s,z)\bigr)e^{-2\pi is}\dd s
 =(e^{2\pi i\tau}-1)\int_0^1w_\varepsilon(s,z)e^{-2\pi is}\dd s.
\end{equation*}
By evenness and \eqref{eq:first-mode}, the last Fourier coefficient is
nonzero.  Thus $e^{2\pi i\tau}=1$, and the least period is $1$.
Consequently, $v_\varepsilon$ has least period $T_\varepsilon$.

We next prove that $v_\varepsilon$ is bumpy for $\varepsilon\ne0$.
Let $\mathscr H=(0,1/2)\times\mathbb S^{N-1}_+$ and consider
\[
 \mathscr L_{T,w}
 =-T^{-2}\partial_{ss}-\Delta_{\mathbb S^{N-1}}+\frac{(N-2)^2}{4}-\frac{N+2}{N-2}w^{\frac{4}{N-2}}
\]
with zero Dirichlet data on $\partial\mathscr H$.  Denote its first two
eigenvalues by $\lambda_1(T,w)\leq\lambda_2(T,w)$.  Since
\[
 \mathscr L_{T_0,\phi}\bigl(\psi_1(z)\sin(2\pi s)\bigr)=0,
 \qquad \psi_1(z)\sin(2\pi s)>0\quad\text{in }\mathscr H,
\]
zero is the simple principal eigenvalue; hence
$\lambda_1(T_0,\phi)=0<\lambda_2(T_0,\phi)$.
Differentiating \eqref{eq:fixed-cylinder} and using
\eqref{eq:CR-curve}, we obtain
\[
 \mathscr L_{T_\varepsilon,w_\varepsilon}
 \partial_s w_\varepsilon=0,
 \qquad
 -\partial_s w_\varepsilon
 =2\pi\varepsilon\psi_1(z)\sin(2\pi s)+O(\varepsilon^2)
 \quad\text{in }C^{1,\alpha}.
\]
For $0<|\varepsilon|\ll1$, the second relation gives
$\partial_s w_\varepsilon\not\equiv0$. By continuity of the spectrum,
we also have $\lambda_2(T_\varepsilon,w_\varepsilon)>0$. Thus zero remains the
simple principal eigenvalue, and the expansion implies
\[
 -\operatorname{sgn}(\varepsilon)\partial_s w_\varepsilon>0
 \quad\text{in }\mathscr H.
\]
For $\varepsilon<0$, a translation by half a period preserves evenness and reverses
this sign.  Hence $v_\varepsilon$ is bumpy for every $\varepsilon\ne0$.

Finally, we prove the local classification up to translation.
Let $(T,v)$ be a positive nonstationary periodic solution
sufficiently close to $(T_0,\phi)$, and set $w(s,z)=v(Ts,z)$.  As noted
above, $w$ is axial.  For $\tau\in\mathbb T$, define
\[
 C(\tau)=\int_{\Sigma_1}w(s+\tau,z)\psi_1(z)e^{-2\pi i s}
 \,\dd s\dd\sigma=e^{2\pi i\tau}C(0).
\]
Choose $\tau$ so that $C(\tau)\in\mathbb R$, and set
$\widetilde w(s,z)=w(s+\tau,z)$ and
$q(s,z)=\widetilde w(s,z)-\widetilde w(-s,z)$.  Then $q$ is axial and
odd in $s$, has zero Dirichlet trace, and
\[
 \int_{\Sigma_1}q(s,z)\psi_1(z)\sin(2\pi s)\,\dd s\dd\sigma=0.
\]
For $j,m\geq1$, the odd Fourier basis functions satisfy
\[
 \mathscr L_{T_0,\phi}\bigl(\psi_j(z)\sin(2\pi ms)\bigr)
 =\biggl(\lambda_j+\frac{4\pi^2m^2}{T_0^2}\biggr)
  \psi_j(z)\sin(2\pi ms),
\]
and the coefficient vanishes only for $(j,m)=(1,1)$.
Since $q$ is orthogonal to $\psi_1(z)\sin(2\pi s)$,  there exists $c_0>0$ such that
\[
 \langle\mathscr L_{T_0,\phi}q,q\rangle_{L^2}
 \geq c_0\|q\|_{H^1(\Sigma_1)}^2.
\]
Subtracting the equations for $\widetilde w(s,z)$ and
$\widetilde w(-s,z)$, we obtain
\[
\begin{aligned}
 -T^{-2}\partial_{ss}q-\Delta_{\mathbb S^{N-1}}q
 +\left(\frac{(N-2)^2}{4}-b\right)q=0,
\end{aligned}
\]
where
\[
\begin{aligned}
 b(s,z)=\frac{N+2}{N-2}\int_0^1
 \bigl[\theta\widetilde w(s,z)+(1-\theta)\widetilde w(-s,z)\bigr]^{\frac{4}{N-2}}
 \,\dd\theta.
\end{aligned}
\]
Note
\[
 \delta(T,w):=|T^{-2}-T_0^{-2}|
 +\Big\|b-\frac{N+2}{N-2}\phi^{\frac{4}{N-2}}\Big\|_{L^\infty}\longrightarrow0
 \quad\text{as }(T,w)\to(T_0,\phi).
\]
Testing the above equation against $q$, we obtain
\[
\begin{aligned}
 c_0\|q\|_{H^1}^2
 &\leq\langle\mathscr L_{T_0,\phi}q,q\rangle_{L^2}\\
 &=\bigl(T_0^{-2}-T^{-2}\bigr)\|\partial_s q\|_{L^2}^2
   +\int_{\Sigma_1}\left(b-\frac{N+2}{N-2}
   \phi^{\frac{4}{N-2}}\right)q^2\,\dd s\dd\sigma\\
 &\leq C\delta(T,w)\|q\|_{H^1}^2.
\end{aligned}
\]
Thus $q=0$ when $(T,w)$ is sufficiently close to $(T_0,\phi)$, so
$\widetilde w\in\mathscr X$.  By the local classification
already proved in this space, we have
$(T,\widetilde w)=(T_\varepsilon,w_\varepsilon)$ for some
$\varepsilon\ne0$.  Returning to the original variable, we have
$(T,v(\cdot+T\tau,\cdot))=(T_\varepsilon,v_\varepsilon)$, as claimed.
\end{proof}

\subsection{Global continuation}
\label{sec:global-continuation}

\begingroup
For $m\geq1$, set
\begin{equation}\label{eq:global-resonances}
 T_m=mT_0=\frac{2\pi m}{\sqrt{-\lambda_1}}.
\end{equation}
\begin{proposition}\label{prop:global-continuum}
Let $\mathscr C$ be the connected component containing $(T_0,\phi)$ of
the closure, in $(0,\infty)\times C^0_{\mathrm{ax}}(\overline{\Sigma_1})$, of all
positive nonstationary solutions
$(T,w)\in(0,\infty)\times\mathscr X$ of \eqref{eq:fixed-cylinder}.
Then $\mathscr C$ consists of positive solutions and contains the local
curve in Theorem \ref{thm:local-bifurcation}.  Moreover:
\begin{enumerate}
 \renewcommand{\labelenumi}{(\roman{enumi})}
 \item every nonstationary member of $\mathscr C$ is bumpy and has
 least period $1$ in $s$;
 \item the only stationary member of $\mathscr C$ is $(T_0,\phi)$;
 \item $\mathscr C$ is not relatively compact in
 $(0,\infty)\times C^0_{\mathrm{ax}}(\mathbb T\times
 \overline{\mathbb S^{N-1}_+})$.
\end{enumerate}
Consequently, $\mathscr C$ contains a sequence $(T_n,w_n)$ such that
\begin{equation}\label{eq:global-escape}
  \text{either}\quad
 T_n\longrightarrow\infty
 \quad\text{or}\quad 
 \|w_n\|_{L^\infty}\longrightarrow\infty.
\end{equation}
\end{proposition}

The case $T_n\to\infty$ in \eqref{eq:global-escape} will be excluded in
Proposition \ref{prop:uniform-bounds}.
Let
\[
\mathscr X_0=\bigl\{u\in C_{\mathrm{ax}}^{0}(\overline{\Sigma_1}):
 u(-s,z)=u(s,z),\ u|_{\partial\Sigma_1}=0\bigr\}
\]
be equipped with the norm $\|u\|_{\mathscr X_0}=\|u\|_{C^0}$.  For $T>0$, set
\[
 \widehat g_T:=T^2\dd s^2+g_{\mathbb S^{N-1}},
 \qquad -L_{\widehat g_T}=-T^{-2}\partial_{ss}-\Delta_{\mathbb S^{N-1}}+\frac{(N-2)^2}{4}.
\]
We impose periodic boundary conditions in $s$ and zero Dirichlet data on the
lateral boundary.  
Note that
$-L_{\widehat g_T}$ is invertible.  Fix $0<\tau_0<\tau_1<\infty$ and $r>N$.
By the $W^{2,r}$ estimate of Gilbarg--Trudinger
\cite[Theorem 9.15]{GT2001}, there exists $C=C(\tau_0,\tau_1,r)>0$ such that
\[
 \|(-L_{\widehat g_T})^{-1}f\|_{W^{2,r}}
 \leq C\|f\|_{L^r},\qquad \tau_0\leq T\leq\tau_1.
\]
Since $W^{2,r}\Subset C^0$, the map
$(T,f)\mapsto(-L_{\widehat g_T})^{-1}f$ maps bounded subsets of
$[\tau_0,\tau_1]\times\mathscr X_0$ into relatively compact subsets of
$\mathscr X_0$.  It also depends continuously on $(T,f)$.
Indeed, subtracting
$-L_{\widehat g_T}u=f$ and
$-L_{\widehat g_{\widetilde T}}\widetilde u=\widetilde f$ and using the same
estimate, we obtain
\[
 \|u-\widetilde u\|_{W^{2,r}}
 \leq C\bigl(\|f-\widetilde f\|_{L^r}
 +|T^{-2}-\widetilde T^{-2}|
 \|\widetilde u\|_{W^{2,r}}\bigr).
\]
For $s\in\mathbb R$, set $s_+:=\max\{s,0\}$.
Since $\frac{N+2}{N-2}>1$, the superposition map
$u\mapsto (u_+)^{\frac{N+2}{N-2}}$ on $\mathscr X_0$ is $C^1$, with derivative
$h\mapsto \frac{N+2}{N-2}(u_+)^{\frac{4}{N-2}}h$ at $u\in\mathscr X_0$.
Writing $w=\phi+\zeta$, define
\begin{equation}\label{eq:global-fixed-point-map}
 \mathscr K(T,\zeta)=(-L_{\widehat g_T})^{-1}
 \bigl((\phi+\zeta)_+^{\frac{N+2}{N-2}}-\phi^{\frac{N+2}{N-2}}\bigr).
\end{equation}
Thus $\mathscr K$ is continuous, compact on bounded parameter strips,
and continuously Fr\'echet differentiable in $\zeta$.  Moreover,
$\zeta=\mathscr K(T,\zeta)$ is equivalent to
\begin{equation}\label{eq:positive-extension}
 -L_{\widehat g_T} w=(w_+)^{\frac{N+2}{N-2}}.
\end{equation}
Set $w_-=(-w)_+$. Testing \eqref{eq:positive-extension}
against $w_-$, we obtain
\[
 0=-\int_{\mathbb T\times\mathbb S^{N-1}_+}
 \bigl(T^{-2}|\partial_s w_-|^2
 +|\nabla_{\mathbb S^{N-1}}w_-|^2+\frac{(N-2)^2}{4}|w_-|^2\bigr)\dd s\dd\sigma.
\]
Hence $w\geq0$; if $w\not\equiv0$, the strong maximum principle
(Gilbarg--Trudinger \cite[Theorem 3.5]{GT2001}) gives $w>0$ in $\Sigma_1$.
By evaluating \eqref{eq:positive-extension} at an interior maximum $M>0$,
we obtain $M^{\frac{N+2}{N-2}}\geq \frac{(N-2)^2}{4}M$;
therefore every positive fixed point satisfies
\begin{equation}\label{eq:positive-lower-bound}
 \|w\|_{L^\infty}\geq \left(\frac{N-2}{2}\right)^{\frac{N-2}{2}}.
\end{equation}
Thus each fixed point corresponds either to $w\equiv0$ (that is, $\zeta=-\phi$) or to
a positive solution of \eqref{eq:fixed-cylinder}, and the latter lies
at a fixed positive distance in $C^0$ from the former.
Denote by
\[
 \mathscr S^*=\bigl\{(T,w)\in(0,\infty)\times\mathscr X_0:
 w-\phi=\mathscr K(T,w-\phi),\ w\ne\phi\bigr\}
\]
the set of pairs $(T,w)$ corresponding, under
$\zeta=w-\phi$, to the nonzero fixed points of $\mathscr K(T,\cdot)$.

The positive nonstationary solutions form the set
\[
 \mathscr S^+=\bigl\{(T,w)\in\mathscr S^*:w>0,\,
 \partial_s w\not\equiv0\bigr\}.
\]
By uniqueness of $\phi$, we have
$\mathscr S^*=\mathscr S^+\cup\{(T,0):T>0\}$.
In this subsection, all closures and connected
components are taken in $(0,\infty)\times\mathscr X_0$.
Thus $\mathscr C$ is the component of $\overline{\mathscr S^+}$
containing $(T_0,\phi)$.
We first prove compactness for uniformly bounded solutions
on compact period intervals and improve the $C^0$ convergence to convergence in $\mathscr X$.

\begin{lemma}\label{lem:global-bootstrap}
Let $0<\alpha<\min\{1,\frac{4}{N-2}\}$ be as in
Theorem \ref{thm:local-bifurcation}, let $0<\tau_0<\tau_1<\infty$, and let
$M<\infty$.  There exist
$\widehat\alpha\in(\alpha,1)$ and $C=C(\tau_0,\tau_1,M)$ such that
\[
 \|w\|_{C^{2,\widehat\alpha}(\overline{\Sigma_1})}\leq C
\]
for every solution $(T,w)$ of \eqref{eq:positive-extension} with
$\tau_0\leq T\leq\tau_1$ and $\|w\|_{C^0(\overline{\Sigma_1})}\leq M$.
For any sequence $(T_n,w_n)$ of such solutions, if $T_n\to T$ and
$w_n\to w$ in $C^0(\overline{\Sigma_1})$, then, for every
$\beta\in(\alpha,\widehat\alpha)$,
\begin{equation}\label{eq:C0-to-Holder}
 w_n\longrightarrow w\quad\text{in }C^{2,\beta}(\overline{\Sigma_1}).
\end{equation}
If $w=\phi$, then, for all sufficiently large $n$,
\[
 \inf_{\overline{\Sigma_1}}
 \frac{w_n}{z_N}>0.
\]
\end{lemma}

\begin{proof}
Fix $r>N$ with $1-N/r>\alpha$ and choose $\widehat\alpha\in(\alpha,1-N/r)$.
The uniform $W^{2,r}$ estimate and Sobolev embedding give, with
$C=C(\tau_0,\tau_1,M)$,
\[
 \|w\|_{C^{1,1-N/r}(\overline{\Sigma_1})}
 \leq C\|w\|_{W^{2,r}(\Sigma_1)}
 \leq C\|(w_+)^{\frac{N+2}{N-2}}\|_{L^r(\Sigma_1)}\leq C.
\]
Since $t\mapsto(t_+)^{\frac{N+2}{N-2}}$ is Lipschitz on $[-M,M]$, the boundary Schauder
estimate (Gilbarg--Trudinger \cite[Theorem 6.6]{GT2001}) yields
\[
 \|w\|_{C^{2,\widehat\alpha}(\overline{\Sigma_1})}
 \leq C\bigl(\|(w_+)^{\frac{N+2}{N-2}}\|_{C^{0,\widehat\alpha}(\overline{\Sigma_1})}
              +\|w\|_{C^0(\overline{\Sigma_1})}\bigr)\leq C.
\]
By the compact emdedding of
$C^{2,\widehat\alpha}(\overline{\Sigma_1})\hookrightarrow
C^{2,\beta}(\overline{\Sigma_1})$ and the known $C^0$ convergence, we have 
\eqref{eq:C0-to-Holder} for the full sequence.
When $w=\phi$, the Hopf lemma and \eqref{eq:boundary-division} give
\[
 \frac{w_n}{z_N}\longrightarrow\frac{\phi}{z_N}
 \quad\text{in }C^{1,\beta}(\overline{\Sigma_1}),
 \qquad \inf_{\overline{\Sigma_1}}\frac{\phi}{z_N}>0.
\]
The last assertion follows by uniform convergence.
\end{proof}

Consequently, on the solution set and its closure, the
$C^0_{\mathrm{ax}}(\overline{\Sigma_1})$ and
$C^{2,\alpha}_{\mathrm{ax}}(\overline{\Sigma_1})$ topologies agree.  Indeed,
$C^{2,\alpha}_{\mathrm{ax}}$ convergence implies
$C^0_{\mathrm{ax}}$ convergence, while the converse
for a convergent sequence follows from
Lemma \ref{lem:global-bootstrap}.  We use $C^0_{\mathrm{ax}}$ below to construct the
continuum and compute the Leray--Schauder index, because the fixed point
map \eqref{eq:global-fixed-point-map} is compact there.  After the
construction, the preceding equivalence gives the
$C^{2,\alpha}_{\mathrm{ax}}$ topology used in
Theorem \ref{thm:local-bifurcation}.

We write $\operatorname{ind}_{\mathscr X_0}(I-\mathscr K(T,\cdot),0)$ for
the local Leray--Schauder index at an isolated fixed point $0$.
We use the following form of the global alternative of
Rabinowitz \cite[Lemma 1.2 and Theorem 1.3]{Rabinowitz1971};
see also Nussbaum \cite{Nussbaum1975}.

\begin{theorem}
\label{thm:global-alternative}
For $\mathscr K$ and $\mathscr S^*$ defined above, suppose that $T_0>0$
is an isolated singular parameter: $I-D_\zeta\mathscr K(T_0,0)$ is
not invertible, whereas $I-D_\zeta\mathscr K(T,0)$ is invertible for
$0<|T-T_0|<\delta$, for some $\delta\in(0,T_0)$.
Assume also that, for every sufficiently small $\varepsilon>0$,
\[
  \operatorname{ind}_{\mathscr X_0}
  (I-\mathscr K(T_0-\varepsilon,\cdot),0)
  \ne
 \operatorname{ind}_{\mathscr X_0}
  (I-\mathscr K(T_0+\varepsilon,\cdot),0).
\]
Then the component $\mathscr C$ of $\overline{\mathscr S^*}$ in
$(0,\infty)\times\mathscr X_0$ containing
$(T_0,\phi)$ either is not relatively compact in
$(0,\infty)\times\mathscr X_0$ or contains $(\widehat T,\phi)$ for some
singular parameter $\widehat T\ne T_0$.
\end{theorem}

We compute the index through the self-adjoint operator
$\mathscr L_{T,\phi}:=-T^{-2}\partial_{ss}+\mathscr A_\phi$
on the subspace of $L^2_{\mathrm{ax}}(\Sigma_1)$ that is even in $s$,
with zero Dirichlet
data on the lateral boundary.  Its Morse index
$\operatorname{ind}^{-}(\mathscr L_{T,\phi})$ is the number of negative
eigenvalues, counted with multiplicity.

\begin{lemma}\label{lem:index-bridge}
For the map $\mathscr K$ in \eqref{eq:global-fixed-point-map},
$I-D_\zeta\mathscr K(T,0)$ is noninvertible precisely at the parameters $T_m$,
$m\geq1$, and
\[
 \ker\mathscr L_{T_m,\phi}
 =\operatorname{span}\{\psi_1(z)\cos(2\pi ms)\}.
\]
If $T\notin\{T_m:m\geq1\}$, then $0$ is an isolated fixed point of
$\mathscr K(T,\cdot)$ and
\[
 \operatorname{ind}_{\mathscr X_0}(I-\mathscr K(T,\cdot),0)
 =
 \operatorname{ind}_{\mathscr X_0}(I-D_\zeta\mathscr K(T,0),0)
 =(-1)^{\operatorname{ind}^-(\mathscr L_{T,\phi})}.
\]
Consequently, for every $m\geq1$ and all sufficiently small
$\varepsilon>0$,
\[
 \operatorname{ind}_{\mathscr X_0}
 (I-\mathscr K(T_m-\varepsilon,\cdot),0)
 =-
 \operatorname{ind}_{\mathscr X_0}
 (I-\mathscr K(T_m+\varepsilon,\cdot),0).
\]
\end{lemma}

\begin{proof}
We compare $D_\zeta\mathscr K(T,0)$ on $\mathscr X_0$ and on the
subspace of $L^2_{\mathrm{ax}}(\Sigma_1)$ that is even in $s$.
By the standard elliptic regularity, for $\lambda\ne0$
and $\ell\geq1$, we have
\[
 \ker\bigl(D_\zeta\mathscr K(T,0)-\lambda I\bigr)^\ell
 \subset\mathscr X.
\]
Thus the nonzero generalized eigenspaces agree.  Set
\[
 A:=(-L_{\widehat g_T})^{-1/2},
 \qquad B:=(-L_{\widehat g_T})^{-1/2}\frac{N+2}{N-2}\phi^{\frac{4}{N-2}}.
\]
Then $D_\zeta\mathscr K(T,0)=AB$, whereas 
\[
 \mathscr S_T:=BA
 =(-L_{\widehat g_T})^{-1/2}\frac{N+2}{N-2}\phi^{\frac{4}{N-2}}(-L_{\widehat g_T})^{-1/2}
\]
is compact, self-adjoint, and nonnegative.
For every $\lambda\in\mathbb C$ and $\ell\geq1$,
\[
 (AB-\lambda I)^\ell A=A(BA-\lambda I)^\ell,\qquad
 (BA-\lambda I)^\ell B=B(AB-\lambda I)^\ell.
\]
Since $\phi>0$ in the interior, both $A$ and $B$ are injective.  For
$\lambda\ne0$, they therefore induce injections in opposite directions
between the generalized eigenspaces of $BA$ and $AB$.
Thus $D_\zeta\mathscr K(T,0)$ and $\mathscr S_T$ have the same nonzero
algebraic spectrum.  For $y=(-L_{\widehat g_T})^{1/2}h$,
\[
 \langle\mathscr L_{T,\phi}h,h\rangle_{L^2}
 =\|y\|_{L^2}^2-\langle\mathscr S_T y,y\rangle_{L^2}.
\]
By the min--max principle, the Morse index is the number of eigenvalues
of $\mathscr S_T$ larger than one.
By the Leray--Schauder index formula
(cf. Nirenberg \cite[Chapter 2]{Nirenberg2001}), we have
\[
 \operatorname{ind}_{\mathscr X_0}(I-D_\zeta\mathscr K(T,0),0)
 =(-1)^{\#\{\lambda\in\sigma(D_\zeta\mathscr K(T,0)):\,\lambda>1\}}
 =(-1)^{\operatorname{ind}^-(\mathscr L_{T,\phi})},
\]
where eigenvalues are counted with algebraic multiplicity.

Moreover, $(I-D_\zeta\mathscr K(T,0))h=0$ if and only if
$\mathscr L_{T,\phi}h=0$.  By separation of variables, for
$j\geq1$ and $m\geq0$, we have
\[
 \mathscr L_{T,\phi}\bigl(\psi_j(z)\cos(2\pi ms)\bigr)
 =\biggl(\lambda_j+\frac{4\pi^2m^2}{T^2}\biggr)\psi_j(z)\cos(2\pi ms).
\]
By Lemma \ref{lem:spectrum}, the positive parameters at
which $\mathscr L_{T,\phi}$ is not invertible are precisely $T_m$,
$m\geq1$, and the kernel at $T_m$ is the space displayed in the
statement.  The eigenvalue
$\lambda_1+4\pi^2m^2/T^2$ vanishes at $T_m$, and
\[
 \frac{\dd}{\dd T}
 \biggl(\lambda_1+\frac{4\pi^2m^2}{T^2}\biggr)\bigg|_{T=T_m}
 =-\frac{8\pi^2m^2}{T_m^3}<0.
\]
All other eigenvalues remain nonzero there.  Hence the Morse index
increases by one as $T$ passes through $T_m$, and the fixed point index
changes sign.
 Finally, if $T\notin\{T_m:m\geq1\}$, then
$I-D_\zeta\mathscr K(T,0)$ is invertible.  Using the expansion
$\mathscr K(T,\zeta)=D_\zeta\mathscr K(T,0)\zeta
+o(\|\zeta\|_{\mathscr X_0})$, we choose $\rho>0$ sufficiently small
so that the homotopy
\[
 \mathscr K_t(T,\zeta)
 :=(1-t)D_\zeta\mathscr K(T,0)\zeta+t\mathscr K(T,\zeta),
 \qquad 0\leq t\leq1,
\]
has no fixed point on $\partial B_\rho(0)$.  By homotopy invariance,
the first two indices in the statement agree. Combining this equality
with the change of Morse index at each $T_m$, we obtain the last assertion.
\end{proof}

We now show that strict monotonicity on each half period is
preserved along the continuum.  For a solution $w\in\mathscr X$,
$\xi=-\partial_s w$ satisfies
\begin{equation}\label{eq:q-linearized}
 \begin{cases}
 \bigl(-T^{-2}\partial_{ss}-\Delta_{\mathbb S^{N-1}}+\frac{(N-2)^2}{4}
      -\frac{N+2}{N-2}w^{\frac{4}{N-2}}\bigr)\xi=0&\text{in }\mathscr H,\\
 \xi=0&\text{on }\partial \mathscr H.
 \end{cases}
\end{equation}
Set
\[
 \mathcal M:=\bigl\{(T,w)\in\mathscr S^+:
 \text{either }-\partial_s w>0\text{ or }\partial_s w>0
 \text{ in }\mathscr H\bigr\}.
\]
The two signs are interchanged by a translation by half a period.

\begin{lemma}\label{lem:nodal-preservation}
The set $\mathcal M$ is relatively open and closed in $\mathscr S^+$.
\end{lemma}

\begin{proof}
We adapt the monotonicity argument of Dancer
\cite[Theorem 3]{Dancer2001}.
Let $\lambda_1(T,w)\leq\lambda_2(T,w)$ be the first two eigenvalues
of the operator in \eqref{eq:q-linearized} in $\mathscr H$, with zero
Dirichlet boundary condition.
For $(T,w)\in\mathcal M$, the derivative $\partial_s w$ has one strict
sign and satisfies \eqref{eq:q-linearized}. Hence
\[
 \lambda_1(T,w)=0<\lambda_2(T,w).
\]
By the min--max principle, $\lambda_2(T',w')>0$ for
$(T',w')\in\mathscr S^+$ sufficiently close to $(T,w)$.
The nonzero derivative $\partial_s w'$ also satisfies
\eqref{eq:q-linearized}. Thus
\[
 \lambda_1(T',w')=0,
 \qquad \partial_s w'>0\ \text{or}\ \partial_s w'<0
 \quad\text{in }\mathscr H.
\]
Thus $(T',w')\in\mathcal M$, proving relative openness.

For closedness, let $(T_n,w_n)\in\mathcal M$ converge to
$(T,w)\in\mathscr S^+$. After passing to a subsequence, fix
$\sigma\in\{-1,1\}$ such that $\sigma\partial_s w_n>0$ in $\mathscr H$.
By Lemma \ref{lem:global-bootstrap} and nonstationarity of $w$,
\[
 \sigma\partial_s w_n\longrightarrow\sigma\partial_s w
 \quad\text{in }C^1(\overline{\mathscr H}),
 \qquad \sigma\partial_s w\geq0,
 \qquad \partial_s w\not\equiv0.
\]
By \eqref{eq:q-linearized} and the Harnack inequality,
$\sigma\partial_s w>0$ in $\mathscr H$. Hence $(T,w)\in\mathcal M$.
\end{proof}

Now we prove Proposition \ref{prop:global-continuum}.
\begin{proof}[Proof of Proposition \ref{prop:global-continuum}]
\emph{Step 1. The global alternative.}
The local curve in Theorem \ref{thm:local-bifurcation} is a connected
subset of $\overline{\mathscr S^+}$ containing $(T_0,\phi)$, and hence
lies in $\mathscr C$.  By \eqref{eq:positive-lower-bound} and
Lemma \ref{lem:global-bootstrap}, every member of $\mathscr C$ is a
positive solution in $\mathscr X$.
The same lower bound separates the zero branch from
$\overline{\mathscr S^+}$, so $\mathscr C$ is also the component of
$\overline{\mathscr S^*}$ containing $(T_0,\phi)$.
By Theorem \ref{thm:global-alternative} and
Lemma \ref{lem:index-bridge}, we therefore obtain
\begin{equation*}
 \mathscr C\text{ is not relatively compact in }
 (0,\infty)\times\mathscr X_0,
 \qquad\text{or}\qquad
 (T_m,\phi)\in\mathscr C\quad\text{for some }m\geq2.
\end{equation*}

\emph{Step 2. The bumpy property.}
By Theorem \ref{thm:local-curve} and
Lemma \ref{lem:global-bootstrap}, we can choose a neighborhood $U$
of $(T_0,\phi)$ such that
\begin{equation*}
 \mathscr S^+\cap U\subset\mathcal M,
 \qquad (T_0,\phi)\in\overline{\mathcal M}.
\end{equation*}
Suppose that $(T_n,w_n)\in\mathcal M$ and
$(T_n,w_n)\to(\widehat T,\phi)$ in
$(0,\infty)\times\mathscr X_0$.  After passing to a subsequence and
translating by half a period if necessary, assume $-\partial_s w_n>0$ in
$\mathscr H$.
Since $u\mapsto(u_+)^{\frac{N+2}{N-2}}$ is $C^1$ on $\mathscr X_0$, the normalized
differences satisfy
\begin{equation}\label{eq:normalized-difference}
 h_n:=\frac{w_n-\phi}{\|w_n-\phi\|_{\mathscr X_0}},
 \qquad \mathscr L_{T_n,\phi}h_n\longrightarrow0
 \quad\text{in }C^0.
\end{equation}
Using $\|h_n\|_{C^0}=1$, we apply elliptic regularity to
\eqref{eq:normalized-difference}.  Passing to a subsequence, we have
\begin{equation*}
 h_n\longrightarrow h\quad\text{in }C^1,
 \qquad \|h\|_{\mathscr X_0}=1,
 \qquad \mathscr L_{\widehat T,\phi}h=0.
\end{equation*}
By Lemma \ref{lem:index-bridge},
\begin{equation*}
 \widehat T=T_m,\qquad
 h=c\psi_1(z)\cos(2\pi ms)
 \quad\text{for some }m\geq1\text{ and }c\ne0.
\end{equation*}
Since $-\partial_s h_n>0$ in $\mathscr H$, the limit $h$ satisfies
\[
 -\partial_s h=2\pi cm\psi_1(z)\sin(2\pi ms)\geq0
 \quad\text{in }\mathscr H.
\]
Since $\psi_1>0$ and $c\ne0$, we must have $m=1$ and $\widehat T=T_0$.
Every limit of $\mathcal M$ in $(0,\infty)\times\mathscr X_0$ is a
positive solution.  By Lemma \ref{lem:nodal-preservation}, we therefore have
$\overline{\mathcal M}=\mathcal M\cup\{(T_0,\phi)\}$.

The set $\overline{\mathcal M}$ is closed in
$\overline{\mathscr S^+}$ and is also relatively open.
At a point of $\mathcal M$, choose a neighborhood disjoint from the
stationary branch.  Every point of $\overline{\mathscr S^+}$ in this
neighborhood then lies in $\mathscr S^+$.
Relative openness there follows from
Lemma \ref{lem:nodal-preservation}.
At $(T_0,\phi)$, the neighborhood $U$ chosen above satisfies
$\overline{\mathscr S^+}\cap U\subset\overline{\mathcal M}$.
Since $\mathscr C$ is connected and contains
$(T_0,\phi)$, it follows that
\begin{equation*}
 \mathscr C\subset\overline{\mathcal M}
 =\mathcal M\cup\{(T_0,\phi)\}.
\end{equation*}
In particular, $(T_0,\phi)$ is the only stationary member of
$\mathscr C$, and by the alternative in Step 1,
$\mathscr C$ is not relatively compact.

\emph{Step 3. Least periods and escape alternatives.}
Translation by half a period defines a homeomorphism of
$\overline{\mathscr S^+}$ fixing $(T_0,\phi)$, and hence leaves
$\mathscr C$ invariant. Since every nonstationary element of
$\mathscr C$ lies in $\mathcal M$, a half-period translation if necessary gives
\[
-\partial_s w>0 \qquad \text{in }\mathscr H.
\]
Thus $w$ is bumpy. For each $z\in\mathbb S^{N-1}_+$, evenness and strict
monotonicity give a unique maximum of $w(\cdot,z)$ modulo $1$. Hence $w$
has least period $1$ in $s$.

Set $\xi=-\partial_s w>0$ in $\mathscr H$.  Testing
\eqref{eq:q-linearized} against $\xi$ and applying the sharp Poincar\'e
inequalities on $(0,1/2)$ and $\mathbb S^{N-1}_+$, we obtain
\begin{align*}
 \frac{N+2}{N-2}\int_{\mathscr H}w^{\frac{4}{N-2}}\xi^2\,\dd s\,\dd\sigma
 &=\int_{\mathscr H}\bigl(T^{-2}|\xi_s|^2
   +|\nabla_{\mathbb S^{N-1}}\xi|^2+\frac{(N-2)^2}{4}\xi^2\bigr)
   \,\dd s\,\dd\sigma \\
 &\geq\biggl(\frac{4\pi^2}{T^2}+N-1+\frac{(N-2)^2}{4}\biggr)
   \int_{\mathscr H}\xi^2\,\dd s\,\dd\sigma.
\end{align*}
Since $w\leq\|w\|_{L^\infty}$, the preceding inequality
implies
\begin{equation}\label{eq:period-amplitude}
 \frac{N+2}{N-2}\|w\|_{L^\infty}^{\frac{4}{N-2}}
 \geq \frac{4\pi^2}{T^2}+(N-1)+\frac{(N-2)^2}{4}.
\end{equation}
For every $0<\tau_0<\tau_1<\infty$ and $M<\infty$,
we deduce from Lemma \ref{lem:global-bootstrap} and the closedness of
$\mathscr C$ that
$\{(T,w)\in\mathscr C:\tau_0\leq T\leq\tau_1,\ \|w\|_{C^0}\leq M\}$ is compact.
Since $\mathscr C$ is noncompact, it contains a sequence
$\{(T_n,w_n)\}$ with no convergent subsequence.  After passage to a
subsequence, either $T_n\to0$, $T_n\to\infty$, or
$\|w_n\|_{L^\infty}\to\infty$.
If $T_n\to0$, then by \eqref{eq:period-amplitude}, we have
$\|w_n\|_{L^\infty}\to\infty$.  Hence this case is contained in the last
alternative, and
\eqref{eq:global-escape} follows.
\end{proof}
\endgroup

\section{Blow-up analysis}
\label{sec:finite-period-compactification}

Throughout this section, we work with $T$-periodic solutions $v$ of
\eqref{eq:EF-critical} in the original variable $t$.  On $\Sigma_T$,
this equation takes the form
\begin{equation}\label{eq:t-cylinder}
 \begin{cases}
  -L_{g_{\mathrm{cyl}}}v=v^{\frac{N+2}{N-2}}
  &\text{in }\Sigma_T,\\
  v=0&\text{on }(\mathbb R/T\mathbb Z)
                    \times\partial\mathbb S^{N-1}_+.
 \end{cases}
\end{equation}
Testing \eqref{eq:t-cylinder} against $v$, we obtain
\begin{equation}\label{eq:cylindrical-energy}
 E_T(v):=\int_{\Sigma_T}v^{\frac{2N}{N-2}}\dd V
 =\int_{\Sigma_T}\bigl(|\nabla v|^2+\frac{(N-2)^2}{4}v^2\bigr)\dd V.
\end{equation}

\subsection{\texorpdfstring{Periodic Green functions}
{Periodic Green functions}}

Set
\[
 c_N:=\frac{1}{(N-2)|\mathbb S^{N-1}|},
 \qquad -\Delta\bigl(c_N|x|^{2-N}\bigr)=\delta_0.
\]
For $X=(t,z)$ and
$Y=(\tau,\eta)$, set
\begin{equation}\label{eq:full-cylinder-kernel}
 \Gamma(X,Y)=c_N2^{-\frac{N-2}{2}}
 \bigl(\cosh(t-\tau)-z\cdot\eta\bigr)^{-\frac{N-2}{2}}.
\end{equation}
Thus $(-L_{g_{\mathrm{cyl}},X})\Gamma(X,Y)=\delta_Y$.  For
$\eta\in\mathbb S^{N-1}_+$, write
$\eta^*=(\eta_1,\ldots,\eta_{N-1},-\eta_N)$.  By conformal
covariance, the Dirichlet Green function on
$\mathbb R\times\mathbb S^{N-1}_+$ is
\begin{align}
 G_\infty((t,z),(\tau,\eta))
 =\Gamma((t,z),(\tau,\eta))-\Gamma((t,z),(\tau,\eta^*)).
 \label{eq:universal-cylinder-green}
\end{align}
Its periodization is
\begin{equation}\label{eq:periodized-cylinder-green}
 G_T((t,z),(\tau,\eta))
 =\sum_{m\in\mathbb Z}
 G_\infty((t+mT,z),(\tau,\eta)).
\end{equation}
In Euclidean coordinates, set
\begin{equation}\label{eq:flat-periodized-green-kernel}
 \widetilde G_T(x,e_N)
 =|x|^{-\frac{N-2}{2}}G_T
 \Bigl(\bigl(-\log|x|,{x}/{|x|}\bigr),(0,e_N)\Bigr).
\end{equation}
By the above formula and periodization, we have
\begin{align*}
 \widetilde G_T(x,e_N)
 =c_N\sum_{m\in\mathbb Z}e^{-\frac{N-2}{2}mT}\bigl(
 |e^{-mT}x-e_N|^{2-N}-|e^{-mT}x+e_N|^{2-N}\bigr).
\end{align*}
Pairing the terms with indices $m$ and $-m$ and expanding at $e_N$, we
obtain
\begin{equation*}
 c_N^{-1}\lim_{x\to e_N}
 \bigl(\widetilde G_T(x,e_N)-c_N|x-e_N|^{2-N}\bigr)
 =\mathscr M_N(T),
\end{equation*}
where $\mathscr M_N(T)$ is defined in \eqref{eq:intro-regular-part}.
For every $\tau>0$, summing the paired series and its first derivative
with respect to $T$ gives the following expansion. The remainder bounds,
including those for its first $T$ derivative, hold uniformly for $T\geq\tau$:
\begin{equation}\label{eq:periodic-green-regular-expansion}
 \begin{aligned}
  \widetilde G_T(x,e_N)
  =c_N|x-e_N|^{2-N}+c_N\mathscr M_N(T)+R_T(x),\quad
  R_T(x)=O(|x-e_N|),\quad \nabla R_T(x)=O(1).
 \end{aligned}
\end{equation}

\begin{proposition}\label{prop:unique-regular-zero}
The function $\mathscr M_N$ is
continuous and strictly decreasing on $(0,\infty)$, with
\[
 \lim_{T\downarrow0}\mathscr M_N(T)=+\infty,
 \qquad
 \lim_{T\to\infty}\mathscr M_N(T)=-2^{2-N}.
\]
Consequently, it has a unique zero $T^*>0$.
\end{proposition}

\begin{proof}
For $0<\rho<1$, set
$
 f(\rho)=\rho^{\frac{N-2}{2}}\bigl((1-\rho)^{2-N}-(1+\rho)^{2-N}\bigr).
$
Then
\begin{align*}
 f'(\rho)
 =\frac{N-2}{2}\rho^{\frac{N-4}{2}}\bigl((1-\rho)^{2-N}-(1+\rho)^{2-N}\bigr)
 +(N-2)\rho^{\frac{N-2}{2}}\bigl((1-\rho)^{1-N}+(1+\rho)^{1-N}\bigr)>0.
\end{align*}
Moreover, as $\rho\downarrow0$,
\[
 f(\rho)=2(N-2)\rho^{\frac{N}{2}}+O(\rho^{\frac{N+4}{2}}),\qquad
 f'(\rho)=N(N-2)\rho^{\frac{N-2}{2}}+O(\rho^{\frac{N+2}{2}}).
\]
For every $\tau_0>0$, the preceding estimates give
$C=C(N,\tau_0)>0$ such that, for $T\geq\tau_0$ and $\ell\geq1$,
\[
 |f(e^{-\ell T})|
 \leq Ce^{-N\ell\tau_0/2},\qquad
 \Bigl|\frac{\dd}{\dd T}f(e^{-\ell T})\Bigr|
 \leq C\ell e^{-N\ell\tau_0/2}.
\]
Since the bounds are summable in $\ell$,
termwise differentiation yields
\[
 \mathscr M_N'(T)
 =-2\sum_{\ell=1}^{\infty}
 \ell e^{-\ell T}f'(e^{-\ell T})<0.
\]
As $T\to\infty$,
\[
 0\leq2\sum_{\ell=1}^{\infty}f(e^{-\ell T})
 \leq C\sum_{\ell=1}^{\infty}e^{-N\ell T/2}\longrightarrow0.
\]
As $T\downarrow0$, every summand is positive and
\[
 2e^{-\frac{N-2}{2}T}\bigl((1-e^{-T})^{2-N}-(1+e^{-T})^{2-N}\bigr)
 =2T^{2-N}(1+o(1))\longrightarrow+\infty.
\]
This proves the endpoint limits; by continuity and strict monotonicity,
the zero is unique.
\end{proof}

We use the following uniform estimate away from the diagonal.
\begin{lemma}\label{lem:long-cylinder-green}
For every $\tau_0,r_0>0$, there exists $C=C(N,\tau_0,r_0)$ such that,
for $T\geq\tau_0$ and $X,Y\in\overline{\Sigma_T}$ with
$d_T(X,Y)\geq r_0$,
\begin{equation}\label{eq:long-cylinder-green-bounds}
 |\nabla_X^i\nabla_Y^jG_T(X,Y)|\leq C,
 \qquad i,j\in\{0,1\}.
\end{equation}
\end{lemma}

\begin{proof}
Choose lifts $X=(t,z)$ and $Y=(\tau,\eta)$ with $|t-\tau|\leq T/2$.
Since $d_T(X,Y)\geq r_0$, each point $(t+mT,z)$ is at distance at least
$r_0$ from both $Y$ and its reflection across the equator.
Differentiating \eqref{eq:full-cylinder-kernel} and
\eqref{eq:universal-cylinder-green}, we obtain
\[
 |\nabla_X^i\nabla_Y^jG_\infty((t+mT,z),Y)|
 \leq C_{N,r_0}e^{-\frac{N-2}{2}|t-\tau+mT|},
 \qquad i,j\in\{0,1\}.
\]
For $m\ne0$, we have
$|t-\tau+mT|\geq(|m|-1/2)T\geq(|m|-1/2)\tau_0$.
The series in \eqref{eq:periodized-cylinder-green} and its indicated
derivatives therefore converge uniformly for $d_T(X,Y)\geq r_0$.
Differentiating term by term and summing the bounds, we obtain
\[
 |\nabla_X^i\nabla_Y^jG_T(X,Y)|
 \leq C_{N,r_0}\biggl(1+2\sum_{m=1}^\infty
 e^{-\frac{N-2}{2}(m-1/2)\tau_0}\biggr)\leq C(N,\tau_0,r_0).\qedhere
\]
\end{proof}

We also use the following local Pohozaev
identity \cite{Pohozaev1965}.
For a smooth
function $h$ near $\partial B_r(Q)\subset\mathbb R^N$, let $\nu$ denote
the outward unit normal to $\partial B_r(Q)$ and set
\begin{equation}\label{eq:local-pohozaev-functional}
 \mathcal P_Q(r,h):= \int_{\partial B_r(Q)} \Bigl(\frac{N-2}{2}h\partial_\nu h-\frac r2|\nabla h|^2 +r(\partial_\nu h)^2\Bigr)\dd S.
\end{equation}
If $-\Delta h=\varepsilon h^{\frac{N+2}{N-2}}$ in $B_r(Q)$, then
\begin{equation}\label{eq:local-pohozaev-identity}
 0=\mathcal P_Q(r,h)+\frac{(N-2)\varepsilon r}{2N}
 \int_{\partial B_r(Q)}h^{\frac{2N}{N-2}}\dd S.
\end{equation}
Moreover, if $h$ is harmonic in a punctured neighborhood of $Q$ and
\begin{equation*}
 h(x)=A|x-Q|^{2-N}+B+O(|x-Q|),\qquad
 \nabla h(x)=-(N-2)A|x-Q|^{1-N}\frac{x-Q}{|x-Q|}+O(1),
\end{equation*}
then direct substitution into
\eqref{eq:local-pohozaev-functional} gives
\begin{equation}\label{eq:local-pohozaev-pole-limit}
 \lim_{r\downarrow0}\mathcal P_Q(r,h)
 =-\frac{(N-2)^2}{2}|\mathbb S^{N-1}|AB.
\end{equation}

To study dependence on the period in
Section \ref{sec:resonant-end}, we now identify cylinders of different
periods by a family of diffeomorphisms.  We first specify a lift to the
universal cover.  The
localization below freezes the metric,
the bubble, and the cutoff near $P$; a linear lift
$t\mapsto (T/T^*)t$ would not have this property.
The metric $g_T$ defined below is therefore distinct from
the metric $\widehat g_T$ used for the linear period identification in
Section \ref{sec:bifurcation-branches}.

Fix $r_0>0$ so small that $16r_0<T^*$.  Choose a nonnegative, even
function $\beta\in C^\infty(\mathbb R)$ of period $T^*$ such that
\begin{equation}\label{eq:fixed-domain-cutoff}
 \beta(t)=0\quad\text{if }
 \operatorname{dist}(t,T^*\mathbb Z)\leq8r_0,
 \qquad
 \int_0^{T^*}\beta(s)\dd s=1.
\end{equation}
Choose $\delta_T\in(0,T^*/2)$ sufficiently small and set
\begin{equation}\label{eq:fixed-domain-interval}
 I=[T^*-\delta_T,T^*+\delta_T],\qquad
 \delta_T\|\beta\|_{L^\infty}\leq\frac12.
\end{equation}
For $T\in I$, define
\begin{equation}\label{eq:fixed-domain-lift}
 \psi(t)=\int_0^t \beta(s)\dd s,
 \qquad
 J_T(t)=1+(T-T^*)\beta(t),
 \qquad
 \theta_T(t)=t+(T-T^*)\psi(t).
\end{equation}
Then $\psi$ and $\theta_T$ are odd,
\begin{equation}\label{eq:fixed-domain-equivariance}
 \psi(t+T^*)=\psi(t)+1,
 \qquad
 \theta_T(t+T^*)=\theta_T(t)+T,
 \qquad
 \theta_T'(t)=J_T(t)\geq\frac12.
\end{equation}
Thus $\theta_T$ descends to the diffeomorphism
\begin{equation}\label{eq:fixed-domain-map}
 \Phi_T:\Sigma_{T^*}\longrightarrow\Sigma_T,
 \qquad
 \Phi_T([t]_{T^*},z)=([\theta_T(t)]_T,z).
\end{equation}
Set $g_T:=\Phi_T^*g_{\mathrm{cyl}}$.  Then
$\Phi_T^*\dd V=J_T\dd t\dd\sigma$, and
\begin{equation}\label{eq:fixed-domain-metric}
 g_T=J_T^2\dd t^2+g_{\mathbb S^{N-1}},\qquad
 -L_{g_T}=-J_T^{-1}\partial_t\bigl(J_T^{-1}\partial_t\bigr)
 -\Delta_{\mathbb S^{N-1}}+\frac{(N-2)^2}{4}.
\end{equation}
The map $\Phi_T$ is the identity on $\mathcal B_{8r_0}(P)$.  Hence
$g_T=g_{\mathrm{cyl}}$ there, and $\partial_TL_{g_T}$ is supported
outside this ball.  We henceforth use $\Phi_T$ to identify every
$\Sigma_T$ with the fixed cylinder $\Sigma_{T^*}$.  Thus, if $f_T$ is a
family of functions on $\Sigma_T$, then $\partial_Tf_T$ denotes the
derivative of $\Phi_T^*f_T$ with respect to $T$.  Kernels are pulled back
in both variables, and operators are conjugated by $\Phi_T$.  In
particular, we continue to write $G_T$ for the pulled back Green kernel:
\begin{equation}\label{eq:fixed-domain-green-pullback}
 (\Phi_T\times\Phi_T)^*G_T(X,Y)=G_T(\Phi_T(X),\Phi_T(Y)),
 \qquad X,Y\in\Sigma_{T^*}.
\end{equation}
All derivatives with respect to $T$ below are understood in this fixed
domain sense.  For example, suppose that $f$ is independent of $T$ and
supported in $\mathcal B_{8r_0}(P)$, and that
$(-L_{g_{\mathrm{cyl}}})u_T=f$ on $\Sigma_T$.  Since $\Phi_T$ is the
identity on $\operatorname{supp}f$, the pulled back function, still denoted
by $u_T$, satisfies
\[
 (-L_{g_T})u_T=f,\qquad
 (-L_{g_T})\partial_Tu_T
 =-\bigl(\partial_T(-L_{g_T})\bigr)u_T.
\]
For $X,Y\in\mathcal B_{2r_0}(P)$, define the regular part by
$\mathcal H_T(X,Y):=G_T(X,Y)-\Gamma(X,Y)$, where $\Gamma$ is the
Green kernel on the full cylinder introduced above.

\begin{lemma}\label{lem:local-periodic-green}
The family $T\mapsto\mathcal H_T$ is $C^1$ with values in
$C^2(\mathcal B_{2r_0}(P)^2)$ and
\begin{equation}\label{eq:local-periodic-green-regularity}
 \sup_{T\in I}\sum_{j=0}^1
 \|\partial_T^j\mathcal H_T\|_{C^2(\mathcal B_{2r_0}(P)^2)}<\infty,
 \qquad
 \partial_T^j\mathcal H_T(P,P)=c_N\partial_T^j\mathscr M_N(T),
 \quad j=0,1.
\end{equation}
Moreover, there exist constants $c,C>0$, independent of $T\in I$,
such that
\[
\begin{gathered}
 G_T(X,Y)\leq C d_T(X,Y)^{2-N}
 \min\Bigl\{1,\frac{\delta(X)}{d_T(X,Y)}\Bigr\}
 \min\Bigl\{1,\frac{\delta(Y)}{d_T(X,Y)}\Bigr\},\\
 0\leq\int_{\Sigma_T}G_T(X,Y)\dd V_Y\leq C\delta(X).
\end{gathered}
\]
If $Y\in\mathcal B_{2r_0}(P)$ and $r(X)\geq4r_0$, then
\begin{equation}\label{eq:uniform-periodic-green-separated}
 c\delta(X)\leq G_T(X,Y)\leq C\delta(X),
 \qquad
 |\partial_TG_T(X,Y)|\leq C\delta(X).
\end{equation}
\end{lemma}

\begin{proof}
The map $\Phi_T$ is the identity near $P$. By the lifted
periodization formula, for $X,Y\in\mathcal B_{2r_0}(P)$ and
$Y^*=(\tau,\eta^*)$, we have
\[
 \mathcal H_T(X,Y)=-\Gamma(X,Y^*)
 +\sum_{m\ne0}\{\Gamma((t+mT,z),Y)
                 -\Gamma((t+mT,z),Y^*)\}.
\]
The first term is smooth, while every derivative of the summands of
order at most two in $X,Y$ and one in $T$ is bounded by
$C|m|e^{-c|m|\min I}$.  We may therefore differentiate the series
termwise in the claimed $C^1_TC^2_{X,Y}$ topology.
Comparing its diagonal value with
\eqref{eq:periodic-green-regular-expansion} proves the first assertion.

The metrics in \eqref{eq:fixed-domain-metric} form a compact uniformly
elliptic family.  By Lierl--Saloff-Coste \cite{LierlSaloffCoste2014},
the Dirichlet heat kernels satisfy the uniform bounds
\[
 K_T^D(s,X,Y)\leq Cs^{-N/2}e^{-d_T(X,Y)^2/(Cs)}
 \min\Bigl\{1,\frac{\delta(X)}{\sqrt s}\Bigr\}
 \min\Bigl\{1,\frac{\delta(Y)}{\sqrt s}\Bigr\}
 \quad\text{if }0<s\leq1,
\]
and $K_T^D(s,X,Y)\leq Ce^{-cs}\delta(X)\delta(Y)$ for $s\geq1$.
Integrating the heat kernel representation and using the
Green function estimate of Gr\"uter--Widman \cite{GruterWidman1982},
we obtain the pointwise bound in Lemma \ref{lem:local-periodic-green}.
By the maximum principle and
the uniform boundary $C^1$ estimate applied to
$(-L_{g_T})^{-1}1$, we obtain
the corresponding integral estimate.

Finally, in the separated region we obtain the upper and
lower bounds from boundary regularity, the Hopf lemma, positivity, and
compactness of $I$. By
\eqref{eq:fixed-domain-cutoff}--\eqref{eq:fixed-domain-metric}, the
coefficients of $\partial_TL_{g_T}$ are uniformly bounded and supported
away from $Y$.  Differentiating the pulled back Green equation and
using the zero lateral trace, we therefore obtain
$|\partial_TG_T(X,Y)|\leq C\delta(X)$.
\end{proof}

\subsection{Blow-up with a single bubble}

We say that a periodic solution is centered at $P$ if it
attains its maximum at $P$.  We begin with the uniqueness of the peak.

\begin{lemma}\label{lem:unique-peak}
Let $v$ be a positive bumpy solution of
\eqref{eq:t-cylinder} centered at $P$. Then
\begin{equation}\label{eq:angular-monotonicity}
 v_\theta(t,\theta)<0,
 \qquad 0<\theta<\frac{\pi}{2}.
\end{equation}
Consequently, $P$ is its unique maximum.
\end{lemma}

\begin{proof}
The method of moving planes of Gidas--Ni--Nirenberg
\cite{GidasNiNirenberg1979} gives angular monotonicity; see also the
half-space argument of the second author
\cite[Proposition 1.1]{Xiong2017}.  Together with the bumpy condition,
this shows that $P$ is the unique maximum.
\end{proof}

\begin{lemma}\label{lem:single-bubble-estimates}
Let $T_n>0$ and let $v_n$ be positive bumpy solutions of
\eqref{eq:t-cylinder} on $\Sigma_{T_n}$, centered at $P$.  Suppose that
\[
 M_n:=v_n(P)\longrightarrow\infty,
 \qquad
 \mu_n:=M_n^{-2/(N-2)},
 \qquad
 \frac{\mu_n}{\min\{1,T_n\}}\longrightarrow0.
\]
\begin{enumerate}
 \item[(A)] At the concentration scale,
 \begin{equation}\label{eq:core-bubble-limit}
  M_n^{-1}v_n\bigl(
   \exp_{P}^{g_{\mathrm{cyl}}}(\mu_n y)\bigr)
  \longrightarrow U_1(y)
  \quad\text{in }C^2_{\mathrm{loc}}(\mathbb R^N).
 \end{equation}

 \item[(B)] There exist $r_0\in(0,1/8)$ and $R_0,C>0$, independent of
 $n$, such that
 \begin{equation}\label{eq:uniform-local-bubble-bound}
  C^{-1}d_{T_n}(X,P)^{2-N}
  \leq M_nv_n(X)
  \leq C d_{T_n}(X,P)^{2-N}
  \quad\text{if }R_0\mu_n\leq d_{T_n}(X,P)
  \leq r_0\min\{1,T_n\}.
 \end{equation}

 \pagebreak[3]
 \item[(C)] If, in addition, $T_n\to T\in(0,\infty)$, then
 \begin{equation}\label{eq:finite-period-green-comparison}
  C^{-1}G_{T_n}(X,P)
  \leq M_nv_n(X)
  \leq C G_{T_n}(X,P)
  \quad\text{if }d_{T_n}(X,P)\geq R_0\mu_n.
 \end{equation}
 Moreover, after lifting to $\mathbb R\times\mathbb S^{N-1}_+$,
 \begin{equation}\label{eq:finite-period-green-profile}
  M_nv_n\longrightarrow\gamma_NG_T(\,\cdot\,,P)
  \quad\text{in }C^2_{\mathrm{loc}}
  \Bigl((\mathbb R\times\mathbb S^{N-1}_+)
  \setminus\{(mT,e_N):m\in\mathbb Z\}\Bigr).
 \end{equation}
 Here $\gamma_N:=\int_{\mathbb R^N}U_1^{\frac{N+2}{N-2}}\dd y$.
\end{enumerate}
\end{lemma}

For (A), we rescale at $\mu_n$. The local elliptic estimates and the
classification theorem of Caffarelli--Gidas--Spruck
\cite[Corollary 8.2(b)]{CGS1989} give the stated limit.  For (B), we
rescale the ball in the statement to unit size.  On this ball the
coefficients are uniformly controlled and the relative bubble scale tends
to zero.  The local blow-up scheme used here originates in
Schoen \cite{Schoen1991} and was developed further by
Li \cite{Li1995} and Chen--Lin \cite{ChenLin1997}.  The axial and angular
monotonicity reduce the point selected in
Li \cite[Proposition 4.1]{Li1995} to $P$.  By
Li \cite[Proposition 3.1]{Li1995}, the limiting equation with constant
coefficients satisfies the additional hypothesis of
Li \cite[Proposition 2.3]{Li1995}.  We may therefore apply
Li \cite[Propositions 2.2 and 2.3]{Li1995} to obtain (B), uniformly in $n$.
If $T_n\to T\in(0,\infty)$, the coefficients and domains
belong to a compact smooth family.  The boundary Harnack inequality, the
maximum principle, and Green representation therefore give (C), with
constants independent of $n$, as in Han \cite[Proposition 1]{Han1991}.

\begin{theorem}\label{thm:finite-period-compactification}
Let $T_n\to T\in(0,\infty)$, and let $v_n$ be positive bumpy solutions of
\eqref{eq:t-cylinder} on $\Sigma_{T_n}$, centered at $P$.  Suppose that
$M_n:=v_n(P)\to\infty$, and set $\mu_n:=M_n^{-2/(N-2)}$.  Define
$\Lambda_n:\Sigma_T\to\Sigma_{T_n}$ by
$\Lambda_n(t,z)=(T_nt/T,z)$.  Then
\[
 \Lambda_n^*\bigl(v_n^{\frac{2N}{N-2}}\dd V\bigr)
 \rightharpoonup
 \biggl(\int_{\mathbb R^N}U_1^{\frac{2N}{N-2}}\dd y\biggr)\delta_P,
\]
and necessarily
$
 T=T^*.
$
\end{theorem}

\begin{proof}
Since $T_n\to T\in(0,\infty)$, we have
$\mu_n/\min\{1,T_n\}\to0$.
By Lemma \ref{lem:single-bubble-estimates}(A), for every
fixed $R>R_0$,
\[
 \int_{\mathcal B_{R\mu_n}(P)}v_n^{\frac{2N}{N-2}}\dd V
 \longrightarrow\int_{B_R}U_1^{\frac{2N}{N-2}}\dd y.
\]
Fix $0<r<\frac12r_0\min\{1,T\}$. By
Lemma \ref{lem:single-bubble-estimates}(B) and
$M_n^{-\frac{2N}{N-2}}=\mu_n^N$, for all large $n$,
\[
 \int_{\mathcal B_r(P)\setminus\mathcal B_{R\mu_n}(P)}
 v_n^{\frac{2N}{N-2}}\dd V
 \leq C\mu_n^N\int_{R\mu_n}^r s^{-N-1}\dd s
 \leq CR^{-N}.
\]
By Lemmas \ref{lem:single-bubble-estimates}(C)
and \ref{lem:long-cylinder-green}, we have
$M_nv_n\leq C$ on $\Sigma_{T_n}\setminus\mathcal B_r(P)$.
Since the volumes of $\Sigma_{T_n}$ are bounded, we obtain
\[
 \int_{\Sigma_{T_n}\setminus\mathcal B_r(P)}v_n^{\frac{2N}{N-2}}\dd V
 \leq CM_n^{-\frac{2N}{N-2}}\longrightarrow0.
\]
Consequently, it holds that
\[
 \limsup_{n\to\infty}
 \int_{\Sigma_{T_n}\setminus\mathcal B_{R\mu_n}(P)}
 v_n^{\frac{2N}{N-2}}\dd V\leq CR^{-N}.
\]
Letting $R\to\infty$ proves the asserted energy convergence.
Set
$\nu_n:=\Lambda_n^*(v_n^{\frac{2N}{N-2}}\dd V)$ and let
$f\in C(\overline{\Sigma_T})$.  Applying the preceding
core, neck, and exterior estimates to
$f\circ\Lambda_n^{-1}$, and using $\Lambda_n(P)=P$ and
$\Lambda_n\to\mathrm{Id}$,
we obtain
\[
 \int_{\Sigma_T}f\,\dd\nu_n
 \longrightarrow
 f(P)\int_{\mathbb R^N}U_1^{\frac{2N}{N-2}}\dd y.
\]
This proves the asserted weak convergence of measures.

It remains to determine the limiting period.  Let $u_n$ be defined from the
lift of $v_n$ with period $T_n$ by \eqref{eq:EF-transform}.  By
Lemma \ref{lem:single-bubble-estimates}(C),
\[
 M_nu_n(x)
 =|x|^{-\frac{N-2}{2}}M_nv_n\Bigl(-\log|x|,\frac{x}{|x|}\Bigr)
 \longrightarrow\gamma_N\widetilde G_T(x,e_N)
\]
in $C^1_{\mathrm{loc}}(B_{\rho}(e_N)\setminus\{e_N\})$, where
$\rho>0$ is chosen so that $\overline{B_\rho(e_N)}\subset\mathbb R^N_+$
contains no other pole $e^{mT}e_N$, $m\ne0$.  Set
$V_n=M_nu_n$.  Then
\[
 -\Delta V_n=M_n^{-\frac{4}{N-2}}V_n^{\frac{N+2}{N-2}}.
\]
For every fixed sufficiently small $r>0$, the local
Pohozaev identity \eqref{eq:local-pohozaev-identity} gives
\[
 0=\mathcal P_{e_N}(r,V_n)
 +\frac{(N-2)M_n^{-\frac{4}{N-2}}r}{2N}
  \int_{\partial B_r(e_N)}V_n^{\frac{2N}{N-2}}\dd S.
\]
The second term tends to zero, while
$\mathcal P_{e_N}(r,V_n)$ converges to
$\mathcal P_{e_N}(r,h)$ for
$h=\gamma_N\widetilde G_T(\cdot,e_N)$.  Letting first $n\to\infty$ and
then $r\downarrow0$, and using \eqref{eq:periodic-green-regular-expansion}
and \eqref{eq:local-pohozaev-pole-limit}, we obtain
\[
 0=-\frac{(N-2)^2}{2}|\mathbb S^{N-1}|
   (\gamma_Nc_N)^2\mathscr M_N(T).
\]
Hence $\mathscr M_N(T)=0$, and
by Proposition \ref{prop:unique-regular-zero}, we have $T=T^*$.
\end{proof}

By Theorem \ref{thm:finite-period-compactification},
we have the following compactness result.

\begin{corollary}
\label{cor:nonresonant-compactness}
Fix $0<\alpha<1$ and a compact set
$K\subset(0,\infty)\setminus\{T^*\}$.  The positive bumpy
solutions of \eqref{eq:t-cylinder} with period $T\in K$ are,
after centering their maxima at $P$, precompact in
$C^{2,\alpha}_{\mathrm{ax}}$
under the period identifications.  Every limit is either $\phi$ or a
positive centered bumpy solution.  Consequently, for every fixed
$T\ne T^*$, the set of such centered solutions with
$\phi$ adjoined is compact in
$C^{2,\alpha}_{\mathrm{ax}}(\overline{\Sigma_T})$.
\end{corollary}

\begin{proof}
Let $(T_n,v_n)$ be a sequence of the stated solutions, with each $v_n$
centered at $P$.
Since $K$ is compact and $T^*\notin K$,
Theorem \ref{thm:finite-period-compactification} implies that
$\sup_n\|v_n\|_{L^\infty}<\infty$.
Set $w_n(s,z)=v_n(T_ns,z)$ and fix $\beta\in(\alpha,1)$.
By standard elliptic estimates on the fixed cylinder $\Sigma_1$,
we have $\|w_n\|_{C^{2,\beta}(\overline{\Sigma_1})}\leq C$.
Passing to a subsequence, we obtain
\[
 T_n\longrightarrow T\in K,\qquad
 w_n\longrightarrow w\quad\text{in }C^{2,\alpha}_{\mathrm{ax}}
 (\overline{\Sigma_1}).
\]
The limit $w$ solves $(P_T)$ and attains its maximum at $P$.
Evaluating \eqref{eq:t-cylinder} at $P$, we have
$w_n(P)^{\frac{4}{N-2}}\geq \frac{(N-2)^2}{4}>0$, and hence $w(P)>0$.
By the strong maximum principle, $w>0$ in $\Sigma_1$.
If $w$ is stationary, then $w=\phi$ by uniqueness for
\eqref{eq:stationary-q}; otherwise, by
Lemma \ref{lem:nodal-preservation}, $w$ is bumpy.
The compactness assertion for a fixed period follows upon adjoining $\phi$.
\end{proof}

\section{Uniform energy and period bounds}
\label{sec:uniform-bounds}

We establish uniform lower and upper bounds for the period and a uniform energy bound for bumpy solutions.

\begin{proposition}\label{prop:uniform-bounds}

There exist constants $0<T_-<\overline T<\infty$ and
$C>0$, depending only on $N$, such that every positive bumpy solution
$v$ of \eqref{eq:t-cylinder} satisfies
\begin{equation}\label{eq:uniform-bounds}
 T_-\leq T\leq \overline T,
 \qquad E_T(v)\leq C.
\end{equation}
The same bounds hold for every $(T,v)\in\mathscr C$.
Moreover, $\mathscr C$ contains a sequence $(T_n,v_n)$ with
$\|v_n\|_{L^\infty}\to\infty$.

\end{proposition}

\subsection{Lower period bound}
We first rule out $T\downarrow0$ by separating the bubble and period
scales, applying Lemma \ref{lem:single-bubble-estimates}(B), and then using a local Pohozaev identity.
\begin{proposition}\label{prop:lower-period-bound}
There exists $T_->0$, depending only on $N$, such that every
positive bumpy solution of \eqref{eq:t-cylinder} satisfies
\begin{equation}\label{eq:lower-period-bound}
 T\geq T_-.
\end{equation}
\end{proposition}

\begin{proof}
Suppose, to the contrary, that there are positive bumpy solutions
$v_n$ with periods $T_n\downarrow0$. 
Let $v_n$ be centered at $P$ and set
\begin{equation*}
 M_n=v_n(P),\qquad
 \mu_n=M_n^{-2/(N-2)},\qquad
 \vartheta_n=\frac{\mu_n}{T_n}.
\end{equation*}
Testing the linearized equation satisfied by $-\partial_t v_n$ on
$(0,T_n/2)\times\mathbb S^{N-1}_+$ as in the derivation of
\eqref{eq:period-amplitude}, we obtain
\begin{equation*}
 \frac{N+2}{N-2}M_n^{\frac{4}{N-2}}
 \geq\Bigl(\frac{2\pi}{T_n}\Bigr)^2+(N-1)+\frac{(N-2)^2}{4}
 \geq 4\pi^2 \vartheta_n^{2} M_n^{\frac{4}{N-2}}.
\end{equation*}
Thus $M_n\to\infty$ and $\vartheta_n\leq C$.

We first show that $\vartheta_n\to0$. On
$\mathcal O_n=(\mathbb R/\mathbb Z)\times
B_{\pi/(2T_n)}^{N-1}$, set
\begin{equation*}
 V_n(s,y)=T_n^{\frac{N-2}{2}}v_n(T_ns,\exp_{e_N}(T_ny)),
 \qquad
 g_n=\dd s^2+g_{ij}(T_ny)\dd y^i\dd y^j,
\end{equation*}
where $g_{ij}$ are the normal coordinate coefficients of
$g_{\mathbb S^{N-1}}$ at $e_N$.  On fixed compact sets,
\begin{equation*}
 g_n\longrightarrow \dd s^2+\dd y^2\quad\text{in }C^\infty,
 \qquad
 -\Delta_{g_n}V_n+\frac{(N-2)^2}{4}T_n^2V_n=V_n^{\frac{N+2}{N-2}}.
\end{equation*}
If $\vartheta_n\geq\vartheta_0>0$ along a subsequence,
then, by maximality at $P$, we have
\begin{equation*}
 0<V_n\leq V_n(0,0)=\vartheta_n^{-\frac{N-2}{2}}
 \leq\vartheta_0^{-\frac{N-2}{2}},
 \qquad V_n(0,0)\geq C^{-\frac{N-2}{2}}.
\end{equation*}
By interior Schauder estimates, we may pass to a subsequence such that
\begin{equation*}
 V_n\longrightarrow V>0
 \quad\text{in }C^2_{\mathrm{loc}}
 \bigl((\mathbb R/\mathbb Z)\times\mathbb R^{N-1}\bigr),
 \qquad -\Delta V=V^{\frac{N+2}{N-2}}.
\end{equation*}
The lift of $V$ to $\mathbb R^N$ is a positive entire solution.
By Caffarelli--Gidas--Spruck \cite[Corollary 8.2(b)]{CGS1989}, it equals
$U_{\lambda,\xi}$ for some $\lambda,\xi$, contrary to its periodicity.
Thus
\begin{equation}\label{eq:small-period-scale-separation}
 \vartheta_n=\frac{\mu_n}{T_n}\longrightarrow0.
\end{equation}
Now set
\begin{equation*}
 H_n=\vartheta_n^{-\frac{N-2}{2}}V_n=V_n(0,0)V_n,
 \qquad
 d(s,y)=
 \bigl(\operatorname{dist}_{\mathbb R/\mathbb Z}
       (s,0)^2+|y|^2\bigr)^{1/2}.
\end{equation*}
By \eqref{eq:small-period-scale-separation} and
Lemma \ref{lem:single-bubble-estimates}(B), there are constants
$c,C>0$ such that
\begin{equation}\label{eq:local-collapsing-green-bound}
 C^{-1}d(s,y)^{2-N}
 \leq H_n(s,y)
 \leq Cd(s,y)^{2-N}
\end{equation}
for $C\vartheta_n\leq d(s,y)\leq c$.  By the bumpy property and
the angular monotonicity in Lemma \ref{lem:unique-peak}, for
$0<\varrho<c$ and any
$e\in\mathbb S^{N-2}$, we have
\begin{equation}\label{eq:global-collapsing-green-bound}
 \sup_{\mathcal O_n\setminus B_\varrho(0)}H_n
 \leq\max\bigl\{H_n(\varrho/\sqrt2,0),
                  H_n(0,\varrho e/\sqrt2)\bigr\}
 \leq C_\varrho.
\end{equation}
For $\varrho\geq c$, \eqref{eq:global-collapsing-green-bound}
follows by the same monotonicity.
The rescaled equation is
\begin{equation*}
 -\Delta_{g_n}H_n+\frac{(N-2)^2}{4}T_n^2H_n
 =\vartheta_n^2H_n^{\frac{N+2}{N-2}}.
\end{equation*}
By \eqref{eq:small-period-scale-separation},
\eqref{eq:global-collapsing-green-bound},
and local Schauder estimates, we may pass to a subsequence such that
\begin{equation}\label{eq:collapsing-harmonic-limit}
 H_n\longrightarrow H
 \quad\text{in }C^2_{\mathrm{loc}}
 \bigl(((\mathbb R/\mathbb Z)\times\mathbb R^{N-1})
       \setminus\{0\}\bigr),
\end{equation}
where $H>0$ is harmonic.  In the local lift $X=(s,y)$, by
B\^ocher's theorem as stated by Axler--Bourdon--Ramey
\cite[Theorem 9.9]{AxlerBourdonRamey2001} and
\eqref{eq:local-collapsing-green-bound}, we have
\begin{equation*}
 H(X)=A|X|^{2-N}+h(X)\quad\text{near }0,
 \qquad A>0,
\end{equation*}
with $h$ harmonic near zero.

Suppose first that $N=3$. By the local expansion and harmonicity
away from the pole, we have, in distributions on
$(\mathbb R/\mathbb Z)\times\mathbb R^2$,
\begin{equation*}
 -\Delta H=\gamma\delta_0,
 \qquad \gamma=4\pi A>0.
\end{equation*}
We set
\begin{equation*}
 \overline H(R)=\frac1{2\pi}
 \int_0^1\int_{\mathbb S^1}H(s,R\zeta)
 \dd\zeta\dd s.
\end{equation*}
By the divergence theorem on $(\mathbb R/\mathbb Z)\times B_R^2$ and
the definition of $\overline H$, we obtain
\begin{align*}
 \gamma
 =-R\int_0^1\int_{\mathbb S^1}
       \partial_RH(s,R\zeta)\dd\zeta\dd s
 =-2\pi R\overline H'(R).
\end{align*}
Integrating in $R$ yields
$\overline H(R)=C_0-\frac{\gamma}{2\pi}\log R$,
which is negative for large $R$, contrary to $\overline H>0$.

It remains to consider $N\geq4$.
The positive periodic fundamental solution
\begin{equation*}
 \Gamma_{\mathrm{per}}(s,y)
 =c_N\sum_{m\in\mathbb Z}
 \bigl((s+m)^2+|y|^2\bigr)^{\frac{2-N}{2}}
\end{equation*}
converges in $C^2_{\mathrm{loc}}$ away from its pole and tends to zero as
$|y|\to\infty$, uniformly in $s$.  At the pole,
\begin{equation*}
 \Gamma_{\mathrm{per}}(X)
 =c_N|X|^{2-N}+m_N+O(|X|),
 \qquad m_N=2c_N\sum_{m=1}^{\infty}m^{2-N}>0.
\end{equation*}
The bounded harmonic function
$H-(A/c_N)\Gamma_{\mathrm{per}}$ has a removable singularity.  Its
periodic lift is constant, say $C_0$, by Liouville's theorem.
Since $H>0$ and $\Gamma_{\mathrm{per}}\to0$, we have $C_0\geq0$. Hence
\begin{equation}\label{eq:positive-collapsed-expansion}
 H(X)=A|X|^{2-N}+B+O(|X|),
 \qquad B=\frac{A}{c_N}m_N+C_0>0.
\end{equation}
To apply the local Pohozaev identity, define $u_n$ by
\eqref{eq:EF-transform} and fix $0<r_0<1/4$. On $B_{r_0}(0)$, set
\begin{equation}\label{eq:exact-collapsing-flat-equation}
 W_n(X)=\vartheta_n^{-\frac{N-2}{2}}T_n^{\frac{N-2}{2}}u_n(e_N+T_nX),
 \qquad -\Delta W_n=\vartheta_n^2W_n^{\frac{N+2}{N-2}}.
\end{equation}
The coordinate map is
\begin{equation*}
 \Psi_n(X)=\Biggl(
 -T_n^{-1}\log|e_N+T_nX|,
 T_n^{-1}\exp_{e_N}^{-1}
  \frac{e_N+T_nX}{|e_N+T_nX|}
 \Biggr).
\end{equation*}
On fixed balls,
\begin{equation*}
 \Psi_n\longrightarrow\Psi_0(X)=(-X_N,X')\quad\text{in }C^2,
 \qquad
 W_n(X)=|e_N+T_nX|^{-\frac{N-2}{2}}H_n(\Psi_n(X)).
\end{equation*}
For all large $n$, the first component of $\Psi_n$ lies in
$(-1/2,1/2)$ on $B_{r_0}$, so the image stays in a single fundamental
domain. Since $\Psi_0$ is orthogonal,
by \eqref{eq:collapsing-harmonic-limit} and
\eqref{eq:positive-collapsed-expansion}, we have
\begin{equation}\label{eq:flat-collapsed-expansion}
 W_n\longrightarrow W:=H\circ\Psi_0
 \quad\text{in }C^2_{\mathrm{loc}}(B_{r_0}\setminus\{0\}),
 \qquad
 W(X)=A|X|^{2-N}+B+O(|X|).
\end{equation}
Applying \eqref{eq:local-pohozaev-identity} with $Q=0$ and
$\varepsilon=\vartheta_n^2$, we obtain
\begin{equation*}
 0=\mathcal P_0(r,W_n)
 +\frac{(N-2)\vartheta_n^2r}{2N}
  \int_{\partial B_r}W_n^{\frac{2N}{N-2}}\dd S.
\end{equation*}
For each fixed $0<r<r_0$,
\eqref{eq:flat-collapsed-expansion} and $\vartheta_n\to0$ give
\begin{equation*}
 0\leq\frac{(N-2)\vartheta_n^2r}{2N}
 \int_{\partial B_r}W_n^{\frac{2N}{N-2}}\dd S
 \leq C_r\vartheta_n^2\longrightarrow0,
 \qquad
 \mathcal P_0(r,W_n)\longrightarrow
 \mathcal P_0(r,W).
\end{equation*}
Thus $\mathcal P_0(r,W)=0$. On the other hand, by
\eqref{eq:local-pohozaev-pole-limit} and
\eqref{eq:flat-collapsed-expansion},
\[
 \lim_{r\downarrow0}\mathcal P_0(r,W)
 =-\frac{(N-2)^2}{2}|\mathbb S^{N-1}|AB<0,
\]
a contradiction. This proves the lower period bound.
\end{proof}

\subsection{Uniform energy bound}

\begin{lemma}\label{lem:core-controls-total-energy}
If $r>0$ and $T\geq4r$, every positive bumpy solution of
\eqref{eq:t-cylinder} centered at $P$ obeys
\begin{equation}\label{eq:core-controls-total-energy}
 E_T(v)
 \leq
 2\int_{\{|t|<r\}\times\mathbb S^{N-1}_+}v^{\frac{2N}{N-2}}\dd V
 +
 \frac{N-2}{2r^2}
  \int_{\{|t|<r\}\times\mathbb S^{N-1}_+}v^2\dd V.
\end{equation}
\end{lemma}

\begin{proof}
Let $T\geq4r$, set
$\mathscr H_T=(0,T/2)\times\mathbb S^{N-1}_+$ and $\xi=-v_t$.
Differentiating \eqref{eq:t-cylinder} and using evenness,
periodicity, and the lateral Dirichlet condition, we obtain
\begin{equation*}
 \xi>0\quad\text{in }\mathscr H_T,
 \qquad
 \left(-L_{g_{\mathrm{cyl}}}
 -\frac{N+2}{N-2}v^{\frac{4}{N-2}}\right)\xi=0
 \quad\text{in }\mathscr H_T,
 \qquad \xi=0\quad\text{on }\partial \mathscr H_T.
\end{equation*}
Define
\begin{equation*}
 Q_v[\zeta]
 :=\int_{\mathscr H_T}\bigl(|\nabla\zeta|^2+\frac{(N-2)^2}{4}\zeta^2
                    -\frac{N+2}{N-2}v^{\frac{4}{N-2}}\zeta^2\bigr)\dd V,
 \qquad \zeta\in H_0^1(\mathscr H_T).
\end{equation*}
For $\zeta\in C_c^\infty(\mathscr H_T)$, put $f=\zeta/\xi$.
Using the equation for $\xi$, we compute
\begin{align*}
 Q_v[\zeta]=Q_v[\xi f]
 &=\int_{\mathscr H_T}\xi^2|\nabla f|^2\dd V+\int_{\mathscr H_T}f^2\xi
 \left(-L_{g_{\mathrm{cyl}}}
 -\frac{N+2}{N-2}v^{\frac{4}{N-2}}\right)\xi\dd V\\
 &=\int_{\mathscr H_T}\xi^2
   \Bigl|\nabla\Bigl(\frac{\zeta}{\xi}\Bigr)\Bigr|^2\dd V\geq0.
\end{align*}
Since $v^{\frac{4}{N-2}}\in L^\infty(\mathscr H_T)$, a standard density argument extends this inequality to
$H_0^1(\mathscr H_T)$.
Choose $\chi\in H_0^1(0,T/2)$ such that
\begin{equation*}
 0\leq\chi\leq1,\qquad
 \chi=1\quad\text{on }[r,T/2-r],
 \qquad |\chi'|\leq r^{-1},
 \qquad
 \operatorname{supp}\chi'\subset[0,r]\cup[T/2-r,T/2].
\end{equation*}
Here $v\chi\in H_0^1(\mathscr H_T)$. Testing \eqref{eq:t-cylinder}
against $v\chi^2$, we obtain
\begin{equation*}
 \int_{\mathscr H_T}\bigl(\chi^2|\nabla v|^2
 +2v\chi\nabla v\cdot\nabla\chi+\frac{(N-2)^2}{4}v^2\chi^2\bigr)\dd V
 =\int_{\mathscr H_T}v^{\frac{2N}{N-2}}\chi^2\dd V.
\end{equation*}
Expanding $|\nabla(v\chi)|^2$ and using $|\nabla\chi|=|\chi'|$,
we therefore obtain
\begin{equation*}
 0\leq Q_v[v\chi]
 =\int_{\mathscr H_T}v^2|\chi'|^2\dd V
  -\frac{4}{N-2}\int_{\mathscr H_T}v^{\frac{2N}{N-2}}\chi^2\dd V.
\end{equation*}
By the bumpy property, for every $\ell\geq1$, we have
\begin{equation*}
 \int_{\{T/2-r<t<T/2\}}v^\ell\dd V
 \leq\int_{\{0<t<r\}}v^\ell\dd V.
\end{equation*}
Combining the identity for $Q_v[v\chi]$ with the monotonicity
estimate, we obtain
\begin{equation*}
 \frac{4}{N-2}\int_{\mathscr H_T}v^{\frac{2N}{N-2}}\chi^2\dd V
 \leq\frac{2}{r^2}\int_{\{0<t<r\}}v^2\dd V.
\end{equation*}
Since $\chi=1$ on $[r,T/2-r]$, we have
\begin{align*}
 \int_{\mathscr H_T}v^{\frac{2N}{N-2}}\dd V
 &\leq\int_{\mathscr H_T}v^{\frac{2N}{N-2}}\chi^2\dd V
 +2\int_{\{0<t<r\}}v^{\frac{2N}{N-2}}\dd V\\
 &\leq2\int_{\{0<t<r\}}v^{\frac{2N}{N-2}}\dd V+
 \frac{N-2}{2r^2}
 \int_{\{0<t<r\}}v^2\dd V.
\end{align*}
Evenness in $t$ now gives
\eqref{eq:core-controls-total-energy}.
\end{proof}

It remains to bound the energy on each fixed central slab.
\begin{lemma}\label{lem:local-energy}
For every $R>0$ there exists $C=C(N,R)>0$ such that every
positive bumpy solution of \eqref{eq:t-cylinder} centered at $P$ satisfies
\begin{equation}\label{eq:local-energy}
 \int_{\{|t|<\min(R,T/2)\}\times\mathbb S^{N-1}_+}
 (v^{\frac{2N}{N-2}}+v^2)\dd V\leq C.
\end{equation}
Consequently, there exists $C>0$, depending only on $N$, such
that
\begin{equation}\label{eq:unconditional-energy-bound}
 E_T(v)\leq C.
\end{equation}
\end{lemma}

\begin{proof}
By Proposition \ref{prop:lower-period-bound}, we have $T\geq T_-$.
Suppose \eqref{eq:local-energy} fails for some fixed $R>0$.
Choose a sequence $(T_n,v_n)$ for which the integral in
\eqref{eq:local-energy} tends to infinity, and set
\[
 M_n=v_n(P),\qquad \mu_n=M_n^{-2/(N-2)}.
\]
Since the slabs have uniformly bounded volume, $M_n\to\infty$ and
\[
 \frac{\mu_n}{\min\{1,T_n\}}
 \leq\frac{\mu_n}{\min\{1,T_-\}}\longrightarrow0.
\]
The outer radius in Lemma \ref{lem:single-bubble-estimates}(B) is
bounded below by a positive constant independent of $n$.
By that lemma and maximality at $P$, we may therefore fix
$r_0\in(0,T_-/4)$, independent of $n$, such that
\begin{equation}\label{eq:local-peak-upper}
 v_n(X)\leq
 \begin{cases}
  M_n,&d_{T_n}(X,P)\leq R_0\mu_n,\\
  CM_n^{-1}d_{T_n}(X,P)^{2-N},
  &R_0\mu_n\leq d_{T_n}(X,P)\leq r_0.
 \end{cases}
\end{equation}
Thus
\[
 \int_{\mathcal B_{r_0}(P)}v_n^{\frac{2N}{N-2}}\dd V
 \leq C\Bigl(M_n^{\frac{2N}{N-2}}\mu_n^N
 +M_n^{-\frac{2N}{N-2}}\int_{R_0\mu_n}^{r_0}r^{-N-1}\dd r\Bigr)\leq C.
\]
The same radial calculation with exponent $2$ gives
\[
 \int_{\mathcal B_{r_0}(P)}v_n^2\dd V
 \leq C\Bigl(M_n^2\mu_n^N
 +M_n^{-2}\int_{R_0\mu_n}^{r_0}r^{3-N}\dd r\Bigr)\leq C.
\]

It remains to control the rest of the fixed slab.  Choose
$0<\rho<r_0$.  By the bumpy property and
Lemma \ref{lem:unique-peak}, $v_n$ is decreasing in both $|t|$ and the
angular distance from $e_N$.  Hence every point outside
$\mathcal B_\rho(P)$ satisfies
\[
 v_n(t,z)\leq
 \max\{v_n(\rho/2,e_N),\ v_n(0,z_\rho)\},
 \qquad d_{\mathbb S^{N-1}}(z_\rho,e_N)=\rho/2.
\]
The quantity on the right is bounded by \eqref{eq:local-peak-upper}; in fact it
is $O(M_n^{-1})$.  Since the slab has volume bounded by
$2R|\mathbb S^{N-1}_+|$, this proves the local estimate.
If $T\leq4$, the local estimate with $R=2$ gives
$E_T(v)\leq C$.  If $T\geq4$, combine
Lemma \ref{lem:core-controls-total-energy} with
the local estimate for $R=1$. 
This proves \eqref{eq:unconditional-energy-bound}.
\end{proof}

\subsection{Upper period bound}

We recall the following version of the
finite energy Liouville theorems of Gidas--Spruck
\cite{GidasSpruck1981Apriori} and Dancer \cite[Section 3]{Dancer1992}.

\begin{lemma}\label{lem:no-universal-profile}
There is no nonzero nonnegative
$W\in H_0^1(\mathbb R\times\mathbb S^{N-1}_+)$ satisfying
\begin{equation}\label{eq:universal-cylinder-profile}
 -L_{g_{\mathrm{cyl}}} W=W^{\frac{N+2}{N-2}}.
\end{equation}
\end{lemma}

We now use this Liouville theorem to exclude arbitrarily
large periods under a fixed energy bound.
\begin{proposition}\label{prop:upper-period-bound}
For every $\Lambda>0$ there exists $T_\Lambda>0$ such that no
positive bumpy solution of \eqref{eq:t-cylinder} satisfies
\begin{equation}\label{eq:large-period-energy-hypothesis}
 T\geq T_\Lambda,
 \qquad E_T(v)\leq\Lambda.
\end{equation}
\end{proposition}

\begin{proof}
Suppose that $T_n\to\infty$ and $E_{T_n}(v_n)\leq\Lambda$.  Center the
maximum at $P$ and write the profiles on
$(-T_n/2,T_n/2)\times\mathbb S^{N-1}_+$.
Set $M_n=v_n(P)$.  If $M_n$ had a bounded subsequence, then,
by interior and boundary Schauder estimates, we could pass to a further
subsequence such that
\[
 v_n\longrightarrow V
 \quad\text{in }C^2_{\mathrm{loc}}
 (\mathbb R\times\overline{\mathbb S^{N-1}_+}).
\]
At the maximum, $M_n^{\frac{4}{N-2}}\geq \frac{(N-2)^2}{4}$, and therefore
$V(P)\geq \left(\frac{N-2}{2}\right)^{\frac{N-2}{2}}>0$.  Moreover, for every fixed $R>0$ and all
large $n$,
\[
 \int_{(-R,R)\times\mathbb S^{N-1}_+}
 \bigl(|\nabla v_n|^2+\frac{(N-2)^2}{4}v_n^2\bigr)\dd V
 \leq E_{T_n}(v_n)\leq\Lambda.
\]
Passing to the limit on the fixed slab and then letting $R\to\infty$,
we obtain $V\in H^1(\mathbb R\times\mathbb S^{N-1}_+)$. The boundary
convergence preserves the zero lateral trace. Using a cutoff function in
the variable $t$, we conclude that
$V\in H_0^1(\mathbb R\times\mathbb S^{N-1}_+)$.  This contradicts
Lemma \ref{lem:no-universal-profile}.  Hence
\[
 M_n\longrightarrow\infty,
 \qquad \mu_n=M_n^{-2/(N-2)}.
\]
Since $T_n\to\infty$, the required scale separation is automatic.
By Lemma \ref{lem:single-bubble-estimates}(A) and (B), after decreasing
$r_0>0$ if necessary, we obtain
\begin{equation}\label{eq:large-period-core-profile}
 M_n^{-1}v_n\bigl(\exp_P^{g_{\mathrm{cyl}}}(\mu_n y)\bigr)
 \longrightarrow U_1(y)
 \quad\text{in }C^2_{\mathrm{loc}}(\mathbb R^N),
\end{equation}
and
\begin{equation}\label{eq:large-period-core-bound}
 v_n(X)\leq C M_n^{-1} d_{T_n}(X,P)^{2-N}
 \quad\text{if }R_0\mu_n\leq d_{T_n}(X,P)\leq2r_0.
\end{equation}
For every fixed $0<r<r_0$, the local estimate also controls the whole
complement of $\mathcal B_r(P)$.  Indeed, if $d_{T_n}(X,P)\geq r$, then
either the axial distance from $0$ or the angular distance from $e_N$ is
at least $r/2$.  By the axial and angular monotonicity and
\eqref{eq:large-period-core-bound} at distance $r/2$, we therefore obtain
\begin{equation}\label{eq:large-period-outer-small}
 \sup_{\Sigma_{T_n}\setminus\mathcal B_r(P)}v_n
 \leq C_rM_n^{-1}.
\end{equation}

We now carry out, uniformly as $T_n\to\infty$, the Green representation
argument in Han \cite[Proposition 1]{Han1991}.  Fix a compact set and a
radius satisfying
\[
 K\Subset
 (\mathbb R\times\overline{\mathbb S^{N-1}_+})\setminus\{P\},
 \qquad
 0<r<\min\Bigl\{r_0,\frac14\operatorname{dist}(K,P)\Bigr\}.
\]
For all sufficiently large $n$, identify $K$ with its image in the
period cell
$(-T_n/2,T_n/2)\times\mathbb S^{N-1}_+$.  Put
$F_n=M_nv_n^{\frac{N+2}{N-2}}$.  For fixed $R>R_0$ and
$R\mu_n<r$, split
\begin{align*}
 M_nv_n(X)&=I_n^{\mathrm{core}}(X)
           +I_n^{\mathrm{neck}}(X)+I_n^{\mathrm{out}}(X),\\
 I_n^{\mathrm{core}}(X)
 &:=\int_{\mathcal B_{R\mu_n}(P)}G_{T_n}(X,Y)F_n(Y)\dd V_Y,\\
 I_n^{\mathrm{neck}}(X)
 &:=\int_{\mathcal B_r(P)\setminus\mathcal B_{R\mu_n}(P)}
       G_{T_n}(X,Y)F_n(Y)\dd V_Y,\\
 I_n^{\mathrm{out}}(X)
 &:=\int_{\Sigma_{T_n}\setminus\mathcal B_r(P)}
       G_{T_n}(X,Y)F_n(Y)\dd V_Y.
\end{align*}
By Lemma \ref{lem:long-cylinder-green} and the mean value
theorem in the pole variable, we have, uniformly for
$Y\in\mathcal B_r(P)$,
\[
 \|G_{T_n}(\cdot,Y)\|_{C^1(K)}\leq C_K,
 \qquad
 \|G_{T_n}(\cdot,Y)-G_{T_n}(\cdot,P)\|_{C^1(K)}
 \leq C_Kd_{T_n}(Y,P).
\]
In the core, changing variables by
$Y=\exp_P^{g_{\mathrm{cyl}}}(\mu_ny)$ and using
$M_n^{\frac{2N}{N-2}}\mu_n^N=1$ and \eqref{eq:large-period-core-profile}, we obtain
\[
 I_n^{\mathrm{core}}
 =G_{T_n}(\cdot,P)\int_{B_R}U_1^{\frac{N+2}{N-2}}\dd y+o_n(1)
 \quad\text{in }C^1(K).
\]
In the neck, \eqref{eq:large-period-core-bound} and
$M_n^{-\frac{4}{N-2}}=\mu_n^2$ give
\[
 F_n(Y)\leq C\mu_n^2d_{T_n}(Y,P)^{-N-2},
\]
and hence
\[
 \|I_n^{\mathrm{neck}}\|_{C^1(K)}
 \leq C_K\mu_n^2\int_{R\mu_n}^{r}s^{-3}\dd s
 \leq C_KR^{-2}.
\]
Finally, \eqref{eq:large-period-outer-small} gives
$F_n\leq C_r\mu_n^2$ outside $\mathcal B_r(P)$.  If
\[
 Q_n(X):=\int_{\Sigma_{T_n}}G_{T_n}(X,Y)\dd V_Y,
\]
then $-L_{g_{\mathrm{cyl}}}Q_n=1$ with zero lateral trace. Comparing
with the supersolution $\frac{4}{(N-2)^2}$, we obtain, uniformly in $T_n$,
$0\leq Q_n\leq \frac{4}{(N-2)^2}$, and therefore
$\|I_n^{\mathrm{out}}\|_{C^0}\leq C_r\mu_n^2$.  Since $K$ is
separated from $\mathcal B_r(P)$, the interior and boundary elliptic
$W^{2,q}$ estimates, with $q>N$, improve this to
\[
 \|I_n^{\mathrm{out}}\|_{C^1(K)}\leq C_{K,r}\mu_n^2=o_n(1).
\]
The bound for $Q_n$ holds on the entire fundamental cylinder and hence
also controls its two expanding ends.  Moreover, we have
\[
 0\leq\gamma_N-\int_{B_R}U_1^{\frac{N+2}{N-2}}\dd y\leq CR^{-2},
 \qquad
 \sup_n\|G_{T_n}(\cdot,P)\|_{C^1(K)}\leq C_K.
\]
Combining these bounds with the estimates for
$I_n^{\mathrm{core}}$, $I_n^{\mathrm{neck}}$, and
$I_n^{\mathrm{out}}$, we obtain
\[
 \limsup_{n\to\infty}
 \|M_nv_n-\gamma_NG_{T_n}(\cdot,P)\|_{C^1(K)}
 \leq C_KR^{-2}.
\]
Letting $R\to\infty$ and then varying $K$, we obtain
\begin{equation}\label{eq:large-period-periodic-green-profile}
 M_nv_n-\gamma_NG_{T_n}(\,\cdot\,,P)
 \longrightarrow0
 \quad\text{in }C^1_{\mathrm{loc}}
 \bigl((\mathbb R\times\overline{\mathbb S^{N-1}_+})\setminus\{P\}\bigr).
\end{equation}

Returning to the Euclidean variables, let $u_n$ be defined by
\eqref{eq:EF-transform} and set $h_n=M_nu_n$.  By conformal covariance and
\eqref{eq:large-period-periodic-green-profile},
\[
 h_n-\gamma_N\widetilde G_{T_n}(\cdot,e_N)
 \longrightarrow0
 \quad\text{in }C^1_{\mathrm{loc}}
 (B_{r_0}(e_N)\setminus\{e_N\}).
\]
Fix $0<r<r_0$.  Since $-\Delta h_n=M_n^{-\frac{4}{N-2}}h_n^{\frac{N+2}{N-2}}$, the local
Pohozaev identity on $B_r(e_N)$ reads
\[
 0=\mathcal P_{e_N}(r,h_n)
 +\frac{M_n^{-\frac{4}{N-2}}r}{\frac{2N}{N-2}}
  \int_{\partial B_r(e_N)}h_n^{\frac{2N}{N-2}}\dd S.
\]
For fixed $r$, the convergence above is uniform in $C^1$ on
$\partial B_r(e_N)$, and the second term tends to zero since
$M_n^{-\frac{4}{N-2}}=\mu_n^2\to0$.  By the expansion
\eqref{eq:periodic-green-regular-expansion} and the definition of
$\mathcal P_{e_N}$, we obtain, uniformly for $T_n$ bounded away from zero,
\[
 \mathcal P_{e_N}
 \bigl(r,\gamma_N\widetilde G_{T_n}(\cdot,e_N)\bigr)
 =-\frac{(N-2)^2}{2}|\mathbb S^{N-1}|
   (\gamma_Nc_N)^2\mathscr M_N(T_n)+O(r).
\]
Consequently, for every fixed sufficiently small $r$,
\[
 0=-\frac{(N-2)^2}{2}|\mathbb S^{N-1}|
   (\gamma_Nc_N)^2\mathscr M_N(T_n)+O(r)+o_n(1).
\]
Letting first $n\to\infty$ and then $r\downarrow0$ contradicts
$\mathscr M_N(T_n)\to-2^{2-N}$ from
Proposition \ref{prop:unique-regular-zero}.  Therefore the periods are
uniformly bounded above under \eqref{eq:large-period-energy-hypothesis}.
\end{proof}

We now combine the preceding period and energy estimates to prove
Proposition \ref{prop:uniform-bounds}.
\begin{proof}[Proof of Proposition \ref{prop:uniform-bounds}]
By Proposition \ref{prop:lower-period-bound} and
Lemma \ref{lem:local-energy}, there exists $C=C(N)>0$ such that
$T\geq T_-$ and $E_T(v)\leq C$.
By Proposition \ref{prop:upper-period-bound}, we also
obtain an upper bound $\overline T$ depending only on $N$.  Enlarging
$\overline T$ if necessary, we have $T_-<\overline T$, and
\eqref{eq:uniform-bounds} follows.
For a nonstationary $(T,v)\in\mathscr C$, the solution $v$ is
bumpy and can be centered at $P$ by a translation, without changing
its period or energy.
Thus \eqref{eq:uniform-bounds} applies to the nonstationary part of
$\mathscr C$.
The stationary point $(T_0,\phi)$ has fixed positive period and finite
energy, both depending only on $N$.  After decreasing
$T_-$ and increasing $\overline T$ and $C$ if necessary,
\eqref{eq:uniform-bounds} holds on all of $\mathscr C$.
Finally, the upper period bound
excludes $T_n\to\infty$ in \eqref{eq:global-escape}.  By
Proposition \ref{prop:global-continuum}, the remaining escape sequence
satisfies $\|v_n\|_{L^\infty}\to\infty$.
\end{proof}

We conclude the section by proving parts (ii) and (iii) of
Theorem \ref{thm:local-bifurcation}.
\begin{proof}[Proof of Theorem \ref{thm:local-bifurcation}(ii)--(iii)]
Under the identification $w(s,z)=v(Ts,z)$, take $\mathscr C$
to be the set constructed in Proposition \ref{prop:global-continuum}.
The continuum is first constructed in the $C^0$ topology.
Lemma \ref{lem:global-bootstrap} shows that the $C^0$ and
$C^{2,\alpha}$ topologies agree on the solution set and its closure.
Thus $\mathscr C$ is closed and connected in the topology of
Theorem \ref{thm:local-bifurcation}.  Proposition \ref{prop:global-continuum}
gives its noncompactness, and Proposition \ref{prop:uniform-bounds}
supplies the uniform period and energy bounds.

Let $(T_n,v_n)\in\mathscr C$ satisfy
$\|v_n\|_{L^\infty}\to\infty$.  By
Proposition \ref{prop:uniform-bounds}, the periods lie in a compact
subset of $(0,\infty)$.  After translating in $t$, we may assume
that each $v_n$ is centered at $P$.  Every subsequence has a further
subsequence for which $T_n\to T\in(0,\infty)$.  Theorem \ref{thm:finite-period-compactification}
then gives $T=T^*$.  Since every convergent subsequence has the same
limit, $T_n\to T^*$ for the full sequence.  This proves part (ii).
By Proposition \ref{prop:global-continuum},  the local curve is
contained in $\mathscr C$ and every member of
$\mathscr C\setminus\{(T_0,\phi)\}$ is bumpy and has least period $T$.
This proves part (iii).
\end{proof}

\section{Construction of the blow-up branch}\label{sec:resonant-end}

We use Lyapunov--Schmidt reduction to construct solutions concentrating at $P$ as the period tends to $T^*$.

\subsection{Approximate solution and preliminary estimates}

Recall that
$\lambda_*=\sqrt{N(N-2)}$ and $\alpha_N=\lambda_*^{\frac{N-2}{2}}$.
For $0<\mu<\mu_0$, with $\mu_0>0$ to be chosen later,
consider the following standard bubble
\begin{equation}\label{eq:res-flat-bubble}
 \begin{aligned}
 U_{\lambda_*/\mu,\,\sqrt{1-\mu^2}\,e_N}(x)
 &=\mu^{-\frac{N-2}{2}}U_{\lambda_*} \biggl(
   \frac{x-\sqrt{1-\mu^2}\,e_N}{\mu}\biggr)\\
 &=\alpha_N\biggl(\frac{\mu}
   {\mu^2+|x-\sqrt{1-\mu^2}\,e_N|^2}\biggr)^{\frac{N-2}{2}},
 \qquad x\in\mathbb R^N.
 \end{aligned}
\end{equation}
Under the Emden--Fowler transformation
\eqref{eq:EF-transform},
the corresponding bubble on $\mathbb R\times\mathbb S^{N-1}$ is
\begin{equation}\label{eq:res-cylinder-bubble}
 \mathcal U_\mu(t,z)
 :=e^{-\frac{N-2}{2}t}U_{\lambda_*/\mu,\,\sqrt{1-\mu^2}\,e_N}(e^{-t}z)
 =\alpha_N\biggl(\frac{\mu}
 {2\bigl(\cosh t-\sqrt{1-\mu^2}\,z_N\bigr)}\biggr)^{\frac{N-2}{2}}.
\end{equation}
The choice of the center $\sqrt{1-\mu^2}\,e_N$ makes
$\mathcal U_\mu$ even in $t$. By conformal covariance,
$-L_{g_{\mathrm{cyl}}}\mathcal U_\mu=\mathcal U_\mu^{\frac{N+2}{N-2}}$. 
Using the Green kernel $\Gamma$ on the full cylinder,
defined in \eqref{eq:full-cylinder-kernel}, we obtain
\[
 (\mu\partial_\mu)^{i}\mathcal U_\mu(X)
 =\int_{\mathbb R\times\mathbb S^{N-1}}
 \Gamma(X,Y)(\mu\partial_\mu)^{i}\mathcal U_\mu^{\frac{N+2}{N-2}}(Y)\dd V_Y,
 \qquad i=0,1,2.
\]
Fix $r_0>0$ and the function $\beta$ as in
Lemma \ref{lem:local-periodic-green}, decreasing $r_0$ if necessary so
that $r_0<1/8$ and $\mathcal B_{8r_0}(P)$ lies in the fixed interior
conformal chart.  Choose an even cutoff
$\chi\in C^\infty_{c,\mathrm{ax}}(\mathcal B_{2r_0}(P))$ satisfying
\[
 0\leq\chi\leq1,\qquad \chi\equiv1\quad\text{on }\mathcal B_{r_0}(P).
\]
We retain the interval $I$ from \eqref{eq:fixed-domain-interval};
our choices ensure $8r_0<T^*-\delta_T$.  We use the fixed domain convention introduced
after \eqref{eq:fixed-domain-green-pullback}; thus every $\partial_T$
below is taken after pullback by $\Phi_T$.  Since $\Phi_T$ is the identity
near $P$, we regard $\chi$ as the same function on every $\Sigma_T$.
Henceforth $T\in I$, and we fix
$0<\alpha<\min\{1,\frac{4}{N-2}\}$. We decrease $\mu_0$ so that $\mu_0<r_0/8$.

Since $\mathcal U_\mu$ is not periodic with period $T$, every integral over $\Sigma_T$
involving the uncut $\mathcal U_\mu$ is taken over the central fundamental
domain $(-T/2,T/2)\times\mathbb S^{N-1}_+$.
For the local estimates below, put
$r_\mu(X)=\min\{1,\mu+r(X)\}$ and write
$r=r(X)$, $\delta=\delta(X)$, and $r_\mu=r_\mu(X)$ whenever no confusion
can arise. Fix a sufficiently small $c\in(0,1/16)$, independent
of $\mu$ and $T$, so that the balls of radius $4cr_\mu(X)$ lie in
uniform interior or boundary coordinate charts.
We use the following elementary bubble estimates.
\begin{lemma}\label{lem:res-bubble-estimates}
The following estimates hold uniformly for $T\in I$ and
$0<\mu<\mu_0$, with $i=0,1,2$:
\begin{enumerate}
\renewcommand{\labelenumi}{\textup{(\roman{enumi})}}
\setlength{\itemsep}{0pt}
 \item $|(\mu\partial_\mu)^{i}\mathcal U_\mu^{\frac{N+2}{N-2}}|^{\frac{N-2}{N+2}}
 +|(\mu\partial_\mu)^{i}\mathcal U_\mu|\leq C\mathcal U_\mu$;
 \item $|(\mu\partial_\mu)^{i}\mathcal U_\mu^{\frac{N+2}{N-2}}(X)|\leq C\mu^{\frac{N+2}{2}}$
 for $X\notin\mathcal B_{r_0/2}(P)$;
 \item $r_\mu(X)^\alpha[(\mu\partial_\mu)^{i}\mathcal U_\mu^{\frac{N+2}{N-2}}]_
 {C^\alpha(\mathcal B_{cr_\mu(X)}(X)\cap \Sigma_T)}
 \leq C\mathcal U_\mu^{\frac{N+2}{N-2}}(X)$ for $r(X)<r_0$;
 \item $\displaystyle
 \int_{\Sigma_T}r(X)^m|(\mu\partial_\mu)^{i}\mathcal U_\mu^{\frac{N+2}{N-2}}|\dd V
 \leq C\mu^{\frac{N-2}{2}+m}$ for $m=0,1$;
 \item $\displaystyle
 \int_{\Sigma_T}\chi(\mu\partial_\mu)^{i}\mathcal U_\mu^{\frac{N+2}{N-2}}\dd V
 =\frac{\alpha_N}{c_N}\left(\frac{N-2}{2}\right)^i\mu^{\frac{N-2}{2}}+O(\mu^{\frac{N}{2}})$.
\end{enumerate}
\end{lemma}

\begin{proof}
By the explicit formula for $\mathcal U_\mu$, we obtain (i)--(iii).
For (iv), we use (ii) to control the second term and obtain
\[
\begin{aligned}
 \int_{\Sigma_T}r(X)^m|(\mu\partial_\mu)^{i}\mathcal U_\mu^{\frac{N+2}{N-2}}|\dd V
 &=\int_{\mathcal B_{r_0}(P)}r^m
 |(\mu\partial_\mu)^{i}\mathcal U_\mu^{\frac{N+2}{N-2}}|\dd V
 +\int_{\Sigma_T\setminus\mathcal B_{r_0}(P)}r^m
 |(\mu\partial_\mu)^{i}\mathcal U_\mu^{\frac{N+2}{N-2}}|\dd V\\
 &\leq C\mu^{\frac{N+2}{2}}\int_0^{r_0}
 \frac{s^{N-1+m}}{(\mu^2+s^2)^{\frac{N+2}{2}}}\dd s+C\mu^{\frac{N+2}{2}}\\
 &=C\mu^{\frac{N-2}{2}+m}\int_0^{r_0/\mu}
 \frac{t^{N-1+m}}{(1+t^2)^{\frac{N+2}{2}}}\dd t+C\mu^{\frac{N+2}{2}}
 \leq C\mu^{\frac{N-2}{2}+m}.
\end{aligned}
\]
For (v), write
$\widehat\chi(x)=\chi(-\log|x|,x/|x|)$ in the conformal chart and extend
it by zero. 
By the Emden--Fowler transformation
\eqref{eq:EF-transform} and \eqref{eq:res-flat-bubble}, changing variables by
$x=\sqrt{1-\mu^2}\,e_N+\mu y$, we obtain
\[
\begin{aligned}
 F(\mu)&:=\int_{\Sigma_T}\chi\mathcal U_\mu^{\frac{N+2}{N-2}}\dd V\\
 &=\int_{\mathbb R^N}\widehat\chi(x)|x|^{-\frac{N-2}{2}}
 U_{\lambda_*/\mu,\,\sqrt{1-\mu^2}e_N}^{\frac{N+2}{N-2}}(x)\dd x\\
 &=\mu^{\frac{N-2}{2}}\int_{\mathbb R^N}(A_\mu(y)-1)U_{\lambda_*}^{\frac{N+2}{N-2}}(y)\dd y
 + \mu^{\frac{N-2}{2}}\int_{\mathbb R^N}U_{\lambda_*}^{\frac{N+2}{N-2}}(y)\dd y,
\end{aligned}
\]
where
\[
 A_\mu(y):=
 \widehat\chi\bigl((1-\mu^2)^{1/2}e_N+\mu y\bigr)
 \bigl|(1-\mu^2)^{1/2}e_N+\mu y\bigr|^{-\frac{N-2}{2}}.
\]
By direct differentiation and radial integration, we obtain
\[
 \sum_{j=0}^2\int_{\mathbb R^N}
 |(\mu\partial_\mu)^j(A_\mu-1)|U_{\lambda_*}^{\frac{N+2}{N-2}}\dd y
 \leq C\mu\int_0^{c/\mu}
 \frac{(1+r)r^{N-1}}{(1+r^2)^{\frac{N+2}{2}}}\dd r
 +C\int_{c/\mu}^{\infty}
 \frac{r^{N-1}}{(1+r^2)^{\frac{N+2}{2}}}\dd r
 \leq C\mu.
\]
Since $\int_{\mathbb R^N}U_{\lambda_*}^{\frac{N+2}{N-2}}\dd y=\alpha_N/c_N$,
we have, for $i=0,1,2$,
\[
\begin{aligned}
 \int_{\Sigma_T}\chi(\mu\partial_\mu)^i
 \mathcal U_\mu^{\frac{N+2}{N-2}}\dd V
 &=(\mu\partial_\mu)^iF(\mu)\\
 &=\frac{\left(\frac{N-2}{2}\right)^i\alpha_N}{c_N}\mu^{\frac{N-2}{2}}
 +(\mu\partial_\mu)^i\Biggl[
 \mu^{\frac{N-2}{2}}\int_{\mathbb R^N}(A_\mu-1)U_{\lambda_*}^{\frac{N+2}{N-2}}\dd y\Biggr]\\
 &=\frac{\left(\frac{N-2}{2}\right)^i\alpha_N}{c_N}\mu^{\frac{N-2}{2}}+O(\mu^{\frac{N}{2}}).
\end{aligned}
\]
\end{proof}

Using the periodic Green function $G_T$ from
\eqref{eq:periodized-cylinder-green}, define the projected bubble by
\begin{equation}\label{eq:res-projected-bubble}
 W_{\mu,T}(X)
 =(-L_{g_{\mathrm{cyl}}})^{-1}(\chi\mathcal U_\mu^{\frac{N+2}{N-2}})(X)
 =\int_{\Sigma_T}G_T(X,Y)\chi(Y)\mathcal U_\mu^{\frac{N+2}{N-2}}(Y)\dd V_Y.
\end{equation}
Let $\mathcal H_T=G_T-\Gamma$ be the regular part of the periodic Green
function introduced in Lemma \ref{lem:local-periodic-green}.  The key estimate for
Lemma \ref{lem:res-projection} is, for
$X\in\mathcal B_{r_0}(P)$, $\ell,j=0,1$, and $i=0,1,2$,
\begin{equation}\label{eq:res-Green-expansion}
 \Biggl|\nabla_X^\ell\Biggl[
 \partial_T^j(\mu\partial_\mu)^{i}(W_{\mu,T}-\mathcal U_\mu)(X)
 -\frac{\alpha_N}{c_N}\left(\frac{N-2}{2}\right)^i\mu^{\frac{N-2}{2}}
  \partial_T^j\mathcal H_T(X,P)\Biggr]\Biggr|
 \leq C\mu^{\frac{N}{2}}.
\end{equation}
We collect the projected bubble estimates below and prove them in
Appendix \ref{sec:res-technical-appendix}.

\begin{lemma}\label{lem:res-projection}
The estimates below hold uniformly for $T\in I$, $0<\mu<\mu_0$,
$j=0,1$, and $i=0,1,2$:
\begin{enumerate}
\renewcommand{\labelenumi}{\textup{(\roman{enumi})}}
\setlength{\itemsep}{0pt}
 \item $\begin{aligned}[t]
  C^{-1}\mu^{\frac{N-2}{2}}(\mu^2+r(X)^2)^{-\frac{N-2}{2}}
  &\leq W_{\mu,T}(X)\leq C\mu^{\frac{N-2}{2}}(\mu^2+r(X)^2)^{-\frac{N-2}{2}},
  &\quad&r(X)<r_0,\\
  C^{-1}\mu^{\frac{N-2}{2}}\delta(X)
  &\leq W_{\mu,T}(X)\leq C\mu^{\frac{N-2}{2}}\delta(X),
  &&r(X)\geq r_0.
 \end{aligned}$
 \item $\begin{aligned}[t]
  |\partial_TW_{\mu,T}(X)|
  +|\partial_T(\mu\partial_\mu W_{\mu,T})(X)|
  &\leq C\mu^{\frac{N-2}{2}},
  &\quad&r(X)<r_0,\\
  |\partial_TW_{\mu,T}(X)|
  +|\partial_T(\mu\partial_\mu W_{\mu,T})(X)|
  &\leq C\mu^{\frac{N-2}{2}}\delta(X),
  &&r(X)\geq r_0.
 \end{aligned}$
 \item $\displaystyle
 |\partial_T^j(\mu\partial_\mu)^{i}W_{\mu,T}(X)|
 \leq CW_{\mu,T}(X),\qquad X\in\Sigma_T$.
 \item $\displaystyle
  \sum_{\ell=1}^{2}r_\mu^{\ell}
   \bigl|\nabla^\ell\partial_T^j(\mu\partial_\mu)^{i}W_{\mu,T}(X)\bigr|
  +r_\mu^{2+\alpha}
 [\nabla^2\partial_T^j(\mu\partial_\mu)^{i}W_{\mu,T}]_
  {C^\alpha(\mathcal B_{cr_\mu}(X)\cap \Sigma_T)}
  \leq C\mu^{\frac{N-2}{2}} r_\mu^{2-N}$.
\end{enumerate}
\end{lemma}

To measure the residual and the correction, set
\begin{equation}\label{eq:res-rho}
 \rho_N(\mu)=
 \begin{cases}
  \mu^{N-2},&3\leq N\leq5,\\
  \mu^4\sqrt{1+|\log\mu|},&N=6,\\
  \mu^{\frac{N+2}{2}},&N\geq7,
 \end{cases}
\end{equation}
and
\begin{equation}\label{eq:res-solution-weight}
 \Psi_\mu(r)
 =\begin{cases}
   \mu^{\frac{N-2}{2}},&\text{if }0\leq r\leq\mu,\\
   \mu^{3/2}r^{-1},&\text{if }\mu<r\leq r_0\text{ and }N=3,\\
   \mu^3r^{-2}\{1+\log(r/\mu)\},
      &\text{if }\mu<r\leq r_0\text{ and }N=4,\\
   \mu^{\frac{N+2}{2}}r^{-2},&\text{if }\mu<r\leq r_0\text{ and }N\geq5,\\
   \Psi_\mu(r_0),&\text{if }r>r_0.
  \end{cases}
\end{equation}
The energy norm $\|\cdot\|_T$ fixed in the Introduction is
uniformly equivalent to the $H^1(\Sigma_T)$ norm for $T\in I$.
For functions on $\Sigma_T$, define
\begin{equation}\label{eq:res-weighted-norm}
\begin{aligned}
 \|\varphi\|
 &=\rho_N(\mu)^{-1}\|\varphi\|_T
  +\sup_{X\in \Sigma_T^\circ}
     \frac{|\varphi(X)|}{\delta(X)\Psi_\mu(r(X))} \\
 &\quad+\sup_{X\in \Sigma_T}
 \frac{\sum_{j=1}^2r_\mu(X)^{j}|\nabla^j\varphi(X)|
 +r_\mu(X)^{2+\alpha}
 [\nabla^2\varphi]_{C^\alpha(\mathcal B_{cr_\mu(X)}(X)\cap \Sigma_T)}}
 {\Psi_\mu(r(X))}.
\end{aligned}
\end{equation}
Set $Z_{\mu,T}:=\mu\partial_\mu W_{\mu,T}$.  
By the definition \eqref{eq:res-projected-bubble} of $W_{\mu,T}$, it satisfies
\begin{equation}\label{eq:res-scale-mode-equation}
 (-L_{g_{\mathrm{cyl}}})Z_{\mu,T}
 =\mu\partial_\mu(\chi\mathcal U_\mu^{\frac{N+2}{N-2}})
 =\chi(\mu\partial_\mu)\mathcal U_\mu^{\frac{N+2}{N-2}}.
\end{equation}
Set
\[
 Z_0(y):=-\lambda\partial_\lambda U_\lambda(y)
 \big|_{\lambda=\lambda_*}
 =\frac{N-2}{2}\alpha_N\frac{|y|^2-1}{(1+|y|^2)^{\frac{N}{2}}},\qquad
 \kappa_N:=\int_{\mathbb R^N}\frac{N+2}{N-2}U_{\lambda_*}^{\frac{4}{N-2}}Z_0^2\dd y>0.
\]
We first estimate the scale mode in the energy and
weighted norms.
\begin{lemma}\label{lem:res-scale-mode}
Under the assumptions of Lemma \ref{lem:res-projection},
as $\mu\downarrow0$, we have
\begin{equation}\label{eq:res-scale-mode-norms}
 \sup_{T\in I}\,\bigl|\|Z_{\mu,T}\|_T^2-\kappa_N\bigr|\longrightarrow0,\qquad
 \sup_{T\in I}\Bigl(
 \mu^{-\frac{N-2}{2}}\|\partial_TZ_{\mu,T}\|_T
 +\mu^{N-2}\|Z_{\mu,T}\|\Bigr)\leq C.
\end{equation}
\end{lemma}

\begin{proof}
Testing \eqref{eq:res-scale-mode-equation} against $Z_{\mu,T}$,
we obtain
\[
\begin{aligned}
 \|Z_{\mu,T}\|_T^2
 =\int_{\Sigma_T}\chi(\mu\partial_\mu\mathcal U_\mu)
       (\mu\partial_\mu\mathcal U_\mu^{\frac{N+2}{N-2}})\dd V
       +\int_{\Sigma_T}\chi
       (Z_{\mu,T}-\mu\partial_\mu\mathcal U_\mu)
       (\mu\partial_\mu\mathcal U_\mu^{\frac{N+2}{N-2}})\dd V.
\end{aligned}
\]
The Green expansion \eqref{eq:res-Green-expansion} and
Lemma \ref{lem:res-bubble-estimates}(iv) imply
\[
\begin{aligned}
 \Biggl|\int_{\Sigma_T}\chi
       (Z_{\mu,T}-\mu\partial_\mu\mathcal U_\mu)
       (\mu\partial_\mu\mathcal U_\mu^{\frac{N+2}{N-2}})\dd V\Biggr|
 &\leq C\mu^{\frac{N-2}{2}}\int_{\Sigma_T}
       |(\mu\partial_\mu)\mathcal U_\mu^{\frac{N+2}{N-2}}|\dd V
 \leq C\mu^{N-2}.
\end{aligned}
\]
At $x=x(y)=\sqrt{1-\mu^2}\,e_N+\mu y$, by direct differentiation,
we obtain
\[
 \mu^{\frac{N-2}{2}}(\mu\partial_\mu)
 U_{\lambda_*/\mu,\,\sqrt{1-\mu^2}e_N}(x)
 =Z_0(y)+\frac{\mu}{\sqrt{1-\mu^2}}
   \partial_{y_N}U_{\lambda_*}(y).
\]
Write $\widehat\chi(x)=\chi(-\log|x|,x/|x|)$.  
By the Emden--Fowler transformation \eqref{eq:EF-transform}
and the change of variables $x=x(y)$, we obtain
\[
\begin{aligned}
 \int_{\Sigma_T}\chi \frac{N+2}{N-2}\mathcal U_\mu^{\frac{4}{N-2}}
       (\mu\partial_\mu\mathcal U_\mu)^2\dd V
 =\int_{\mathbb R^N}
 \widehat\chi\bigl(x(y)\bigr)
 \frac{N+2}{N-2}U_{\lambda_*}^{\frac{4}{N-2}}
 \bigg(Z_0+\frac{\mu}{\sqrt{1-\mu^2}}
       \partial_{y_N}U_{\lambda_*}\bigg)^2\dd y.
\end{aligned}
\]
Since $\chi=1$ near $P$, for some $c_0>0$ the factor
$1-\widehat\chi(x(y))$ vanishes if
$|y|<c_0/\mu$. 
Then we have
\[
\begin{aligned}
 &\Biggl|\int_{\mathbb R^N}
 \Bigl\{1-\widehat\chi\bigl(x(y)\bigr)\Bigr\}
 \frac{N+2}{N-2}U_{\lambda_*}^{\frac{4}{N-2}}\\[-2pt]
 &\qquad\times
 \bigg(Z_0+\frac{\mu}{\sqrt{1-\mu^2}}
       \partial_{y_N}U_{\lambda_*}\bigg)^2\dd y\Biggr|
 \leq C\int_{c_0/\mu}^{\infty}
 \bigl(r^{-2N}+\mu^2r^{-2N-2}\bigr)r^{N-1}\dd r
 \leq C\mu^N.
\end{aligned}
\]
Moreover, $Z_0$ and $U_{\lambda_*}$ are radial, whereas
$\partial_{y_N}U_{\lambda_*}$ is odd in $y_N$.  Therefore
\[
\begin{aligned}
 \int_{\mathbb R^N}\frac{N+2}{N-2}U_{\lambda_*}^{\frac{4}{N-2}}
 \bigg(Z_0+\frac{\mu}{\sqrt{1-\mu^2}}
       \partial_{y_N}U_{\lambda_*}\bigg)^2\dd y 
 &=\kappa_N+\frac{\mu^2}{1-\mu^2}
 \int_{\mathbb R^N}\frac{N+2}{N-2}U_{\lambda_*}^{\frac{4}{N-2}}
       |\partial_{y_N}U_{\lambda_*}|^2\dd y\\
 &=\kappa_N+O(\mu^2).
\end{aligned}
\]
Combining these estimates,
we obtain
$
 \|Z_{\mu,T}\|_T^2=\kappa_N+O(\mu^2+\mu^{N-2}),
$
which shows the first assertion.

We next estimate the period derivative.  Applying the differentiation
rule on the fixed domain following \eqref{eq:fixed-domain-green-pullback} to
\eqref{eq:res-scale-mode-equation}, we obtain
\[
 (-L_{g_T})\partial_TZ_{\mu,T}
 =-\bigl(\partial_T(-L_{g_T})\bigr)Z_{\mu,T}.
\]
The coefficients of $\partial_TL_{g_T}$ are uniformly bounded and
supported away from $P$.  Testing this equation with
$\partial_TZ_{\mu,T}$, and then using Lemma \ref{lem:res-projection}, we obtain
\[
\begin{aligned}
 \|\partial_TZ_{\mu,T}\|_T^2
 &\leq C\Bigl|\Bigl\langle
   \bigl(\partial_T(-L_{g_T})\bigr)Z_{\mu,T},
   \partial_TZ_{\mu,T}\Bigr\rangle\Bigr|
 \\&\leq C\|Z_{\mu,T}\|_{H^1(\operatorname{supp}\partial_TL_{g_T})}
          \|\partial_TZ_{\mu,T}\|_T
 \leq C\mu^{\frac{N-2}{2}}\|\partial_TZ_{\mu,T}\|_T.
\end{aligned}
\]
Consequently, $\|\partial_TZ_{\mu,T}\|_T\leq C\mu^{\frac{N-2}{2}}$.

We finally estimate the weighted norm.  From the first assertion,
$\|Z_{\mu,T}\|_T\leq C$.  Since the energy and $H^1$ norms are uniformly
equivalent and $\mu^{N-2}/\rho_N(\mu)\leq C$, we obtain
\[
 \mu^{N-2}\rho_N(\mu)^{-1}
 \|Z_{\mu,T}\|_{H^1(\Sigma_T)}\leq C.
\]
By Lemma \ref{lem:res-projection} and
\eqref{eq:res-solution-weight},
applied separately when $r\leq\mu$, $\mu<r<r_0$, and $r\geq r_0$,
we have the direct estimates
\[
 \mu^{N-2}\frac{|Z_{\mu,T}(X)|}{\delta(X)\Psi_\mu(r(X))}
 \leq C\mu^{N-2}
 \frac{W_{\mu,T}(X)}{\delta(X)\Psi_\mu(r(X))}\leq C,
\]
\[
 \begin{aligned}
 &\mu^{N-2}
 \frac{\displaystyle\sum_{\ell=1}^{2}r_\mu(X)^\ell
 |\nabla^\ell Z_{\mu,T}(X)|
 +r_\mu(X)^{2+\alpha}
 [\nabla^2Z_{\mu,T}]_{C^\alpha(\mathcal B_{cr_\mu(X)}(X)\cap\Sigma_T)}}
 {\Psi_\mu(r(X))}\leq C\mu^{\frac{3(N-2)}{2}}
 \frac{r_\mu(X)^{2-N}}{\Psi_\mu(r(X))}\leq C.
 \end{aligned}
\]
Together with the $H^1$ estimate above, this proves
$\mu^{N-2}\|Z_{\mu,T}\|\leq C$ and completes the proof.
\end{proof}

To estimate the residual, choose a fixed cutoff function
$\chi_0=\chi_0(r)$ on the cylinder such that $\chi_0=0$ if
$r\leq r_0/2$ and $\chi_0=1$ if $r\geq r_0$.  With $\Psi_\mu$
defined in \eqref{eq:res-solution-weight}, set
\begin{equation*}
 \Theta_\mu(X)
 =(1-\chi_0(X))\mu^{\frac{N+2}{2}}(\mu^2+r(X)^2)^{-2}
   +\chi_0(X)\Psi_\mu(r_0).
\end{equation*}
Thus
\begin{equation}\label{eq:res-source-weight-regions}
 \Theta_\mu=\mu^{\frac{N+2}{2}}(\mu^2+r^2)^{-2}\quad\text{if }r\leq r_0/2,
 \qquad
 \Theta_\mu=\Psi_\mu(r_0)\quad\text{if }r\geq r_0,
\end{equation}
and on $r_0/2<r<r_0$ we have
$C^{-1}\mu^{\frac{N+2}{2}}\leq\Theta_\mu\leq C\Psi_\mu(r_0)$.
Define the residual by
\begin{equation}\label{eq:residual-identity}
 R_{\mu,T}:=(-L_{g_{\mathrm{cyl}}})W_{\mu,T}-W_{\mu,T}^{\frac{N+2}{N-2}}
 =\chi\mathcal U_\mu^{\frac{N+2}{N-2}}-W_{\mu,T}^{\frac{N+2}{N-2}}.
\end{equation}
\begin{lemma}\label{lem:res-residual-estimates}
The following estimates hold uniformly for $T\in I$ and
$0<\mu<\mu_0$:
\begin{enumerate}
\renewcommand{\labelenumi}{\textup{(\roman{enumi})}}
\setlength{\itemsep}{0pt}
 \item $\displaystyle
 |\partial_T^j(\mu\partial_\mu)^{i}R_{\mu,T}(X)|
 +r_\mu(X)^\alpha
 [\partial_T^j(\mu\partial_\mu)^{i}R_{\mu,T}]_
 {C^\alpha(\mathcal B_{cr_\mu(X)}(X)\cap\Sigma_T)}
 \leq C\Theta_\mu(X)$ for $j,i\in\{0,1\}$.
 \item $\displaystyle
 \|\partial_T^jR_{\mu,T}\|_{H^{-1}(\Sigma_T)}
 +\|\partial_T^j(\mu\partial_\mu R_{\mu,T})\|_{H^{-1}(\Sigma_T)}
 \leq C\rho_N(\mu)$ for $j\in\{0,1\}$.
 \item $\displaystyle
 \Big\|T\longmapsto
 \mu^{2-N}\int_{\Sigma_T}R_{\mu,T}Z_{\mu,T}\dd V
 +\frac{(N-2)\alpha_N^2}{2c_N}\mathscr M_N(T)\,
 \Big\|_{C^1(I)}\longrightarrow0$ as $\mu\downarrow0$.
\end{enumerate}
\end{lemma}

\begin{proof}
\emph{(i).}
We first take $j=i=0$.  
If $r<r_0/2$, then $\chi=1$, and by
\eqref{eq:residual-identity}, \eqref{eq:res-Green-expansion}, and
Lemma \ref{lem:res-projection}, 
\[
 |R_{\mu,T}|\leq C\mathcal U_\mu^{\frac{4}{N-2}}|W_{\mu,T}-\mathcal U_\mu|
 \leq C\mu^{\frac{N-2}{2}}\mathcal U_\mu^{\frac{4}{N-2}}
 \leq C\mu^{\frac{N+2}{2}}(\mu^2+r^2)^{-2}=C\Theta_\mu.
\]
Differentiating the same identity and applying the same
estimates, we obtain, for $r<r_0/2$,
\[
\begin{aligned}
 r_\mu|\nabla R_{\mu,T}|
 &\leq Cr_\mu\mathcal U_\mu^{\frac{4}{N-2}}
       |\nabla(W_{\mu,T}-\mathcal U_\mu)|
   +Cr_\mu\mathcal U_\mu^{\frac{6-N}{N-2}}
       |W_{\mu,T}-\mathcal U_\mu|\,|\nabla\mathcal U_\mu|
 \leq C\mu^{\frac{N-2}{2}}\mathcal U_\mu^{\frac{4}{N-2}}\leq C\Theta_\mu.
\end{aligned}
\]
On the remaining regions, it follows from \eqref{eq:res-flat-bubble} and Lemma \ref{lem:res-projection}(i), (iv) that
\[
\begin{aligned}
 |R_{\mu,T}|+r_\mu|\nabla R_{\mu,T}|
 &\leq C\mu^{\frac{N+2}{2}}\leq C\Theta_\mu,
 &&\text{if }r_0/2\leq r\leq2r_0,\\
 |R_{\mu,T}|+r_\mu|\nabla R_{\mu,T}|
 &\leq C\mu^{\frac{N+2}{2}}(\delta^{\frac{N+2}{N-2}}+\delta^{\frac{4}{N-2}})
 \leq C\Theta_\mu,
 &&\text{if }r\geq2r_0.
\end{aligned}
\]

If $r(X)<r_0/4$, the ball
$\mathcal B_{cr_\mu(X)}(X)$ lies in $\mathcal B_{r_0/2}(P)$, and on
this ball $r_\mu\asymp r_\mu(X)$ and $\Theta_\mu\asymp\Theta_\mu(X)$.
By the core estimates above, we obtain
\[
\begin{aligned}
 r_\mu(X)^\alpha
 [R_{\mu,T}]_{C^\alpha(\mathcal B_{cr_\mu(X)}(X)\cap\Sigma_T)}
 &\leq C\Bigl(\|R_{\mu,T}\|_{L^\infty}
 +r_\mu(X)\|\nabla R_{\mu,T}\|_{L^\infty}\Bigr)
 \leq C\Theta_\mu(X),
\end{aligned}
\]
where the two norms are taken over the same ball.
If $r(X)\geq r_0/4$, the same ball stays outside
$\mathcal B_{r_0/8}(P)$. The preceding pointwise and gradient estimates,
also valid on $r_0/8\leq r<r_0/2$, give
\[
 r_\mu(X)^\alpha
 [R_{\mu,T}]_{C^\alpha(\mathcal B_{cr_\mu(X)}(X)\cap\Sigma_T)}
 \leq C\mu^{\frac{N+2}{2}}\leq C\Theta_\mu(X).
\]
This proves (i) for $j=i=0$.

For the parameter derivatives, we differentiate the identity
\[
 R_{\mu,T}
 =-\frac{N+2}{N-2}(W_{\mu,T}-\mathcal U_\mu)\int_0^1
       \bigl((1-\tau)\mathcal U_\mu+\tau W_{\mu,T}\bigr)^{\frac{4}{N-2}}\dd\tau
\]
in the core. By \eqref{eq:res-Green-expansion} and Lemma \ref{lem:res-projection},
the factors $\partial_T^j(\mu\partial_\mu)^i(W_{\mu,T}-\mathcal U_\mu)$ and their scaled gradients
are bounded by $C\mu^{\frac{N-2}{2}}$, whereas the remaining factors are controlled
by $\mathcal U_\mu^{\frac{4}{N-2}}$. Thus, for $i,j\in\{0,1\}$ and $r<r_0/2$,
\[
 |\partial_T^j(\mu\partial_\mu)^iR_{\mu,T}|
 +r_\mu|\nabla\partial_T^j(\mu\partial_\mu)^iR_{\mu,T}|
 \leq C\mu^{\frac{N-2}{2}}\mathcal U_\mu^{\frac{4}{N-2}}.
\]
Away from the core, all differentiated source terms are
$O(\mu^{\frac{N+2}{2}})$. The only new product in the mixed derivative of
$W_{\mu,T}^{\frac{N+2}{N-2}}$ is
$W_{\mu,T}^{\frac{6-N}{N-2}}(\partial_TW_{\mu,T})(\mu\partial_\mu W_{\mu,T})$.
By Lemma \ref{lem:res-projection}, near the lateral
boundary we have
\[
 \begin{aligned}
 |W_{\mu,T}^{\frac{6-N}{N-2}}(\partial_TW_{\mu,T})(\mu\partial_\mu W_{\mu,T})|
 &\leq C\mu^{\frac{N+2}{2}}\delta^{\frac{N+2}{N-2}},\\
 |\nabla\{W_{\mu,T}^{\frac{6-N}{N-2}}(\partial_TW_{\mu,T})(\mu\partial_\mu W_{\mu,T})\}|
 &\leq C\mu^{\frac{N+2}{2}}\delta^{\frac{4}{N-2}}.
 \end{aligned}
\]
These estimates hold also when $1<\frac{N+2}{N-2}<2$, since
$|\partial_TW_{\mu,T}|+|\mu\partial_\mu W_{\mu,T}|\leq CW_{\mu,T}$
and both factors vanish on the boundary.
The remaining products satisfy the same bounds. Applying the preceding
core and exterior interpolation arguments to each derivative proves (i).

\emph{(ii).}
For $\eta\in H_0^1(\Sigma_T)$,
\eqref{eq:res-source-weight-regions}, the Cauchy--Schwarz inequality,
and Hardy's inequality give
\[
\begin{aligned}
 \int_{\Sigma_T}\Theta_\mu|\eta|\dd V
 &\leq C\mu^{\frac{N+2}{2}}\int_{\mathcal B_{r_0}(P)}
       \frac{|\eta|}{(\mu^2+r^2)^2}\dd V
       +C\Psi_\mu(r_0)\|\eta\|_{L^1}\\
 &\leq C\Bigl(\mu^{N+2}\int_{\mathcal B_{r_0}(P)}
       \frac{r^2}{(\mu^2+r^2)^4}\dd V\Bigr)^{1/2}
       \Bigl(\int_{\Sigma_T}\frac{|\eta|^2}{r^2}\dd V\Bigr)^{1/2}
       +C\Psi_\mu(r_0)\|\eta\|_{H^1}\\
 &\leq C\|\eta\|_{H^1}
       \Bigl(\mu^{N+2}\int_0^{r_0}
       \frac{r^{N+1}}{(\mu^2+r^2)^4}\dd r\Bigr)^{1/2}
       +C\Psi_\mu(r_0)\|\eta\|_{H^1}\\
 &\leq C\rho_N(\mu)\|\eta\|_{H^1}.
\end{aligned}
\]
Combining this inequality with (i), we obtain (ii).

\emph{(iii).}
Fix $K>1$ and take $\mu>0$ with $K\mu<r_0/2$.
Split $\Sigma_T$ into the core, neck, and outer regions
\[
 \mathcal B_{K\mu}(P),\qquad
 \mathcal B_{2r_0}(P)\setminus\mathcal B_{K\mu}(P),\qquad
 \Sigma_T\setminus\mathcal B_{2r_0}(P).
\]
We compute the core contribution in $C^1(I)$ and show that
the other two contributions vanish as first $\mu\downarrow0$ and then
$K\to\infty$.

\emph{Core region.}
Set $X=\exp_P(\mu y)$.  
Since $K\mu<r_0/2$, in exponential coordinates, we have
\[
 \mathcal B_{K\mu}(P)=\exp_P(\mu B_K),\qquad
 (\exp_P(\mu\,\cdot))^*\dd V
 =\mu^N\bigl(1+O_K(\mu^2)\bigr)\dd y.
\]
Using \eqref{eq:res-Green-expansion} and
$\mathcal H_T(P,P)=c_N\mathscr M_N(T)$, we obtain the following limits
in $C^1(I;C^0(B_K))$:
\[
\begin{aligned}
 \mu^{\frac{N-2}{2}}Z_{\mu,T}(\exp_P(\mu y))
 &\longrightarrow Z_0(y),\\
 \mu^{-\frac{N-2}{2}}(W_{\mu,T}-\mathcal U_\mu)(\exp_P(\mu y))
 &\longrightarrow \alpha_N\mathscr M_N(T),\\
 \mu^{\frac{6-N}{2}}R_{\mu,T}(\exp_P(\mu y))
 &\longrightarrow -\frac{N+2}{N-2}\alpha_N\mathscr M_N(T)
 U_{\lambda_*}^{\frac{4}{N-2}}(y).
\end{aligned}
\]
Hence in $C^1(I)$, we have
\[
\begin{aligned}
 \mu^{2-N}\int_{\mathcal B_{K\mu}(P)}R_{\mu,T}Z_{\mu,T}\dd V
 &=\int_{B_K}
 \bigl[\mu^{\frac{6-N}{2}}R_{\mu,T}(\exp_P(\mu y))\bigr]
 \bigl[\mu^{\frac{N-2}{2}}Z_{\mu,T}(\exp_P(\mu y))\bigr]
 \bigl(1+O_K(\mu^2)\bigr)\dd y\\
 &\longrightarrow
 -\frac{N+2}{N-2}\alpha_N\mathscr M_N(T)
 \int_{B_K}U_{\lambda_*}^{\frac{4}{N-2}}Z_0\dd y.
\end{aligned}
\]

\begin{samepage}
\emph{Neck region.}
Splitting at $r=r_0$ and using the estimates proved in (i) and
Lemma \ref{lem:res-projection}, we obtain
\[
\begin{aligned}
 &\quad\mu^{2-N}\int_{\mathcal B_{2r_0}(P)\setminus
 \mathcal B_{K\mu}(P)}
 \bigl(|R_{\mu,T}Z_{\mu,T}|
 +|\partial_TR_{\mu,T}||Z_{\mu,T}|
 +|R_{\mu,T}||\partial_TZ_{\mu,T}|\bigr)\dd V \\
 &\leq C\mu^{2-N}\Biggl(
 \int_{K\mu}^{r_0}
 \mu^{\frac{N+2}{2}}r^{-4}\,\mu^{\frac{N-2}{2}}r^{2-N}r^{N-1}\dd r
 +\mu^N\Biggr)\\
 &\leq C\mu^2\int_{K\mu}^{r_0}r^{-3}\dd r+C\mu^2
 \leq CK^{-2}.
\end{aligned}
\]
\end{samepage}

\emph{Outer region.}
On $\Sigma_T\setminus\mathcal B_{2r_0}(P)$ we have $\chi=0$, and hence
$R_{\mu,T}=-W_{\mu,T}^{\frac{N+2}{N-2}}$.  By
\eqref{eq:fixed-domain-interval}--\eqref{eq:fixed-domain-lift}, we have
$J_T\geq1/2$ and $|\partial_TJ_T|=|\beta|\leq C$.
Using Lemma \ref{lem:res-projection}, we obtain 
\[
 \begin{aligned}
 &\quad\mu^{2-N}\int_{\Sigma_T\setminus\mathcal B_{2r_0}(P)}
 \bigl(|R_{\mu,T}Z_{\mu,T}|
 +|\partial_TR_{\mu,T}||Z_{\mu,T}|
 +|R_{\mu,T}||\partial_TZ_{\mu,T}|
 +|R_{\mu,T}Z_{\mu,T}|\,|\partial_TJ_T|/J_T\bigr)\dd V\\
 &\leq C\mu^2
 \int_{\Sigma_T\setminus\mathcal B_{2r_0}(P)}\delta^{\frac{2N}{N-2}}\dd V
 \leq C\mu^2=o(1).
 \end{aligned}
\]
To identify the coefficient after $K\to\infty$, we compute
\[
 \frac{N+2}{N-2}\int_{\mathbb R^N}U_{\lambda_*}^{\frac{4}{N-2}}Z_0\dd y
 =\lambda\partial_\lambda\int_{\mathbb R^N}
 \Biggl\{\alpha_N\Biggl(\frac{\lambda}
 {\lambda^2+|y|^2}\Biggr)^{\frac{N-2}{2}}\Biggr\}^{\frac{N+2}{N-2}}\dd y\bigg|_{\lambda=1}
 =\frac{(N-2)\alpha_N}{2c_N}.
\]
Letting first $\mu\downarrow0$ and then $K\to\infty$, we combine the
above estimates with this identity to prove (iii).
\end{proof}

\subsection{Uniform linear theory}

We next invert the linearized operator modulo the scaling direction
$Z_{\mu,T}$, uniformly in the concentration scale and the period.
\begin{lemma}\label{lem:res-inverse}
There exist $\mu_0>0$ and $C<\infty$ such that, for
$0<\mu<\mu_0$, $T\in I$, every
$h\in H^{-1}_{\mathrm{ax}}(\Sigma_T)$ that is even in $t$, and every
$\gamma\in\mathbb R$, there exists a unique
pair $(\varphi,b)\in H^1_{0,\mathrm{ax}}(\Sigma_T)\times\mathbb R$, with
$\varphi$ even in $t$, satisfying
\begin{equation}\label{eq:res-augmented-linear-problem}
 \begin{cases}
 \left(-L_{g_{\mathrm{cyl}}}
 -\frac{N+2}{N-2}W_{\mu,T}^{\frac{4}{N-2}}\right)\varphi
   =h+b(-L_{g_{\mathrm{cyl}}})Z_{\mu,T},\\
 \langle\varphi,Z_{\mu,T}\rangle_T=\gamma,
 \end{cases}
\end{equation}
and
\begin{equation}\label{eq:res-energy-inverse}
 \|\varphi\|_T+|b|
 \leq C\bigl(\|h\|_{H^{-1}(\Sigma_T)}+|\gamma|\bigr).
\end{equation}
\end{lemma}

\begin{proof}
We first consider the case $\gamma=0$.  
Differentiating \eqref{eq:residual-identity} gives
\begin{equation}\label{eq:res-scale-residual-identity}
 \left(-L_{g_{\mathrm{cyl}}}
 -\frac{N+2}{N-2}W_{\mu,T}^{\frac{4}{N-2}}\right)Z_{\mu,T}
 =\mu\partial_\mu R_{\mu,T}.
\end{equation}
Pairing \eqref{eq:res-augmented-linear-problem} with \(Z_{\mu,T}\)
and using \eqref{eq:res-scale-residual-identity}, we obtain
\[
 b\|Z_{\mu,T}\|_T^2
 =\int_{\Sigma_T}\varphi\,\mu\partial_\mu R_{\mu,T}\dd V
  -\langle h,Z_{\mu,T}\rangle.
\]
Hence, by Lemma \ref{lem:res-scale-mode} and
Lemma \ref{lem:res-residual-estimates}(ii),
\[
 |b|\leq C\Bigl(
 \rho_N(\mu)\|\varphi\|_T+\|h\|_{H^{-1}}\Bigr).
\]
If the desired estimate for $\gamma=0$ failed, then, using this
bound and $\rho_N(\mu)\to0$, we could normalize and choose
$\mu_n\downarrow0$ and $T_n\to T\in I$.
Set $W_n:=W_{\mu_n,T_n}$ and $Z_n:=Z_{\mu_n,T_n}$.  Then
\[
 \|\varphi_n\|_{T_n}=1,\qquad
 \langle\varphi_n,Z_n\rangle_{T_n}=0,\qquad
 \|h_n\|_{H^{-1}}\to0.
\]
Moreover,
\begin{equation}\label{eq:res-inverse-contradiction-equation}
 \left(-L_{g_{\mathrm{cyl}}}
 -\frac{N+2}{N-2}W_n^{\frac{4}{N-2}}\right)\varphi_n
 =h_n+b_n(-L_{g_{\mathrm{cyl}}})Z_n.
\end{equation}
Indeed, by the preceding estimate, we have $b_n\to0$.

Fix $r\in(0,r_0)$ sufficiently small and $K>1$ sufficiently large.
For all sufficiently large $n$, we have $2K\mu_n<r/2$ and split
$\Sigma_{T_n}$ into the outer, core, and neck regions
\[
 \Sigma_{T_n}\setminus\mathcal B_{r/2}(P),\qquad
 \mathcal B_{2K\mu_n}(P),\qquad
 \mathcal B_{r/2}(P)\setminus\mathcal B_{2K\mu_n}(P).
\]
We prove that the energy on each region tends to zero,
contradicting $\|\varphi_n\|_{T_n}=1$.

\emph{Outer region.}
Use $\Phi_{T_n}\circ\Phi_T^{-1}$ to identify $\Sigma_T$ with
$\Sigma_{T_n}$.  After passing to a subsequence,
$\varphi_n\rightharpoonup\varphi$ in $H_0^1(\Sigma_T)$; we use the same
identification for the test functions.  For every
$\zeta\in C_c^\infty(\Sigma_T\setminus\{P\})$,
\eqref{eq:res-inverse-contradiction-equation} and
Lemma \ref{lem:res-projection}(i) give
\[
\begin{aligned}
 \langle -L_{g_{\mathrm{cyl}}}\varphi, \zeta\rangle=\langle\varphi,\zeta\rangle_T
 =\lim_{n\to\infty}\Biggl\{
 \frac{N+2}{N-2}\int_{\Sigma_{T_n}}W_n^{\frac{4}{N-2}}\varphi_n\zeta\dd V
 +\langle h_n,\zeta\rangle
 +b_n\langle Z_n,\zeta\rangle_{T_n}\Biggr\}=0.
\end{aligned}
\]
Since $\operatorname{Cap}_{H^1}(\{P\})=0$ for $N\geq3$,
we have $-L_{g_{\mathrm{cyl}}}\varphi=0$ on $\Sigma_T$ and hence $\varphi=0$.
Choose $\eta_r=0$ on $\mathcal B_{r/4}(P)$ and $\eta_r=1$ outside
$\mathcal B_{r/2}(P)$.  
Testing \eqref{eq:res-inverse-contradiction-equation} against
$\eta_r^2\varphi_n$, we obtain
\[
\begin{aligned}
 \|\eta_r\varphi_n\|_{T_n}^2
 &=\frac{N+2}{N-2}\int_{\Sigma_{T_n}}W_n^{\frac{4}{N-2}}\eta_r^2\varphi_n^2\dd V
 +\langle h_n,\eta_r^2\varphi_n\rangle
 +b_n\langle Z_n,\eta_r^2\varphi_n\rangle_{T_n}
 +\int_{\Sigma_{T_n}}|\nabla\eta_r|^2\varphi_n^2\dd V\\
 &\leq o(1)\|\eta_r\varphi_n\|_{T_n}^2+o(1)
 +C\int_{\mathcal B_{r/2}(P)\setminus\mathcal B_{r/4}(P)}
 \varphi_n^2\dd V=o(1).
\end{aligned}
\]
The last integral tends to zero by the compact embedding $H^1\hookrightarrow L^2$ on the fixed annulus.  Since
$\eta_r=1$ outside $\mathcal B_{r/2}(P)$, the first assertion follows.

\emph{Core region.}
Choose $\eta\in C_c^\infty(\mathcal B_{r_0}(P))$ with $\eta=1$ on
$\mathcal B_{r_0/2}(P)$. Set $F_n(y)=\exp_P(\mu_n y)$ for
$y\in B_{r_0/\mu_n}$, and define
\[
 \widehat\varphi_n(y)=\mu_n^{\frac{N-2}{2}}(\eta\varphi_n)(F_n(y))
 \quad\text{on }B_{r_0/\mu_n},
 \qquad \widehat\varphi_n=0\quad\text{elsewhere}.
\]
By scaling and the energy bound,
\[
 \|\nabla\widehat\varphi_n\|_{L^2(\mathbb R^N)}
 +\|\widehat\varphi_n\|_{L^{\frac{2N}{N-2}}(\mathbb R^N)}\leq C.
\]
On $B_{r_0/\mu_n}$, write
$\widehat g_n=\mu_n^{-2}F_n^*g_{\mathrm{cyl}}$ and
$\widehat W_n=\mu_n^{\frac{N-2}{2}}W_n\circ F_n$.
By \eqref{eq:res-Green-expansion},
Lemma \ref{lem:res-projection}, and local compactness, we may pass to a
subsequence so that
\[
 \begin{aligned}
 \widehat\varphi_n&\rightharpoonup\widehat\varphi
 &&\text{in }\mathcal D^{1,2}(\mathbb R^N),\\
 \widehat\varphi_n&\longrightarrow\widehat\varphi
 &&\text{in }L^2_{\mathrm{loc}}(\mathbb R^N),\\
 \widehat W_n&\longrightarrow U_{\lambda_*}
 &&\text{in }C^2_{\mathrm{loc}}(\mathbb R^N).
 \end{aligned}
\]
For every fixed ball and all sufficiently large $n$, the cutoff is one
and the rescaled equation is
\[
 \bigl(-\Delta_{\widehat g_n}+\frac{(N-2)^2}{4}\mu_n^2
       -\frac{N+2}{N-2}\widehat W_n^{\frac{4}{N-2}}\bigr)\widehat\varphi_n
 =\mu_n^{\frac{N+2}{2}}
   \bigl(h_n+b_n(-L_{g_{\mathrm{cyl}}})Z_n\bigr)\circ F_n.
\]
The first term on the right is understood distributionally. Testing it
against a fixed test function amounts to testing $h_n$ against the
corresponding function scaled by $\mu_n^{-\frac{N-2}{2}}$, whose $H^1$ norm is
uniformly bounded. Hence this term tends to zero in
$H^{-1}_{\mathrm{loc}}$. The second term tends to zero in the same
space, since $b_n\to0$ and $\mu_n^{\frac{N-2}{2}}Z_n\circ F_n$ is locally bounded
in $H^1$. We therefore obtain
\[
 \left(-\Delta-\frac{N+2}{N-2}
 U_{\lambda_*}^{\frac{4}{N-2}}\right)\widehat\varphi=0
 \quad\text{in }\mathbb R^N.
\]
By Euclidean nondegeneracy, as stated by Rey \cite[Appendix D]{Rey1990},
and the imposed axial symmetry and evenness, $\widehat\varphi=cZ_0$.
To pass the orthogonality condition to the limit, use
$(-L_{g_{\mathrm{cyl}}})Z_n=\chi\mu_n\partial_{\mu_n}\mathcal U_{\mu_n}^{\frac{N+2}{N-2}}$.
For fixed $R>1$ and large $n$, H\"older's inequality and
Lemma \ref{lem:res-bubble-estimates} give
\[
 \begin{aligned}
 \left|\int_{r>R\mu_n}\varphi_n
       (-L_{g_{\mathrm{cyl}}})Z_n\dd V\right|
 \leq C\|\chi\mu_n\partial_{\mu_n}\mathcal U_{\mu_n}^{\frac{N+2}{N-2}}\|_
 {L^{2N/(N+2)}(\{r>R\mu_n\})}
 \leq CR^{-(N+2)/2}.
 \end{aligned}
\]
On $r<R\mu_n$, the integral converges by local $L^2$ convergence.
Letting first $n\to\infty$ and then $R\to\infty$, we obtain
\[
 0=\lim_{n\to\infty}\langle\varphi_n,Z_n\rangle_{T_n}
 =\int_{\mathbb R^N}\frac{N+2}{N-2}U_{\lambda_*}^{\frac{4}{N-2}}Z_0\widehat\varphi\dd y
 =c\kappa_N.
\]
Hence $\widehat\varphi=0$.
Let $\zeta_K=1$ on $B_{2K}$ and
$\operatorname{supp}\zeta_K\subset B_{4K}$.  Testing
\eqref{eq:res-inverse-contradiction-equation} in the rescaled coordinates
against $\zeta_K^2\widehat\varphi_n$, we obtain
\[
 \|\varphi_n\|_{H^1(\mathcal B_{2K\mu_n}(P))}^2
 \leq C_K\int_{B_{4K}}\widehat\varphi_n^2\dd y+o(1)=o(1).
\]

\emph{Neck region.}
By Lemma \ref{lem:res-projection}(i) and the choice of $K$,
\[
 C\|\frac{N+2}{N-2}W_n^{\frac{4}{N-2}}\|_{L^{\frac{N}{2}}
 (\mathcal B_r(P)\setminus\mathcal B_{K\mu_n}(P))}
 \leq C\Biggl(\int_{K\mu_n}^{r}
 \mu_n^Ns^{-2N}s^{N-1}\dd s\Biggr)^{2/N}
 \leq CK^{-2}<\frac13.
\]
Let $\eta_{n,K,r}$ equal one on
$\mathcal B_{r/2}(P)\setminus\mathcal B_{2K\mu_n}(P)$ and be supported
in $\mathcal B_r(P)\setminus\mathcal B_{K\mu_n}(P)$.  Testing
\eqref{eq:res-inverse-contradiction-equation} with
$\eta_{n,K,r}^2\varphi_n$ and using Sobolev's inequality, we obtain
\[
 \begin{aligned}
  \|\eta_{n,K,r}\varphi_n\|_{T_n}^2
  \leq C\|\frac{N+2}{N-2}W_n^{\frac{4}{N-2}}\|_{L^{\frac{N}{2}}}
       \|\eta_{n,K,r}\varphi_n\|_{T_n}^2
  +C\int_{\operatorname{supp}\nabla\eta_{n,K,r}}
       |\nabla\eta_{n,K,r}|^2\varphi_n^2\dd V+o(1).
 \end{aligned}
\]
Absorbing the first term proves the neck region assertion by the core and outer limits:
\[
 \int_{\operatorname{supp}\nabla\eta_{n,K,r}}
 |\nabla\eta_{n,K,r}|^2\varphi_n^2\dd V
 \leq C_K\int_{B_{2K}\setminus B_K}\widehat\varphi_n^2\dd y
 +C_r\int_{\mathcal B_r(P)\setminus\mathcal B_{r/2}(P)}
       \varphi_n^2\dd V=o(1).
\]
Combining this with the outer and core assertions, we obtain
$\|\varphi_n\|_{T_n}\to0$, contradicting $\|\varphi_n\|_{T_n}=1$.
Hence the desired estimate holds when $\gamma=0$.

For general $\gamma$, set
\[
 \widetilde\varphi
 :=\varphi-\frac{\gamma}{\|Z_{\mu,T}\|_T^2}Z_{\mu,T}.
\]
Then $\langle\widetilde\varphi,Z_{\mu,T}\rangle_T=0$ and, by
\eqref{eq:res-scale-residual-identity},
\[
 \left(-L_{g_{\mathrm{cyl}}}
 -\frac{N+2}{N-2}W_{\mu,T}^{\frac{4}{N-2}}\right)\widetilde\varphi
 =h-\frac{\gamma}{\|Z_{\mu,T}\|_T^2}\mu\partial_\mu R_{\mu,T}
  +b(-L_{g_{\mathrm{cyl}}})Z_{\mu,T}.
\]
By the estimate already proved and Lemmas \ref{lem:res-scale-mode}
and \ref{lem:res-residual-estimates}(ii), we have
\[
 \|\widetilde\varphi\|_T+|b|
 \leq C\Biggl\|h-
 \frac{\gamma}{\|Z_{\mu,T}\|_T^2}\mu\partial_\mu R_{\mu,T}
 \Biggr\|_{H^{-1}}
 \leq C\bigl(\|h\|_{H^{-1}}+|\gamma|\bigr),
\]
and hence
\[
 \|\varphi\|_T+|b|
 \leq \|\widetilde\varphi\|_T+|b|
 +\frac{|\gamma|}{\|Z_{\mu,T}\|_T}
 \leq C\bigl(\|h\|_{H^{-1}}+|\gamma|\bigr).
\]
Finally, the map
\[
 (\varphi,b)\longmapsto
 \Biggl(
 \left(-L_{g_{\mathrm{cyl}}}
 -\frac{N+2}{N-2}W_{\mu,T}^{\frac{4}{N-2}}\right)\varphi
 -b(-L_{g_{\mathrm{cyl}}})Z_{\mu,T},
 \ \langle\varphi,Z_{\mu,T}\rangle_T
 \Biggr)
\]
is Fredholm of index zero on these symmetry subspaces.
By the preceding estimate, its kernel is trivial;
hence it is bijective, and \eqref{eq:res-augmented-linear-problem}
admits a unique solution for every $(h,\gamma)$.
\end{proof}

We next extend the $H^1$ estimate in Lemma \ref{lem:res-inverse}
to the weighted norm, including one derivative with respect to $T$.
We prove the estimate in Appendix \ref{sec:res-technical-appendix} and use it in
Proposition \ref{prop:res-correction} to construct the nonlinear correction
and control its dependence on $T$.

\begin{lemma}\label{lem:res-weighted-inverse}
Let $m\in\{0,1\}$, and assume the hypotheses of
Lemma \ref{lem:res-inverse}.  For each $T\in I$, let
$(\varphi_T,b_T)$ solve \eqref{eq:res-augmented-linear-problem} with data
$h=h_T$ and $\gamma=\gamma_T$.
When $m=1$, assume that, after pullback, the data are continuously
differentiable in $T$, with values in
$C^\alpha_{\mathrm{ax}}(\overline{\Sigma_{T^*}})\times\mathbb R$.
Suppose that, for some $A\geq0$ and all $0\leq j\leq m$,
\begin{equation}\label{eq:res-weighted-data}
 |\partial_T^jh_T(X)|
 +r_\mu(X)^{\alpha}
 [\partial_T^jh_T]_
 {C^\alpha(\mathcal B_{cr_\mu(X)}(X)\cap \Sigma_T)}
 \leq A\Theta_\mu(X),\qquad
 |\partial_T^j\gamma_T|\leq A\mu^{N-2},
\end{equation}
then
\begin{equation}\label{eq:res-weighted-inverse}
 \sum_{j=0}^m\Bigl(
 \|\partial_T^j\varphi_T\|
 +\mu^{2-N}|\partial_T^jb_T|\Bigr)\leq CA.
\end{equation}
\end{lemma}

\subsection{Nonlinear projected problem}

We now use the contraction mapping principle to solve for
the correction, leaving a residual only in the scaling direction.
\begin{proposition}\label{prop:res-correction}
There are $\mu_0>0$ and $K<\infty$ such that, for
$0<\mu<\mu_0$ and $T\in I$, there is a unique pair
$(\varphi_{\mu,T},\lambda_{\mu,T})$ with
$\varphi_{\mu,T}\in H^1_{0,\mathrm{ax}}(\Sigma_T)$ even in $t$,
$\langle\varphi_{\mu,T},Z_{\mu,T}\rangle_T=0$ and
$\|\varphi_{\mu,T}\|\leq K$ such that
 \begin{equation}\label{eq:res-projected-nonlinear-equation}
 (-L_{g_{\mathrm{cyl}}})(W_{\mu,T}+\varphi_{\mu,T})
 -(W_{\mu,T}+\varphi_{\mu,T})_+^{\frac{N+2}{N-2}}
 =\lambda_{\mu,T}(-L_{g_{\mathrm{cyl}}})Z_{\mu,T}.
\end{equation}
After pullback by $\Phi_T$, the pair depends smoothly on
$(\mu,T)\in(0,\mu_0)\times I$ with values in
$C^{2,\alpha}_{\mathrm{ax}}(\overline{\Sigma_{T^*}})\times\mathbb R$.
The same holds after the pullback to the fixed cylinder
$\varphi_{\mu,T}\mapsto\varphi_{\mu,T}(T\,\cdot,\cdot)$ to $\Sigma_1$.
Moreover,
\[
 \|\varphi_{\mu,T}\|+\|\partial_T\varphi_{\mu,T}\|\leq C,\qquad
 |\lambda_{\mu,T}|+|\partial_T\lambda_{\mu,T}|\leq C\mu^{N-2}.
\]
The nonlinear pairing satisfies
 \begin{equation}\label{eq:res-nonlinear-pairing}
 \Biggl\|T\longmapsto\int_{\Sigma_T}
 \bigl[(W_{\mu,T}+\varphi_{\mu,T})_+^{\frac{N+2}{N-2}}-W_{\mu,T}^{\frac{N+2}{N-2}}
       -\frac{N+2}{N-2}W_{\mu,T}^{\frac{4}{N-2}}\varphi_{\mu,T}\bigr]Z_{\mu,T}\dd V
 \Biggr\|_{C^1(I)}=o(\mu^{N-2}).
\end{equation}
\end{proposition}

\begin{proof}
\emph{Step 1. The fixed point argument.}
For $\zeta\in C^{2,\alpha}_{\mathrm{ax}}(\overline{\Sigma_T})$ that is even in $t$ and
has zero lateral trace, set
 \[
 \mathcal Q(\zeta)
 :=(W_{\mu,T}+\zeta)_+^{\frac{N+2}{N-2}}-W_{\mu,T}^{\frac{N+2}{N-2}}-\frac{N+2}{N-2}W_{\mu,T}^{\frac{4}{N-2}}\zeta.
 \]
For the data norm in Lemma \ref{lem:res-weighted-inverse}, write
\[
 \|h\|_{*,\mu,T}:=\sup_{X\in\Sigma_T}
 \frac{|h(X)|+r_\mu(X)^\alpha
 [h]_{C^\alpha(\mathcal B_{cr_\mu(X)}(X)\cap\Sigma_T)}}{\Theta_\mu(X)}.
\]
By Lemma \ref{lem:res-projection} and the definitions of the weights,
\[
 s_\mu:=\sup_{\substack{T\in I\\X\in\Sigma_T^\circ}}
 \frac{\delta(X)\Psi_\mu(r(X))}{W_{\mu,T}(X)}\longrightarrow0,
 \qquad W_{\mu,T}^{\frac{4}{N-2}}\delta\Psi_\mu\leq C\Theta_\mu.
\]
On each ball $\mathcal B_{cr_\mu(X)}(X)\cap\Sigma_T$, we have
\[
 \Big\|\frac{\zeta}{W_{\mu,T}}\Big\|_{L^\infty}
 +r_\mu(X)\Big\|\nabla\frac{\zeta}{W_{\mu,T}}\Big\|_{L^\infty}
 \leq Cs_\mu\|\zeta\|.
\]
Near the lateral boundary, this follows by division by $\delta$, since
\[
 \|\zeta/\delta\|_{C^1}\leq C\Psi_\mu(r_0)\|\zeta\|,
 \qquad \|\delta/W_{\mu,T}\|_{C^1}\leq C\mu^{-\frac{N-2}{2}}.
\]
Fix $K$ and take $\mu_0$ so small that $Ks_\mu\leq1/2$.
By Taylor's formula, for $\|\zeta\|\leq K$ we have
\[
 \mathcal Q(\zeta)=\frac{4(N+2)}{(N-2)^2}
 W_{\mu,T}^{\frac{N+2}{N-2}}(\zeta/W_{\mu,T})^2
 \int_0^1(1-\tau)(1+\tau\zeta/W_{\mu,T})^{\frac{6-N}{N-2}}\dd\tau.
\]
By the scaled product inequality, using also
$\Theta_\mu\geq C^{-1}\mu^{\frac{N+2}{2}}$ on $r\geq r_0/4$, we obtain
\[
 \|\mathcal Q(\zeta)\|_{*,\mu,T}\leq C_Ks_\mu,\qquad
 \|\mathcal Q(\zeta_1)-\mathcal Q(\zeta_2)\|_{*,\mu,T}
 \leq C_Ks_\mu\|\zeta_1-\zeta_2\|,
\]
for $\|\zeta\|,\|\zeta_1\|,\|\zeta_2\|\leq K$.
By Lemma \ref{lem:res-weighted-inverse}, define
$(\mathcal T_{\mu,T}(\zeta),\lambda_\zeta)$, with
$\langle\mathcal T_{\mu,T}(\zeta),Z_{\mu,T}\rangle_T=0$, by
\[
 \left(-L_{g_{\mathrm{cyl}}}
 -\frac{N+2}{N-2}W_{\mu,T}^{\frac{4}{N-2}}\right)
 \mathcal T_{\mu,T}(\zeta)
 =-R_{\mu,T}+\mathcal Q(\zeta)
 +\lambda_\zeta(-L_{g_{\mathrm{cyl}}})Z_{\mu,T}.
\]
By Lemma \ref{lem:res-residual-estimates}(i) and the preceding bounds,
\[
 \|\mathcal T_{\mu,T}(\zeta)\|+\mu^{2-N}|\lambda_\zeta|
 \leq C+C_Ks_\mu,\qquad
 \|\mathcal T_{\mu,T}(\zeta_1)-\mathcal T_{\mu,T}(\zeta_2)\|
 \leq C_Ks_\mu\|\zeta_1-\zeta_2\|.
\]
Choose first $K$ sufficiently large and then $\mu_0$ sufficiently small.
The map is a contraction on
$\{\|\zeta\|\leq K,\ \langle\zeta,Z_{\mu,T}\rangle_T=0\}$,
with image in $\{\|\zeta\|\leq K/2\}$.
By the contraction mapping principle, we obtain the unique pair
and the asserted bounds of order zero.  Moreover,
\begin{equation}\label{eq:res-relative-correction-small}
 \sup_{T\in I}\Big\|\frac{\varphi_{\mu,T}}{W_{\mu,T}}\Big\|_{L^\infty}
 \leq Ks_\mu\longrightarrow0.
\end{equation}

\emph{Step 2. Smooth parameter dependence and bounds for the first derivatives.}
We first work on $\Sigma_{T^*}$ using the localized pullback $\Phi_T$.
For $\mu>0$, the source $\chi\mathcal U_\mu^{\frac{N+2}{N-2}}$ depends smoothly on
$\mu$ in every H\"older space and is independent of $T$.
The coefficients of $-L_{g_T}$ depend smoothly on $T$, and the
Dirichlet inverse depends smoothly on $T$ as a bounded operator from
$C^\alpha$ to $C^{2,\alpha}\cap H_0^1$.
Thus $W_{\mu,T}$ and $Z_{\mu,T}=\mu\partial_\mu W_{\mu,T}$ depend
smoothly on $(\mu,T)$ in $C^{2,\alpha}_{\mathrm{ax}}$.

By Step 1 and Lemma \ref{lem:res-projection}(i),
$(W_{\mu,T}+\varphi_{\mu,T})/z_N$ extends to a strictly positive
function on $\overline{\Sigma_{T^*}}$.
The boundary division argument of Lemma \ref{lem:power-map} applies
on this cylinder as well.  Hence, near each constructed solution,
the map $u\mapsto u_+^{\frac{N+2}{N-2}}$ agrees with the smooth map
$u\mapsto u^{\frac{N+2}{N-2}}$ from the subspace of $C^{2,\alpha}$ with zero trace
to $C^\alpha$.
Together with the smooth dependence of the energy pairing under pullback,
this shows that, for each fixed $(\mu,T)$, the projected equation and
orthogonality condition define a smooth map
\[
 \bigl(C^{2,\alpha}_{\mathrm{ax}}(\overline{\Sigma_{T^*}})
       \cap H^1_{0,\mathrm{ax}}(\Sigma_{T^*})\bigr)\times\mathbb R
 \longrightarrow C^\alpha_{\mathrm{ax}}(\overline{\Sigma_{T^*}})\times\mathbb R
\]
on the subspaces of functions that are even in $t$,
and that this map depends smoothly on $(\mu,T)$.
At $(\varphi_{\mu,T},\lambda_{\mu,T})$, the linearization is a
perturbation of the map in Lemma \ref{lem:res-inverse}.
By Step 1, for every $\eta\in C^{2,\alpha}_{\mathrm{ax}}$ that is even in $t$ and
 has zero lateral trace, we have
 \[
 \bigl\|\frac{N+2}{N-2}\bigl[(W_{\mu,T}+\varphi_{\mu,T})_+^{\frac{4}{N-2}}
       -W_{\mu,T}^{\frac{4}{N-2}}\bigr]\eta
 \bigr\|_{*,\mu,T}\leq C_Ks_\mu\|\eta\|.
\]
By Lemma \ref{lem:res-weighted-inverse} with $m=0$, we invert the
linearization by a Neumann series.  The weighted and H\"older norms
are equivalent at fixed $(\mu,T)$; hence, by the smooth implicit
function theorem and uniqueness in Step 1, the solution pair depends
smoothly on $(\mu,T)$ in the stated H\"older spaces.

We verify the same regularity directly on $\Sigma_1$.
Under $t=Ts$, the operator becomes $-L_{\widehat g_T}$, with
$\widehat g_T=T^2\dd s^2+g_{\mathbb S^{N-1}}$ as in Section 2.
On $-1/2<s<1/2$, the source is $(\chi\mathcal U_\mu^{\frac{N+2}{N-2}})(Ts,z)$.
Its support stays away from the endpoints, so its periodic extension
is smooth jointly in $(\mu,T,s,z)$ for $\mu>0$.
Applying $(-L_{\widehat g_T})^{-1}$, we obtain smooth parameter
dependence of $W_{\mu,T}(T\,\cdot,\cdot)$ and
$Z_{\mu,T}(T\,\cdot,\cdot)$ in
$C^{2,\alpha}_{\mathrm{ax}}(\overline{\Sigma_1})$.
The energy pairing for $\widehat g_T$ also depends smoothly on $T$.
For each fixed $T$, the change of coordinates is a bounded
isomorphism of the relevant H\"older spaces, so the transformed
augmented linearization is invertible.
By Lemma \ref{lem:power-map} and the smooth implicit function theorem,
the solution pair after this pullback therefore depends
smoothly on $(\mu,T)$ as well.
These arguments are local for $\mu>0$ and require no uniform bounds
for higher parameter derivatives as $\mu\downarrow0$.
It remains to establish uniform estimates for the first
derivatives.  For this purpose, return to the localized pullback and its
convention for $\partial_T$.
By Lemma \ref{lem:res-projection} and the estimates in Step 1,
\[
 \|\partial_T\{\mathcal Q(\varphi_{\mu,T})\}\|_{*,\mu,T}
 \leq C_Ks_\mu\bigl(1+\|\partial_T\varphi_{\mu,T}\|\bigr).
\]
Apply Lemma \ref{lem:res-weighted-inverse} with $m=1$,
$h=-R_{\mu,T}+\mathcal Q(\varphi_{\mu,T})$, and $\gamma=0$.
By Lemma \ref{lem:res-residual-estimates}(i),
\[
 \sup_{T\in I}\bigl(\|\partial_T\varphi_{\mu,T}\|
       +\mu^{2-N}|\partial_T\lambda_{\mu,T}|\bigr)
 \leq C+C_Ks_\mu\sup_{T\in I}\|\partial_T\varphi_{\mu,T}\|.
\]
Decreasing $\mu_0$ and absorbing the last term, we obtain the
asserted bounds for the derivatives with respect to $T$.

\emph{Step 3. The nonlinear contribution to the reduced equation.}
By Taylor's formula, \eqref{eq:res-relative-correction-small}, and
Lemma \ref{lem:res-projection}, for every $\frac{N+2}{N-2}>1$ we have
\[
 \sum_{j=0}^1\bigl|\partial_T^j
   \{\mathcal Q(\varphi_{\mu,T})Z_{\mu,T}\}\bigr|
 \leq CW_{\mu,T}^{\frac{4}{N-2}}|\varphi_{\mu,T}|
 \bigl(|\varphi_{\mu,T}|+|\partial_T\varphi_{\mu,T}|\bigr).
\]
The bounds $|\partial_TJ_T|/J_T\leq C$ and
$\|W_{\mu,T}^{\frac{4}{N-2}}\|_{L^{\frac{N}{2}}(\Sigma_T)}\leq C$, together with
Steps 1--2, \eqref{eq:res-rho}, H\"older's inequality, and the Sobolev
inequality, give
\begin{equation*}
 \Big\|T\mapsto\int_{\Sigma_T}
 \mathcal Q(\varphi_{\mu,T})Z_{\mu,T}\dd V\Big\|_{C^1(I)}
 \leq C\sup_T\|\varphi_{\mu,T}\|_T
 \bigl(\|\varphi_{\mu,T}\|_T+\|\partial_T\varphi_{\mu,T}\|_T\bigr)
 \leq C\rho_N(\mu)^2=o(\mu^{N-2}).\qedhere
\end{equation*}
\end{proof}

\subsection{Reduction to a scalar equation}

By Proposition \ref{prop:res-correction}, the construction of solutions
to \eqref{eq:t-cylinder} reduces to the scalar condition
$\lambda_{\mu,T}=0$.  In this subsection we solve this equation for the
period near $T^*$ and obtain a smooth family of positive solutions that
blow up as $\mu\downarrow0$. 

\begin{theorem}\label{thm:res-family}

There exist $\mu_0>0$ and a family $(T(\mu),v_\mu)$, $0<\mu<\mu_0$,
with $T(\mu)\in I$ and $v_\mu$ a positive solution of
\eqref{eq:t-cylinder} on $\Sigma_{T(\mu)}$, such that
\[
 \mu\longmapsto\bigl(T(\mu),v_\mu(T(\mu)\,\cdot,\cdot)\bigr)
\]
is smooth with values in
$\mathbb R\times C^{2,\alpha}_{\mathrm{ax}}(\overline{\Sigma_1})$, and
\[
 T(\mu)\longrightarrow T^*,\qquad
 \|v_\mu\|_{L^\infty}\longrightarrow\infty
 \quad\text{as }\mu\downarrow0.
\]

\end{theorem}

\begin{proof}
Define
\begin{equation}\label{eq:res-reduced-map}
 \mathcal F(\mu,T)=
 \int_{\Sigma_T}\{(-L_{g_{\mathrm{cyl}}})(W_{\mu,T}+\varphi_{\mu,T})
 -(W_{\mu,T}+\varphi_{\mu,T})_+^{\frac{N+2}{N-2}}\}Z_{\mu,T}\dd V.
\end{equation}
By \eqref{eq:res-projected-nonlinear-equation}, we have
\begin{equation}\label{eq:res-reduced-multiplier}
 \mathcal F(\mu,T)=\lambda_{\mu,T}\|Z_{\mu,T}\|_T^2.
\end{equation}
Note that
 \[
 \begin{aligned}
 &(-L_{g_{\mathrm{cyl}}})(W_{\mu,T}+\varphi_{\mu,T})
 -(W_{\mu,T}+\varphi_{\mu,T})_+^{\frac{N+2}{N-2}}\\
 =&\,R_{\mu,T}
 +\left(-L_{g_{\mathrm{cyl}}}
 -\frac{N+2}{N-2}W_{\mu,T}^{\frac{4}{N-2}}\right)\varphi_{\mu,T}
 -\biggl((W_{\mu,T}+\varphi_{\mu,T})_+^{\frac{N+2}{N-2}}
 -W_{\mu,T}^{\frac{N+2}{N-2}}
 -\frac{N+2}{N-2}W_{\mu,T}^{\frac{4}{N-2}}\varphi_{\mu,T}\biggr).
 \end{aligned}
\]
By the self-adjointness of $-L_{g_{\mathrm{cyl}}}$ and
\eqref{eq:res-scale-residual-identity}, we have
\[
 \begin{aligned}
 \int_{\Sigma_T}
 \left(-L_{g_{\mathrm{cyl}}}
 -\frac{N+2}{N-2}W_{\mu,T}^{\frac{4}{N-2}}\right)
 \varphi_{\mu,T}Z_{\mu,T}\dd V
 &=\int_{\Sigma_T}\varphi_{\mu,T}
 \left(-L_{g_{\mathrm{cyl}}}
 -\frac{N+2}{N-2}W_{\mu,T}^{\frac{4}{N-2}}\right)
 Z_{\mu,T}\dd V\\
 &=\int_{\Sigma_T}\varphi_{\mu,T}
 \mu\partial_\mu R_{\mu,T}\dd V.
 \end{aligned}
\]
Consequently,
\begin{equation*}
 \begin{aligned}
 \mathcal F(\mu,T)
 &=\int_{\Sigma_T}R_{\mu,T}Z_{\mu,T}\dd V
   +\int_{\Sigma_T}\varphi_{\mu,T}\mu\partial_\mu R_{\mu,T}\dd V\\
 &\quad-\int_{\Sigma_T}
 \bigl((W_{\mu,T}+\varphi_{\mu,T})_+^{\frac{N+2}{N-2}}-W_{\mu,T}^{\frac{N+2}{N-2}}
 -\frac{N+2}{N-2}W_{\mu,T}^{\frac{4}{N-2}}\varphi_{\mu,T}\bigr)Z_{\mu,T}\dd V.
 \end{aligned}
\end{equation*}
By \eqref{eq:fixed-domain-interval}--\eqref{eq:fixed-domain-metric},
Proposition \ref{prop:res-correction}, and
Lemma \ref{lem:res-residual-estimates}(ii), we obtain
\begin{align*}
 \sum_{j=0}^1\Biggl|\partial_T^j\int_{\Sigma_T}\varphi_{\mu,T}\mu\partial_\mu R_{\mu,T}\dd V\Biggr|
 \leq C\sum_{j=0}^1\|\partial_T^j\varphi_{\mu,T}\|_T
          \sum_{j=0}^1
          \|\partial_T^j(\mu\partial_\mu R_{\mu,T})\|_{H^{-1}(\Sigma_T)}
 \leq C\rho_N(\mu)^2.
\end{align*}
Combining this estimate with Lemma \ref{lem:res-residual-estimates} and
Proposition \ref{prop:res-correction}, we obtain
\begin{equation}\label{eq:res-reduced-expansion}
 \Bigl\|T\longmapsto
 \mu^{2-N}\mathcal F(\mu,T)
 +\frac{(N-2)\alpha_N^2}{2c_N}\mathscr M_N(T)
 \Bigr\|_{C^1(I)}\longrightarrow0.
\end{equation}
By the derivative calculation in
Proposition \ref{prop:unique-regular-zero}, we have
\begin{equation}\label{eq:res-regular-transversality}
 \mathscr M_N'(T)<0\quad\text{if }T>0,
 \qquad \mathscr M_N'(T^*)\ne0.
\end{equation}
Choose $\eta\in(0,\delta_T)$ so that
$-\mathscr M_N'(T)\geq c>0$ for $|T-T^*|\leq\eta$.  Since
$\mathscr M_N(T^*)=0$,
\[
 \mathscr M_N(T^*-\eta)>0>
 \mathscr M_N(T^*+\eta).
\]
By \eqref{eq:res-reduced-expansion}, for all sufficiently small $\mu$,
\begin{equation}\label{eq:res-reduced-monotonicity}
 \mathcal F(\mu,T^*-\eta)<0<\mathcal F(\mu,T^*+\eta),\qquad
 \partial_T\mathcal F(\mu,T)>0\quad\text{for }|T-T^*|\leq\eta.
\end{equation}
Thus $\mathcal F(\mu,\cdot)$ has a unique zero
$T(\mu)\in(T^*-\eta,T^*+\eta)$.
By \eqref{eq:res-reduced-expansion}, $\mathscr M_N(T(\mu))\to0$, and hence
\[
 T(\mu)\longrightarrow T^*\quad \text{as }\mu\downarrow0.
\]
By \eqref{eq:res-reduced-multiplier},
Proposition \ref{prop:res-correction}, and the smooth parameter dependence
of $Z_{\mu,T}$, the function $\mathcal F$ is smooth in $(\mu,T)$ for $\mu>0$.
Since $\partial_T\mathcal F(\mu,T(\mu))>0$ by
\eqref{eq:res-reduced-monotonicity}, we obtain
$T(\cdot)\in C^\infty((0,\mu_0))$ by the smooth implicit function theorem.
At $T=T(\mu)$, it follows from
\eqref{eq:res-reduced-multiplier} and
\eqref{eq:res-scale-mode-norms} that
\[
 0=\lambda_{\mu,T(\mu)}\|Z_{\mu,T(\mu)}\|_{T(\mu)}^2,
 \qquad
 \|Z_{\mu,T(\mu)}\|_{T(\mu)}^2\longrightarrow\kappa_N>0.
\]
Hence $\lambda_{\mu,T(\mu)}=0$.  Set
\[
 v_\mu:=W_{\mu,T(\mu)}+\varphi_{\mu,T(\mu)},
 \]
so that $-L_{g_{\mathrm{cyl}}}v_\mu=(v_\mu)_+^{\frac{N+2}{N-2}}$.  
For all sufficiently small $\mu$,
\eqref{eq:res-relative-correction-small} gives
\[
 v_\mu=W_{\mu,T(\mu)}
 \biggl(1+\frac{\varphi_{\mu,T(\mu)}}{W_{\mu,T(\mu)}}\biggr)>0,
 \qquad
 \biggl\|\frac{\varphi_{\mu,T(\mu)}}{W_{\mu,T(\mu)}}\biggr\|_{L^\infty}
 =o(1).
\]
Hence $v_\mu$ solves \eqref{eq:t-cylinder}.  
Finally, by Proposition \ref{prop:res-correction} and the smooth
parameter dependence of $W_{\mu,T}$ and $T(\mu)$, the family
$\mu\mapsto v_\mu$ is smooth after either identification.  In particular,
\[
 (0,\mu_0)\ni\mu\longmapsto v_\mu(T(\mu)\,\cdot,\cdot)
\]
is smooth with values in
$C^{2,\alpha}_{\mathrm{ax}}(\overline{\Sigma_1})$.
Together with the smoothness of $T(\mu)$, this proves the regularity
asserted in Theorem \ref{thm:res-family}.
In addition, by Lemma \ref{lem:res-projection}(i) and
\eqref{eq:res-relative-correction-small}, we have
\[
 \|v_\mu\|_{L^\infty}\geq v_\mu(P)
 \geq\tfrac12W_{\mu,T(\mu)}(P)
 \geq c\mu^{-\frac{N-2}{2}}\longrightarrow\infty.\qedhere
\]
\end{proof}

\section{Local uniqueness and connection to the global continuum}
\label{sec:local-uniqueness}

We prove local uniqueness near the concentrating limit and identify the constructed branch with part of the global continuum $\mathscr C$.

\begin{lemma}\label{lem:res-entry-modulation}
Let $T_n\to T^*$, and let $v_n$ be positive bumpy solutions on
$\Sigma_{T_n}$, centered at $P$, with $v_n(P)\to\infty$.
For large $n$, define $\widetilde\mu_n$ by
$\mathcal U_{\widetilde\mu_n}(P)=v_n(P)$.
Then there exists $\sigma_0>0$ such that, for all sufficiently large $n$,
the equation
\begin{equation}\label{eq:res-modulation}
 \langle v_n-W_{\mu,T_n},Z_{\mu,T_n}\rangle_{T_n}=0
\end{equation}
has a unique solution
$\mu_n\in[e^{-\sigma_0}\widetilde\mu_n,e^{\sigma_0}\widetilde\mu_n]$.
Moreover,
\[
 \frac{\mu_n}{\widetilde\mu_n}\longrightarrow1,\qquad
 \|v_n-W_{\mu_n,T_n}\|_{T_n}\longrightarrow0.
\]
\end{lemma}

\begin{proof}
\emph{Step 1. Energy approximation.}
We first compare $v_n$ with $W_{\widetilde\mu_n,T_n}$ in the energy norm.
By \eqref{eq:res-cylinder-bubble},
\[
 \mathcal U_\mu(P)=\alpha_N\biggl(\frac{1+\sqrt{1-\mu^2}}{2\mu}\biggr)^{\frac{N-2}{2}}.
\]
This expression tends monotonically to $\infty$ as $\mu\downarrow0$.
Hence $\widetilde\mu_n$ is
well defined, tends to zero, and satisfies
\[
 \widetilde\mu_n^{\frac{N-2}{2}}v_n(P)\longrightarrow\alpha_N,
 \qquad
 \frac{v_n(P)^{-2/(N-2)}}{\widetilde\mu_n}
 \longrightarrow\frac{1}{\sqrt{N(N-2)}}=\frac1{\lambda_*}.
\]
Put $q_0=2N/(N+2)$, the H\"older conjugate of $2N/(N-2)$.  
By \eqref{eq:t-cylinder} and \eqref{eq:res-projected-bubble}, we have
\[
 -L_{g_{\mathrm{cyl}}}v_n=v_n^{\frac{N+2}{N-2}},\qquad
 -L_{g_{\mathrm{cyl}}}W_{\widetilde\mu_n,T_n}
 =\chi\mathcal U_{\widetilde\mu_n}^{\frac{N+2}{N-2}}
 \quad\text{in }\Sigma_{T_n}.
\]
Since $T_n\to T^*>0$, the Sobolev inequality holds with a constant
$C>0$ independent of $n$. Subtracting these equations and testing with
$v_n-W_{\widetilde\mu_n,T_n}$, we obtain, by H\"older's inequality and
the Sobolev inequality,
\[
 \begin{aligned}
 \|v_n-W_{\widetilde\mu_n,T_n}\|_{T_n}^2
 &=\int_{\Sigma_{T_n}}(v_n^{\frac{N+2}{N-2}}-\chi\mathcal U_{\widetilde\mu_n}^{\frac{N+2}{N-2}})
 (v_n-W_{\widetilde\mu_n,T_n})\dd V\\
 &\leq\|v_n^{\frac{N+2}{N-2}}-\chi\mathcal U_{\widetilde\mu_n}^{\frac{N+2}{N-2}}\|_{L^{q_0}}
 \|v_n-W_{\widetilde\mu_n,T_n}\|_{L^{\frac{2N}{N-2}}}\\
 &\leq C\|v_n^{\frac{N+2}{N-2}}-\chi\mathcal U_{\widetilde\mu_n}^{\frac{N+2}{N-2}}\|_{L^{q_0}}
 \|v_n-W_{\widetilde\mu_n,T_n}\|_{T_n}.
 \end{aligned}
\]
Consequently,
\begin{equation}\label{eq:res-entry-H1}
 \|v_n-W_{\widetilde\mu_n,T_n}\|_{T_n}
 \leq C\|v_n^{\frac{N+2}{N-2}}-\chi\mathcal U_{\widetilde\mu_n}^{\frac{N+2}{N-2}}\|_{L^{q_0}(\Sigma_{T_n})}.
\end{equation}
It remains to show that the expression on the right tends to zero.
For fixed sufficiently large $R>1$ and all sufficiently large $n$,
split $\Sigma_{T_n}$ into
\[
 \mathcal B_{R\widetilde\mu_n}(P)
 \quad\text{and}\quad
 \Sigma_{T_n}\setminus\mathcal B_{R\widetilde\mu_n}(P).
\]

\emph{Core region.}
Set $X=\exp_P(\widetilde\mu_n y)$.  On $B_R$, we have
\[
 \dd V=\widetilde\mu_n^NJ_n(y)\dd y,\qquad
 J_n(y)=\sqrt{\det g_{ij}(\widetilde\mu_n y)}
 =1+O(\widetilde\mu_n^2|y|^2).
\]
By the scale relation above, using
Lemma \ref{lem:single-bubble-estimates}(A) and
\eqref{eq:res-cylinder-bubble}, respectively, we obtain
\[
 \widetilde\mu_n^{\frac{N-2}{2}}v_n(\exp_P(\widetilde\mu_n\,\cdot))
 \longrightarrow U_{\lambda_*},
 \qquad
 \widetilde\mu_n^{\frac{N-2}{2}}\mathcal U_{\widetilde\mu_n}
     (\exp_P(\widetilde\mu_n\,\cdot))
 \longrightarrow U_{\lambda_*}
 \quad\text{in }C^0(B_R).
\]
Since $\chi=1$ on $\mathcal B_{R\widetilde\mu_n}(P)$ for large $n$, we have
\[
 \begin{aligned}
 \int_{\mathcal B_{R\widetilde\mu_n}(P)}
 |v_n^{\frac{N+2}{N-2}}-\chi\mathcal U_{\widetilde\mu_n}^{\frac{N+2}{N-2}}|^{q_0}\dd V
 &=\int_{B_R}
 \bigl|\widetilde\mu_n^{\frac{N+2}{2}}
 (v_n^{\frac{N+2}{N-2}}-\mathcal U_{\widetilde\mu_n}^{\frac{N+2}{N-2}})
 (\exp_P(\widetilde\mu_n y))\bigr|^{q_0}
 J_n(y)\dd y
 \longrightarrow0.
 \end{aligned}
\]

\emph{Complementary region.}
By Lemmas \ref{lem:local-periodic-green}
and \ref{lem:single-bubble-estimates}(C), together with
the scale relation above and \eqref{eq:res-cylinder-bubble}, we have
\[
 \begin{aligned}
 \int_{\Sigma_{T_n}\setminus\mathcal B_{R\widetilde\mu_n}(P)}
 |v_n^{\frac{N+2}{N-2}}-\chi\mathcal U_{\widetilde\mu_n}^{\frac{N+2}{N-2}}|^{q_0}\dd V
 &\leq C\int_{\Sigma_{T_n}\setminus
                   \mathcal B_{R\widetilde\mu_n}(P)}
 \bigl(v_n^{\frac{2N}{N-2}}+\chi^{q_0}\mathcal U_{\widetilde\mu_n}^{\frac{2N}{N-2}}\bigr)\dd V
 \leq CR^{-N}+C\widetilde\mu_n^N.
 \end{aligned}
\]
Combining the two regions and letting first $n\to\infty$ and then
$R\to\infty$, we obtain
\[
 \|v_n^{\frac{N+2}{N-2}}-\chi\mathcal U_{\widetilde\mu_n}^{\frac{N+2}{N-2}}\|_{L^{q_0}(\Sigma_{T_n})}
 \longrightarrow0.
\]
Estimate \eqref{eq:res-entry-H1} now gives
$\|v_n-W_{\widetilde\mu_n,T_n}\|_{T_n}\to0$.
Since this argument applies to a further subsequence of every subsequence,
the convergence holds for the full sequence.

\emph{Step 2. Choice of the scale.}
We now impose orthogonality.
Set $\mu_{n,\sigma}=e^\sigma\widetilde\mu_n$ for $|\sigma|\leq1$.
Define
\[
 \mathcal F_n(\sigma)=
 \langle v_n-W_{\mu_{n,\sigma},T_n},
 Z_{\mu_{n,\sigma},T_n}\rangle_{T_n}.
\]
By Lemma \ref{lem:res-scale-mode}, uniformly for $|\sigma|\leq1$, we have
$\|Z_{\mu_{n,\sigma},T_n}\|_{T_n}^2=\kappa_N+o(1)$.
As $\partial_\sigma\!W_{\mu_{n,\sigma},T_n} =Z_{\mu_{n,\sigma},T_n}$, we have
\begin{equation*}
\begin{aligned}
 \|v_n-W_{\mu_{n,\sigma},T_n}\|_{T_n}
 &\leq\|v_n-W_{\widetilde\mu_n,T_n}\|_{T_n}+ \|W_{\mu_{n,\sigma},T_n}-W_{\widetilde\mu_n,T_n}\|_{T_n} \\
 &\leq \|v_n-W_{\widetilde\mu_n,T_n}\|_{T_n}
 +\Bigl|\int_{0}^{\sigma}\|Z_{\mu_{n,\tau},T_n}\|_{T_n}\,\dd\tau\,\Bigr|
 = o(1) +C|\sigma|.
\end{aligned}
\end{equation*}
In particular, $\mathcal F_n(0)=o(1)$.
To estimate $\mathcal F_n'$, we first bound
$\partial_\sigma Z_{\mu_{n,\sigma},T_n}
=(\mu\partial_\mu)^2W_{\mu,T_n}|_{\mu=\mu_{n,\sigma}}$.
By \eqref{eq:res-projected-bubble},
Lemma \ref{lem:res-bubble-estimates}(i), and
Lemma \ref{lem:res-projection}(iii),
\begin{equation*}
\begin{aligned}
 \|(\mu\partial_\mu)^2W_{\mu,T}\|_T^2
 =\int_{\Sigma_T}\chi(\mu\partial_\mu)^2\mathcal U_\mu^{\frac{N+2}{N-2}}
       (\mu\partial_\mu)^2W_{\mu,T}\dd V
 \leq C\int_{\Sigma_T}\chi\mathcal U_\mu^{\frac{N+2}{N-2}}W_{\mu,T}\dd V.
\end{aligned}
\end{equation*}
By \eqref{eq:res-Green-expansion} on $\mathcal B_{r_0}(P)$ and Lemma \ref{lem:res-projection}(i)
on $\operatorname{supp}\chi\setminus\mathcal B_{r_0}(P)$, we have
$|W_{\mu,T}-\mathcal U_\mu|\leq C\mu^{\frac{N-2}{2}}$ on $\operatorname{supp}\chi$.
Hence, by \eqref{eq:res-cylinder-bubble} and
Lemma \ref{lem:res-bubble-estimates}(iv),
\[
 \int_{\Sigma_T}\chi\mathcal U_\mu^{\frac{N+2}{N-2}}W_{\mu,T}\dd V
 \leq\int_{\Sigma_T}\chi\mathcal U_\mu^{\frac{2N}{N-2}}\dd V
 +C\mu^{\frac{N-2}{2}}\int_{\Sigma_T}\chi\mathcal U_\mu^{\frac{N+2}{N-2}}\dd V
 \leq C+C\mu^{N-2}\leq C.
\]
Thus $\|\partial_\sigma Z_{\mu_{n,\sigma},T_n}\|_{T_n}\leq C$.
By Cauchy--Schwarz and the preceding bounds, uniformly for $|\sigma|\leq1$,
\[
 \begin{aligned}
 \mathcal F_n'(\sigma)
 &=-\|Z_{\mu_{n,\sigma},T_n}\|_{T_n}^2
 +\langle v_n-W_{\mu_{n,\sigma},T_n},
       \partial_\sigma Z_{\mu_{n,\sigma},T_n}\rangle_{T_n}\\
 &=-\kappa_N+o(1)+O(|\sigma|).
 \end{aligned}
\]
Choose $\sigma_0\in(0,1)$ sufficiently small
and then $n$ sufficiently large so that $\mathcal F_n'\leq-\kappa_N/2$ on
$[-\sigma_0,\sigma_0]$.
Since $\mathcal F_n(0)=o(1)$, we have
$\mathcal F_n(-\sigma_0)>0>\mathcal F_n(\sigma_0)$.
Thus $\mathcal F_n$ has a unique zero
$\sigma_n\in(-\sigma_0,\sigma_0)$, and by the mean value theorem, we have
\[
 |\sigma_n|\leq\frac{2}{\kappa_N}|\mathcal F_n(0)|\longrightarrow0.
\]
Taking $\mu_n=e^{\sigma_n}\widetilde\mu_n$, we obtain the required
orthogonality and
\[
 \frac{\mu_n}{\widetilde\mu_n}=e^{\sigma_n}\longrightarrow1,\qquad
 \|v_n-W_{\mu_n,T_n}\|_{T_n}\leq o(1)+C|\sigma_n|
 \longrightarrow0.\qedhere
\]
\end{proof}

We now use the modulation parameter to show that the
constructed family contains every centered bumpy solution sufficiently
close to the blow-up limit.
\begin{proposition}\label{prop:res-local-exhaustiveness}
Let $(T(\mu),v_\mu)$ be the family constructed in
Theorem \ref{thm:res-family}.
There exist $\varepsilon,M,\bar\mu>0$ with the following property.
Suppose $|T-T^*|<\varepsilon$ and $v$ is a positive bumpy solution on
$\Sigma_T$ that is centered at $P$ and satisfies $v(P)>M$.
Then $T=T(\mu)$ and $v=v_\mu$ for a unique $\mu\in(0,\bar\mu)$.
\end{proposition}

\begin{proof}
We first show that every sequence in the stated class satisfying
$T_n\to T^*$ and $v_n(P)\to\infty$ eventually belongs to the constructed
family. We then prove uniqueness of the representation for small scales and
deduce uniform choices of $\varepsilon$ and $M$.
Let $T_n\to T^*$, and let $v_n$ be positive bumpy solutions on
$\Sigma_{T_n}$, centered at $P$, with $v_n(P)\to\infty$.
Apply Lemma \ref{lem:res-entry-modulation} to choose
$\mu_n\to0$ and set $\psi_n:=v_n-W_{\mu_n,T_n}$.  Then
\[
 v_n=W_{\mu_n,T_n}+\psi_n,\qquad
 \langle\psi_n,Z_{\mu_n,T_n}\rangle_{T_n}=0,\qquad
 \|\psi_n\|_{T_n}=o(1).
\]
Let $\varphi_{\mu_n,T_n}$ and $\lambda_{\mu_n,T_n}$ be the
correction and multiplier associated with $W_{\mu_n,T_n}$ in
Proposition \ref{prop:res-correction}.  Then
\[
 \|\varphi_{\mu_n,T_n}\|_{T_n}
 \leq \|\varphi_{\mu_n,T_n}\|\,\rho_N(\mu_n)
 \leq C\rho_N(\mu_n)=o(1).
\]
To compare the two corrections, set
$\eta_n=\psi_n-\varphi_{\mu_n,T_n}$ and define
 \[
 \mathcal Q_n(\zeta)
 =(W_{\mu_n,T_n}+\zeta)_+^{\frac{N+2}{N-2}}-W_{\mu_n,T_n}^{\frac{N+2}{N-2}}
  -\frac{N+2}{N-2}W_{\mu_n,T_n}^{\frac{4}{N-2}}\zeta.
\]
Then $\langle\eta_n,Z_{\mu_n,T_n}\rangle_{T_n}=0$. Subtracting
\eqref{eq:res-projected-nonlinear-equation} from \eqref{eq:t-cylinder}
for $v_n$, we obtain
\[
 \left(-L_{g_{\mathrm{cyl}}}
 -\frac{N+2}{N-2}W_{\mu_n,T_n}^{\frac{4}{N-2}}\right)\eta_n
 =\mathcal Q_n(\psi_n)-\mathcal Q_n(\varphi_{\mu_n,T_n})
  -\lambda_{\mu_n,T_n}(-L_{g_{\mathrm{cyl}}})Z_{\mu_n,T_n}.
\]
 For $\zeta\in H_0^1(\Sigma_{T_n})$, it follows from H\"older's
 inequality that
 \[
 \frac{N+2}{N-2}\|(W_{\mu_n,T_n}+\zeta)_+^{\frac{4}{N-2}}-W_{\mu_n,T_n}^{\frac{4}{N-2}}\|_{L^{\frac{N}{2}}}
 \leq C\begin{cases}
 \|W_{\mu_n,T_n}\|_{L^{\frac{2N}{N-2}}}^{\frac{6-N}{N-2}}\|\zeta\|_{L^{\frac{2N}{N-2}}}
 +\|\zeta\|_{L^{\frac{2N}{N-2}}}^{\frac{4}{N-2}},&\frac{N+2}{N-2}\geq2,\\
 \|\zeta\|_{L^{\frac{2N}{N-2}}}^{\frac{4}{N-2}},&1<\frac{N+2}{N-2}<2.
 \end{cases}
\]
The identity
$\varphi_{\mu_n,T_n}+\tau\eta_n
=(1-\tau)\varphi_{\mu_n,T_n}+\tau\psi_n$ and the Sobolev inequality
give
\[
 \sup_{0\leq \tau\leq1}
 \|\varphi_{\mu_n,T_n}+\tau\eta_n\|_{L^{\frac{2N}{N-2}}}
 \leq C\bigl(\|\psi_n\|_{T_n}+\|\varphi_{\mu_n,T_n}\|_{T_n}\bigr)=o(1).
\]
By Lemma \ref{lem:res-projection}(i), there exists $C>0$,
independent of $n$, such that $\|W_{\mu_n,T_n}\|_{L^{\frac{2N}{N-2}}}\leq C$.
Set
$D_n:=\mathcal Q_n(\psi_n)-\mathcal Q_n(\varphi_{\mu_n,T_n})$.
The mean value formula, H\"older's inequality, and the Sobolev inequality
give, for large $n$,
\[
 \begin{aligned}
 \|D_n\|_{H^{-1}(\Sigma_{T_n})}
 &\leq C\|D_n\|_{L^{2N/(N+2)}}\\
 &\leq C\|\eta_n\|_{L^{\frac{2N}{N-2}}}\int_0^1
 \|(W_{\mu_n,T_n}+\varphi_{\mu_n,T_n}+\tau\eta_n)_+^{\frac{4}{N-2}}
       -W_{\mu_n,T_n}^{\frac{4}{N-2}}\|_{L^{\frac{N}{2}}}\dd\tau\\
 &\leq C\bigl(\|\psi_n\|_{T_n}+\|\varphi_{\mu_n,T_n}\|_{T_n}\bigr)^{\min\{1,\frac{4}{N-2}\}}
 \|\eta_n\|_{T_n}.
 \end{aligned}
\]
Applying Lemma \ref{lem:res-inverse} to the equation for $\eta_n$ with multiplier $-\lambda_{\mu_n,T_n}$,
we conclude that
\[
 \|\eta_n\|_{T_n}+|\lambda_{\mu_n,T_n}|
 \leq C\|D_n\|_{H^{-1}(\Sigma_{T_n})}
 \leq o(1)\|\eta_n\|_{T_n},
\]
which implies $\eta_n=0$ and $\lambda_{\mu_n,T_n}=0$.
By \eqref{eq:res-reduced-multiplier} and
\eqref{eq:res-reduced-monotonicity}, we conclude that
\[
 T_n=T(\mu_n),\qquad v_n=v_{\mu_n}.
\]

We next prove that, for some $\bar\mu>0$, every centered positive
bumpy solution has at most one representation with $0<\mu<\bar\mu$.
Otherwise, we could find solutions $v_n$ in this class on $\Sigma_{T_n}$
and distinct scales $\mu_n,\mu_n'\to0$ such that
\[
 (T_n,v_n)=(T(\mu_n),v_{\mu_n})=(T(\mu_n'),v_{\mu_n'}).
\]
Then $T_n\to T^*$, and by \eqref{eq:res-Green-expansion} and
\eqref{eq:res-relative-correction-small}, we have
\[
 v_n(P)=\mathcal U_{\mu_n}(P)(1+o(1))
 =\mathcal U_{\mu_n'}(P)(1+o(1))\longrightarrow\infty.
\]
Let $\widetilde\mu_n$ be defined by
$\mathcal U_{\widetilde\mu_n}(P)=v_n(P)$.  Since
$\mathcal U_\mu(P)\sim\alpha_N\mu^{-\frac{N-2}{2}}$, we have
\[
 \frac{\mu_n}{\widetilde\mu_n}\longrightarrow1,\qquad
 \frac{\mu_n'}{\widetilde\mu_n}\longrightarrow1.
\]
Thus both scales lie in the interval of
Lemma \ref{lem:res-entry-modulation} for large $n$.  Both representations
satisfy the orthogonality condition by construction, so by uniqueness
in that interval, we have $\mu_n=\mu_n'$, a contradiction.

Fix such a $\bar\mu$. If the existence assertion failed for every
$\varepsilon,M>0$,
we could find solutions $v_n$ in this class with $T_n\to T^*$ and
$v_n(P)\to\infty$ but no representation with $0<\mu<\bar\mu$.
This contradicts the representations with $\mu_n\to0$ obtained above.
\end{proof}

It is time to prove Theorem \ref{thm:main-resonant-family}.
\begin{proof}[Proof of Theorem \ref{thm:main-resonant-family}]
Take the family $(T(\mu),v_\mu)$ constructed in
Theorem \ref{thm:res-family}, and decrease $\mu_0$ as needed below.
Recall that $\mathscr S^+$ is the set of positive
nonstationary solutions on the fixed cylinder and that $\mathcal M$
consists of those solutions whose derivative with respect to $s$ has
one strict sign on
$(0,1/2)\times\mathbb S^{N-1}_+$.
In the coordinates with period one, set
\[
 w_\mu(s,z)=v_\mu(T(\mu)s,z),\qquad
 \mathscr C_{\mathrm{loc}}=\{(T(\mu),w_\mu):0<\mu<\mu_0\}.
\]
The profiles $w_\mu$ are axial and even by construction.  Since
$\|v_\mu\|_{L^\infty}\to\infty$, they are nonstationary for small
$\mu_0$.  Thus $\mathscr C_{\mathrm{loc}}$ is a connected subset of
$\mathscr S^+$.

By Propositions \ref{prop:global-continuum}
and \ref{prop:uniform-bounds}, $\mathscr C$ is invariant under translation
by half a period and contains a blow-up sequence.  We may therefore choose
$(T_n,w_n)\in\mathscr C$ such that
\[
 T_n\longrightarrow T^*,\qquad
 \|w_n\|_{L^\infty}\longrightarrow\infty,\qquad
 -\partial_s w_n>0\quad\text{in }\mathscr H.
\]
Set $v_n(t,z)=w_n(t/T_n,z)$. By Lemma \ref{lem:unique-peak},
$v_n$ is centered at $P$, so $v_n(P)\to\infty$.
By Proposition \ref{prop:res-local-exhaustiveness}, for all large
$n$ we have $(T_n,w_n)=(T(\mu_n),w_{\mu_n})$ for some
$\mu_n\to0$.  Hence $\mathscr C_{\mathrm{loc}}$ meets $\mathscr C$.
By its connectedness and the maximality of $\mathscr C$ in
$\overline{\mathscr S^+}$, we obtain
$\mathscr C_{\mathrm{loc}}\subset\mathscr C$.

Every member of $\mathscr C_{\mathrm{loc}}$ now lies in $\mathcal M$.
The continuous function
$\mu\mapsto-\partial_s w_\mu(1/4,e_N)$ is therefore nonzero on
$(0,\mu_0)$ and is positive at $\mu_n$.
It follows that $-\partial_s w_\mu>0$ throughout $\mathscr H$ for
every $\mu\in(0,\mu_0)$, or equivalently,
\[
 -\partial_tv_\mu>0
 \qquad\text{in }(0,T(\mu)/2)\times\mathbb S^{N-1}_+.
\]
By Lemma \ref{lem:unique-peak}, each $v_\mu$ has its unique maximum
at $P$. By Proposition \ref{prop:global-continuum}, its
least period is $T(\mu)$.  Together with the smoothness in
Theorem \ref{thm:res-family} and the local uniqueness in
Proposition \ref{prop:res-local-exhaustiveness}, these facts prove
Theorem \ref{thm:main-resonant-family}.
\end{proof}

\appendix
\section{Technical estimates}
\label{sec:res-technical-appendix}

This appendix contains the two technical proofs deferred
from Section \ref{sec:resonant-end}: the projected bubble estimates in
Lemma \ref{lem:res-projection} and the weighted inverse estimate in
Lemma \ref{lem:res-weighted-inverse}.

\begin{proof}[Proof of Lemma \ref{lem:res-projection}]
Set
\[
 h_{\mu,i}:=(\mu\partial_\mu)^{i}\mathcal U_\mu^{\frac{N+2}{N-2}},
 \qquad i=0,1,2.
\]
We first prove the Green expansion \eqref{eq:res-Green-expansion}.
For $X\in\mathcal B_{r_0}(P)$, the projected bubble formula
\eqref{eq:res-projected-bubble}, the Newton representation on the full
cylinder, and $G_T=\Gamma+\mathcal H_T$ give
\[
 (\mu\partial_\mu)^{i}(W_{\mu,T}-\mathcal U_\mu)(X)
 =\int_{\Sigma_T}\mathcal H_T(X,Y)\chi(Y)h_{\mu,i}(Y)\dd V_Y
  -\mathcal E_{\mu,i}(X),
\]
where
\[
 \mathcal E_{\mu,i}(X):=
 \int_{\mathbb R\times\mathbb S^{N-1}}
   \Gamma(X,Y)h_{\mu,i}(Y)\dd V_Y
 -\int_{\Sigma_T}\Gamma(X,Y)\chi(Y)h_{\mu,i}(Y)\dd V_Y.
\]
Since $\operatorname{supp}\chi$ lies in the region on which
the identification with the fixed domain is the identity,
$\mathcal E_{\mu,i}$ is independent of $T$. By the
bubble tail estimate, the local integrability of $\Gamma$ and
$\nabla_X\Gamma$, and the decay of the kernel on the full cylinder, we obtain
\[
 \sum_{\ell=0}^1
 \|\nabla_X^\ell\mathcal E_{\mu,i}\|_
 {L^\infty(\mathcal B_{r_0}(P))}
 \leq C\mu^{\frac{N+2}{2}}.
\]
For $j,\ell\in\{0,1\}$, it follows from Taylor's formula in $Y$ and
Lemma \ref{lem:local-periodic-green} that, for
$X\in\mathcal B_{r_0}(P)$,
\[
\begin{aligned}
 \nabla_X^\ell\!\int_{\Sigma_T}\partial_T^j\mathcal H_T(X,Y)
       \chi(Y)h_{\mu,i}(Y)\dd V_Y
 &=\nabla_X^\ell\partial_T^j\mathcal H_T(X,P)
       \int_{\Sigma_T}\chi h_{\mu,i}\dd V +O\Bigl(\int_{\Sigma_T}r(Y)|h_{\mu,i}(Y)|\dd V_Y\Bigr)\\
 &=\frac{\alpha_N}{c_N}\left(\frac{N-2}{2}\right)^i\mu^{\frac{N-2}{2}}
       \nabla_X^\ell\partial_T^j\mathcal H_T(X,P)+O(\mu^{\frac{N}{2}}),
 \end{aligned}
\]
where in the last line we used Lemma \ref{lem:res-bubble-estimates}(iv),(v).
Combining this identity with the estimate for $\mathcal E_{\mu,i}$,
we obtain \eqref{eq:res-Green-expansion}.

We now prove Lemma \ref{lem:res-projection}(i).  Taking
$j=i=\ell=0$ in \eqref{eq:res-Green-expansion}, we obtain
\[
 |W_{\mu,T}(X)-\mathcal U_\mu(X)|\leq C\mu^{\frac{N-2}{2}},
 \qquad X\in\mathcal B_{r_0}(P).
\]
Since $\mathcal U_\mu(X)\asymp
\mu^{\frac{N-2}{2}}(\mu^2+r(X)^2)^{-\frac{N-2}{2}}$, after decreasing $r_0$ if necessary, the
$O(\mu^{\frac{N-2}{2}})$ remainder is absorbed by the bubble term.  This proves (i) for
$r(X)<r_0$.  By Lemma \ref{lem:res-bubble-estimates}, for
$r(X)\geq r_0$ we have
\[
 \begin{aligned}
 \int_{\mathcal B_{r_0/2}(P)}
 \mathcal U_\mu^{\frac{N+2}{N-2}}\dd V
 \asymp\mu^{\frac{N-2}{2}},\qquad
 \int_{\Sigma_T}|h_{\mu,i}|\dd V
 \leq C\mu^{\frac{N-2}{2}},\qquad
 |h_{\mu,i}|
 \leq C\mu^{\frac{N+2}{2}}
 ~~ \text{outside }\mathcal B_{r_0/2}(P).
 \end{aligned}
\]
The separated Green estimates in
Lemma \ref{lem:local-periodic-green} give, for $r(X)\geq4r_0$,
\[
 c\mu^{\frac{N-2}{2}}\delta(X)\leq W_{\mu,T}(X)\leq C\mu^{\frac{N-2}{2}}\delta(X).
\]
On the annulus $r_0\leq r(X)\leq4r_0$, we have
$\delta(X)=1$.  The points
$X$ and $Y\in\mathcal B_{r_0/2}(P)$ are uniformly separated, so
by positivity and compactness, we have $c\leq G_T(X,Y)\leq C$. Moreover,
\[
 \int_{\Sigma_T\setminus\mathcal B_{r_0/2}(P)}
 G_T(X,Y)\chi(Y)\mathcal U_\mu^{\frac{N+2}{N-2}}(Y)\dd V_Y
 \leq C\mu^{\frac{N+2}{2}}\int_{\Sigma_T}G_T(X,Y)\dd V_Y
 \leq C\mu^{\frac{N+2}{2}}.
\]
It follows that $W_{\mu,T}(X)\asymp\mu^{\frac{N-2}{2}}$ on this annulus, which
completes the proof of (i).

For (ii), taking $(j,i)=(1,0)$ and $(1,1)$ in
\eqref{eq:res-Green-expansion} and recalling that $\mathcal U_\mu$ is
independent of $T$ on the fixed neighborhood of $P$, we obtain
\[
 |\partial_TW_{\mu,T}(X)|
 +|\partial_T(\mu\partial_\mu W_{\mu,T})(X)|
 \leq C\mu^{\frac{N-2}{2}}
 \quad\text{if }r(X)<r_0.
\]
If $r(X)\geq4r_0$, the separated Green estimates give,
for $i=0,1$,
\[
 \begin{aligned}
 |\partial_T(\mu\partial_\mu)^{i}W_{\mu,T}(X)|
 \leq\int_{\Sigma_T}|\partial_TG_T(X,Y)|
      \chi(Y)|h_{\mu,i}(Y)|\dd V_Y
 \leq C\mu^{\frac{N-2}{2}}\delta(X).
 \end{aligned}
\]
On $r_0\leq r(X)\leq4r_0$, the singular part $\Gamma$ is independent
of $T$, while $\partial_TL_{g_T}$ is supported outside
$\mathcal B_{8r_0}(P)$.  Hence $\partial_TG_T(X,Y)$ is uniformly bounded for
$r_0\leq r(X)\leq4r_0$ and $Y\in\operatorname{supp}\chi$.  Therefore
\[
 |\partial_T(\mu\partial_\mu)^{i}W_{\mu,T}(X)|
 \leq C\int_{\Sigma_T}|h_{\mu,i}|\dd V
 \leq C\mu^{\frac{N-2}{2}},
 \qquad i=0,1.
\]
This proves (ii).

For (iii), we obtain from the Green representations above and
Lemma \ref{lem:res-bubble-estimates}(i), for $i=0,1,2$,
\[
 |\partial_T^j(\mu\partial_\mu)^{i}W_{\mu,T}(X)|
 \leq C W_{\mu,T}(X),\qquad j=0,1,
\]
which is (iii).

It remains to prove (iv).  Set
\[
 w_{j,i}:=\partial_T^j(\mu\partial_\mu)^{i}W_{\mu,T}.
\]
For $j=0$ we have
\[
 -L_{g_T}w_{0,i}=\chi h_{\mu,i}.
\]
Applying the scaled interior and boundary Schauder estimates
of Gilbarg--Trudinger \cite[Theorems 6.2 and 6.6]{GT2001} for homogeneous
Dirichlet data on balls of radius $2cr_\mu(X)$, we obtain
\[
 \begin{aligned}
 &\sum_{\ell=1}^2r_\mu(X)^\ell|\nabla^\ell w_{0,i}(X)|
 +r_\mu(X)^{2+\alpha}[\nabla^2w_{0,i}]_
 {C^\alpha(\mathcal B_{cr_\mu(X)}(X)\cap\Sigma_T)}\\
 &\quad\leq C\biggl(
 \|w_{0,i}\|_{L^\infty(\mathcal B_{2cr_\mu(X)}(X)\cap\Sigma_T)}
 +r_\mu(X)^2\|\chi h_{\mu,i}\|_{L^\infty}
 +r_\mu(X)^{2+\alpha}[\chi h_{\mu,i}]_{C^\alpha}
 \biggr),
 \end{aligned}
\]
where the norms on the last line are taken over the same ball.  By
(i), (iii), and parts (i)--(iii) of
Lemma \ref{lem:res-bubble-estimates}, the expression on the right is bounded by
$C\mu^{\frac{N-2}{2}}r_\mu(X)^{2-N}$.  This proves (iv) for $j=0$.

For $j=1$, differentiating on the fixed cylinder, we obtain
\[
 -L_{g_T}w_{1,i}=(\partial_TL_{g_T})w_{0,i}.
\]
The coefficients of $\partial_TL_{g_T}$ are uniformly bounded and
supported outside $\mathcal B_{8r_0}(P)$.  On this support,
$r_\mu$ is bounded from below, so by the estimate just proved
for $w_{0,i}$, we have
\[
 r_\mu^2\|(\partial_TL_{g_T})w_{0,i}\|_{L^\infty}
 +r_\mu^{2+\alpha}[(\partial_TL_{g_T})w_{0,i}]_{C^\alpha}
 \leq C\mu^{\frac{N-2}{2}}.
\]
Using the estimate of order zero in (iii), we apply the same
interior or boundary Schauder estimate to $w_{1,i}$ and obtain
\[
 \sum_{\ell=1}^2r_\mu^\ell|\nabla^\ell w_{1,i}|
 +r_\mu^{2+\alpha}[\nabla^2w_{1,i}]_{C^\alpha}
 \leq C\mu^{\frac{N-2}{2}}r_\mu^{2-N}.
\]
This proves (iv) and completes the proof.
\end{proof}

We next prove Lemma \ref{lem:res-weighted-inverse}.  

\begin{proof}[Proof of Lemma \ref{lem:res-weighted-inverse}]
By the Hardy estimate in Lemma \ref{lem:res-residual-estimates},
Lemma \ref{lem:res-scale-mode}, and
\eqref{eq:res-scale-mode-equation}, respectively, we have
\begin{equation}\label{eq:res-weight-preliminaries}
 \|\Theta_\mu\|_{H^{-1}}\leq C\rho_N(\mu),\qquad
 \mu^{N-2}|Z_{\mu,T}|\leq C\delta\Psi_\mu(r),\qquad
 \mu^{N-2}|(-L_{g_{\mathrm{cyl}}})Z_{\mu,T}|\leq C\Theta_\mu.
\end{equation}
Splitting at $r=r_0/2$ and integrating radially, we obtain
\begin{equation}\label{eq:res-weight-identities-two}
\begin{aligned}
 \int_{\Sigma_T}\Theta_\mu|Z_{\mu,T}|\dd V
 &\leq C\mu^N\int_0^{r_0/2}
       \frac{s^{N-1}\dd s}{(\mu^2+s^2)^{\frac{N+2}{2}}}
       +C\mu^{\frac{N-2}{2}}\Psi_\mu(r_0)
 \leq C\mu^{N-2}.
\end{aligned}
\end{equation}

\emph{Green potential estimate.}
We first show that the Green potential of $\Theta_\mu$ is bounded by
$C\delta\Psi_\mu$.  We use the estimates from
Lemma \ref{lem:local-periodic-green}.  Recall that $r_0$ lies below the
uniform injectivity radius of $\Sigma_T$, $T\in I$, and write
$r=r(X)$.

\emph{Case 1: $r(X)\leq\mu$.}
Then $\delta(X)=1$.  Set $F_\mu(y)=\exp_P(\mu y)$ and write
$X=F_\mu(x)$ and $Y=F_\mu(y)$, where $|x|\leq1$ and
$|y|<r_0/(2\mu)$.  The normal coordinate expansion gives
\[
 \mu^{-2}F_\mu^*g_{\mathrm{cyl}}
 =\bigl\{\delta_{ij}+O(\mu^2|y|^2)\bigr\}\dd y^i\dd y^j,
 \qquad
 \dd V_{F_\mu(y)}\asymp\mu^N\dd y,\qquad
 d_T(F_\mu(x),F_\mu(y))\asymp\mu|x-y|.
\]
Moreover, $r(F_\mu(y))=\mu|y|$,
$\Psi_\mu(r(X))=\mu^{\frac{N-2}{2}}$, and
$\Theta_\mu(F_\mu(y))=\mu^{\frac{N-6}{2}}(1+|y|^2)^{-2}$.
Splitting at $r(Y)=r_0/2$, we obtain from
\eqref{eq:res-source-weight-regions} and
Lemma \ref{lem:local-periodic-green}
\begin{equation*}
\begin{aligned}
  \int_{\Sigma_T}G_T(X,Y)\Theta_\mu(Y)\dd V_Y
  &\leq C\mu^{\frac{N-2}{2}}\int_{|y|<r_0/(2\mu)}
       |x-y|^{2-N}(1+|y|^2)^{-2}\dd y\\
  &\quad+C\Psi_\mu(r_0)\int_{\Sigma_T}G_T(X,Y)\dd V_Y\\
  &\leq C\mu^{\frac{N-2}{2}}\int_{\mathbb R^N}
       |x-y|^{2-N}(1+|y|^2)^{-2}\dd y+C\Psi_\mu(r_0)\\
  &\leq C\mu^{\frac{N-2}{2}}=C\Psi_\mu(r).
\end{aligned}
\end{equation*}

\emph{Case 2: $\mu<r(X)<r_0/2$.}
Then $\delta(X)=1$.  Write
$\exp_P^{-1}(Y)=s\omega$, where $s=r(Y)$ and
$\omega\in\mathbb S^{N-1}\subset T_P\Sigma_T$.  By the normal coordinate
volume expansion, uniform for $T\in I$, we have
\[
 \dd V_Y=(1+O(s^2))s^{N-1}\dd s\dd\omega
 \asymp s^{N-1}\dd s\dd\omega.
\]
Split the integral with respect to $Y$ into
$s<r/2$, $r/2\leq s\leq2r$, $2r<s<r_0$, and $s\geq r_0$.
On $s<r/2$, the triangle inequality gives
$d_T(X,Y)\geq r/2$; hence
\begin{align*}
 \int_{s<r/2}G_T(X,Y)\Theta_\mu(Y)\dd V_Y
 &\leq C\int_{s<r/2}d_T(X,Y)^{2-N}\Theta_\mu(Y)\dd V_Y
 \\
 &\leq Cr^{2-N}\mu^{\frac{N+2}{2}}\int_0^{r/2}
       \frac{s^{N-1}}{(\mu^2+s^2)^2}\dd s\\
 &\leq C\begin{cases}
  \mu^{3/2}r^{-1},&N=3,\\
  \mu^3r^{-2}\{1+\log(r/\mu)\},&N=4,\\
  \mu^{\frac{N+2}{2}}r^{-2},&N\geq5,
 \end{cases}\\
 &\leq C\Psi_\mu(r).
\end{align*}
On $2r<s<r_0$, one has $d_T(X,Y)\geq s/2$, and therefore
\begin{align*}
 \int_{2r<s<r_0}G_T(X,Y)\Theta_\mu(Y)\dd V_Y
 &\leq C\mu^{\frac{N+2}{2}}\int_{2r}^{r_0}
       \frac{s}{(\mu^2+s^2)^2}\dd s
       +C\Psi_\mu(r_0)\int_{2r}^{r_0}s\dd s \\
 &\leq C\mu^{\frac{N+2}{2}}\int_{2r}^{r_0}s^{-3}\dd s
       +Cr_0^2\Psi_\mu(r_0)\\
 &\leq C\bigl(\,\mu^{\frac{N+2}{2}}r^{-2}+\Psi_\mu(r_0)\bigr)\\
 &\leq C\Psi_\mu(r).
\end{align*}
In the last step,
\[
 \mu^{\frac{N+2}{2}}r^{-2}\leq C\Psi_\mu(r),
 \qquad \Psi_\mu(r_0)\leq C\Psi_\mu(r),
 \qquad \mu<r<r_0/2;
\]
for $N=4$ the logarithmic factor in $\Psi_\mu$ is at least one.
On $r/2\leq s\leq2r$, the integration region lies in
$\mathcal B_{3r}(X)$.  By local integrability of the Green singularity,
we obtain
\[
\begin{aligned}
 \int_{r/2\leq s\leq2r}G_T(X,Y)\Theta_\mu(Y)\dd V_Y
 &\leq \sup_{r/2<s<2r}\Theta_\mu(s)
       \int_{\mathcal B_{3r}(X)}G_T(X,Y)\dd V_Y\\
 &\leq C\sup_{r/2<s<2r}\Theta_\mu(s)
       \int_0^{3r}\rho\dd\rho\\
 &\leq Cr^2\{\mu^{\frac{N+2}{2}}r^{-4}+\Psi_\mu(r_0)\}\\
 &\leq C\{\mu^{\frac{N+2}{2}}r^{-2}+\Psi_\mu(r_0)\}
 \leq C\Psi_\mu(r).
\end{aligned}
\]
On $s\geq r_0$, the bound for the source weight and the
resolvent estimate give
\begin{align*}
 \int_{s\geq r_0}G_T(X,Y)\Theta_\mu(Y)\dd V_Y
 \leq C\Psi_\mu(r_0)\int_{\Sigma_T}G_T(X,Y)\dd V_Y
 \leq C\Psi_\mu(r_0)\leq C\Psi_\mu(r).
\end{align*}
Adding the contributions from the four regions, we obtain
\[
 \int_{\Sigma_T}G_T(X,Y)\Theta_\mu(Y)\dd V_Y
 \leq C\Psi_\mu(r).
\]

\emph{Case 3: $r(X)\geq r_0/2$.}
Then
$\Psi_\mu(r)\asymp\Psi_\mu(r_0)$.  We split the integral into
$s=r(Y)<r_0/4$ and $s\geq r_0/4$.  On the first region,
$d_T(X,Y)\geq r_0/4$.  The bound
$G_T(X,Y)\leq C\delta(X)$ then follows from
Lemma \ref{lem:local-periodic-green}.  Hence, in geodesic polar coordinates,
\begin{align*}
 \int_{s<r_0/4}G_T(X,Y)\Theta_\mu(Y)\dd V_Y
 &\leq C\delta(X)\mu^{\frac{N+2}{2}}
       \int_0^{r_0/4}\frac{s^{N-1}}{(\mu^2+s^2)^2}\dd s\\
 &\leq C\delta(X)\begin{cases}
  \mu^{3/2},&N=3,\\
  \mu^3\{1+|\log\mu|\},&N=4,\\
  \mu^{\frac{N+2}{2}},&N\geq5,
 \end{cases}\\
 &\leq C\delta(X)\Psi_\mu(r_0).
\end{align*}
On the second region,
\eqref{eq:res-source-weight-regions} gives
$\Theta_\mu(Y)\leq C\Psi_\mu(r_0)$; therefore
\begin{align*}
 \int_{s\geq r_0/4}G_T(X,Y)\Theta_\mu(Y)\dd V_Y
 \leq C\Psi_\mu(r_0)\int_{\Sigma_T}G_T(X,Y)\dd V_Y
 \leq C\delta(X)\Psi_\mu(r_0).
\end{align*}
Adding the contributions from the two regions, we obtain
\[
 \int_{\Sigma_T}G_T(X,Y)\Theta_\mu(Y)\dd V_Y
 \leq C\delta(X)\Psi_\mu(r).
\]
The three cases give the uniform Green potential estimate
\begin{equation}
 \int_{\Sigma_T}G_T(X,Y)\Theta_\mu(Y)\dd V_Y
 \leq C\delta(X)\Psi_\mu(r(X)).                                \label{eq:res-free-Green-map}
\end{equation}

\emph{Energy and Schauder estimates.}
Define the weighted supremum
\[
 S(\varphi):=\sup_{\Sigma_T^\circ}
 \frac{|\varphi|}{\delta\Psi_\mu}.
\]
By the energy inverse estimate \eqref{eq:res-energy-inverse},
the assumptions on $(h,\gamma)$, \eqref{eq:res-weight-preliminaries},
and \eqref{eq:res-rho}, we obtain
\[
 \|\varphi\|_T+|b|
 \leq C\{\|h\|_{H^{-1}}+|\gamma|\}
 \leq CA\{\rho_N(\mu)+\mu^{N-2}\}
 \leq CA\rho_N(\mu).
\]
Pairing \eqref{eq:res-augmented-linear-problem} with $Z_{\mu,T}$ and
using \eqref{eq:res-scale-residual-identity}, we obtain
\begin{equation}\label{eq:res-b-formula}
 b\|Z_{\mu,T}\|_T^2
 =\int_{\Sigma_T}\varphi\,\mu\partial_\mu R_{\mu,T}\dd V
  -\int_{\Sigma_T}hZ_{\mu,T}\dd V.
\end{equation}
By the $H^{-1}$ estimate in
Lemma \ref{lem:res-residual-estimates},
\eqref{eq:res-weight-identities-two}, and
\eqref{eq:res-rho}, we obtain
\[
 |b|\|Z_{\mu,T}\|_T^2
 \leq C\|\varphi\|_T\rho_N(\mu)
      +A\int_{\Sigma_T}\Theta_\mu|Z_{\mu,T}|\dd V
 \leq CA\{\rho_N(\mu)^2+\mu^{N-2}\}
 \leq CA\mu^{N-2}.
\]
Since $\|Z_{\mu,T}\|_T^2\to\kappa_N>0$, we conclude that
\begin{equation}\label{eq:res-preliminary-weighted-estimates}
 \|\varphi\|_T\leq CA\rho_N(\mu),
 \qquad |b|\leq CA\mu^{N-2}.
\end{equation}
By Lemmas \ref{lem:res-bubble-estimates}
and \ref{lem:res-projection}, together with
\eqref{eq:res-scale-mode-equation},
\[
 W_{\mu,T}^{\frac{4}{N-2}}\delta\Psi_\mu\leq C\Theta_\mu,\qquad
 \mu^{N-2}|(-L_{g_{\mathrm{cyl}}})Z_{\mu,T}|\leq C\Theta_\mu.
\]
Using the equation and \eqref{eq:res-preliminary-weighted-estimates}, we therefore obtain
\begin{gather*}
 |\frac{N+2}{N-2}W_{\mu,T}^{\frac{4}{N-2}}\varphi+h+b(-L_{g_{\mathrm{cyl}}})Z_{\mu,T}|
 \leq C\{S(\varphi)+A\}\Theta_\mu,\\
 r_\mu(X)^2\|W_{\mu,T}^{\frac{4}{N-2}}\|_
 {L^\infty(\mathcal B_{2cr_\mu(X)}(X)\cap \Sigma_T)}
 +r_\mu(X)^{2+\alpha}
 [W_{\mu,T}^{\frac{4}{N-2}}]_
 {C^\alpha(\mathcal B_{2cr_\mu(X)}(X)\cap \Sigma_T)}\leq C.
\end{gather*}

For the coefficient estimate, we use Lemma \ref{lem:res-projection}
and, for $s,t\geq0$, the inequalities
\[
 |s^{\frac{4}{N-2}}-t^{\frac{4}{N-2}}|\leq
 \begin{cases}
 |s-t|^{\frac{4}{N-2}},&1<\frac{N+2}{N-2}<2,\\
 C(s+t)^{\frac{6-N}{N-2}}|s-t|,&\frac{N+2}{N-2}\geq2.
 \end{cases}
\]
Indeed, on a ball scaled by $r_\mu(X)$, the function
$r_\mu(X)^{\frac{N-2}{2}}W_{\mu,T}$ has a uniformly bounded $C^1$ norm.
Raising it to the power $4/(N-2)$ therefore gives a uniformly bounded
$C^\alpha$ norm
since $\alpha<\min\{1,\frac{4}{N-2}\}$.
To estimate the data on these balls, we use
\[
 r_\mu(Y)^2\Theta_\mu(Y)\leq C\Psi_\mu(r(Y)),\qquad
 \Psi_\mu(r(Y))\asymp\Psi_\mu(r(X))
 \quad\text{if }d_T(X,Y)<4cr_\mu(X).
\]
The data hypothesis, a finite covering by the balls appearing there,
and these bounds control $r_\mu(X)^2$ times the scaled $C^\alpha$ norm of
$h+b(-L_{g_{\mathrm{cyl}}})Z_{\mu,T}$ by
$CA\Psi_\mu(r(X))$. Applying the interior or boundary Schauder estimate
to \eqref{eq:res-augmented-linear-problem}, and using
$|\varphi|\leq S(\varphi)\delta\Psi_\mu$ and
\eqref{eq:res-preliminary-weighted-estimates}, we obtain
\begin{equation}\label{eq:res-weighted-zero-order-reduction}
 \|\varphi\|+\mu^{2-N}|b|
 \leq C\{A+S(\varphi)\}.
\end{equation}
Thus it remains to prove
\begin{equation}\label{eq:res-zero-order-target}
 S(\varphi)\leq CA.
\end{equation}

\emph{Blow-up limit.}
For $\mu$ in a compact subinterval of $(0,\mu_0)$, the weights are
uniformly positive and bounded. By Lemma \ref{lem:res-inverse} and
elliptic regularity, the desired estimate is uniform on each such
subinterval. Decrease $\mu_0$ once more. If
\eqref{eq:res-zero-order-target} failed, we could therefore choose
$\mu_n\downarrow0$, $T_n\to T_\infty\in I$, and solutions with
$S(\varphi_n)/A_n\to\infty$.
Dividing each system by $S(\varphi_n)$ and retaining the notation for
the normalized data, we obtain, for some $\varepsilon_n\to0$,
\begin{equation}\label{eq:res-weighted-countersequence}
 \begin{gathered}
 S(\varphi_n)=1,\quad
 |\gamma_n|\leq\varepsilon_n\mu_n^{N-2},\quad
 |h_n(X)|+r_{\mu_n}(X)^\alpha
 [h_n]_{C^\alpha(\mathcal B_{cr_{\mu_n}(X)}(X)\cap\Sigma_{T_n})}
 \leq\varepsilon_n\Theta_{\mu_n}(X).
 \end{gathered}
\end{equation}
By \eqref{eq:res-preliminary-weighted-estimates} and
\eqref{eq:res-weighted-zero-order-reduction}, we have
\[
 \|\varphi_n\|\leq C,\qquad
 \|\varphi_n\|_{T_n}\leq C\varepsilon_n\rho_N(\mu_n),\qquad
 |b_n|\leq C\varepsilon_n\mu_n^{N-2}.
\]
Set $F_n(y)=\exp_P(\mu_n y)$ on $B_{r_0/\mu_n}$, and write
\[
 \widetilde\varphi_n=\mu_n^{-\frac{N-2}{2}}\varphi_n\circ F_n,
 \qquad \widehat W_n=\mu_n^{\frac{N-2}{2}}W_{\mu_n,T_n}\circ F_n,
 \qquad \widehat g_n=\mu_n^{-2}F_n^*g_{\mathrm{cyl}}.
\]
The rescaled equation is
\[
 \bigl(-\Delta_{\widehat g_n}+\frac{(N-2)^2}{4}\mu_n^2
       -\frac{N+2}{N-2}\widehat W_n^{\frac{4}{N-2}}\bigr)\widetilde\varphi_n
 =\mu_n^{\frac{6-N}{2}}
 \bigl(h_n+b_n(-L_{g_{\mathrm{cyl}}})Z_{\mu_n,T_n}\bigr)\circ F_n.
\]
On each fixed ball, the first term on the right is $O(\varepsilon_n)$
in $C^\alpha$, and the second is
$O(\mu_n^{2-N}|b_n|)=O(\varepsilon_n)$ in the same norm.
By the full weighted bound and local compactness, a subsequence
converges in $C^{2,\beta}_{\mathrm{loc}}(\mathbb R^N)$, $0<\beta<\alpha$,
to a solution of
$(-\Delta-\frac{N+2}{N-2}U_{\lambda_*}^{\frac{4}{N-2}})\widetilde\varphi=0$.
From \eqref{eq:res-solution-weight}, we also obtain
\begin{equation*}
 |\widetilde\varphi(y)|\leq C
 \begin{cases}
  (1+|y|)^{-1},&N=3,\\
  (1+|y|)^{-2}\{1+\log(2+|y|)\},&N=4,\\
 (1+|y|)^{-2},&N\geq5.
 \end{cases}
\end{equation*}
On each annulus $B_{|y|/4}(y)$ with $|y|\geq2$, the
rescaled equation and the interior gradient estimate give
\begin{equation*}
 |\nabla\widetilde\varphi(y)|
 \leq C|y|^{-1}
 \sup_{B_{|y|/4}(y)}|\widetilde\varphi|.
\end{equation*}
The preceding decay gives
$\nabla\widetilde\varphi\in L^2(\mathbb R^N)$ for $3\leq N\leq5$.
For $N\geq6$, Kato's inequality, the estimate
$\frac{N+2}{N-2}U_{\lambda_*}^{\frac{4}{N-2}}=O(r^{-4})$, and the exterior Green representation
give, whenever $\widetilde\varphi=O(r^{-\beta})$ and $|x|\geq2R$,
\begin{align*}
 |\widetilde\varphi(x)|
 \leq C|x|^{2-N}
 +C\int_{|y|>R}|x-y|^{2-N}|y|^{-\beta-4}\dd y
 \leq C
 \begin{cases}
 |x|^{-\beta-2},&\beta<N-4,\\
 |x|^{2-N}(1+\log|x|),&\beta=N-4,\\
 |x|^{2-N},&\beta>N-4.
 \end{cases}
\end{align*}
The first term bounds the decaying harmonic extension of the
boundary values on $\partial B_R$. Starting with $\beta=2$, we iterate
until the decay exponent is strictly greater than $(N-2)/2$; a logarithmic
factor at an equality case does not affect this strict inequality.
The annular gradient estimate then gives
$\nabla\widetilde\varphi\in L^2$ and
$\widetilde\varphi\in L^{\frac{2N}{N-2}}$.
Thus $\widetilde\varphi\in\mathcal D^{1,2}(\mathbb R^N)$ in every dimension.
We next pass the orthogonality condition to the Euclidean limit.
For fixed $R>1$, convergence in the core gives
\[
 \mu_n^{2-N}\int_{r<R\mu_n}\varphi_n\,
 \mu_n\partial_{\mu_n}(\mathcal U_{\mu_n}^{\frac{N+2}{N-2}})\dd V
 \longrightarrow
 \int_{|y|<R}\frac{N+2}{N-2}U_{\lambda_*}^{\frac{4}{N-2}}Z_0\widetilde\varphi\dd y.
\]
On the neck, \eqref{eq:res-weighted-countersequence},
Lemma \ref{lem:res-bubble-estimates}, and radial integration give
\[
 \limsup_{n\to\infty}\mu_n^{2-N}
 \int_{R\mu_n<r<r_0}
 \bigl|\varphi_n\,
 \mu_n\partial_{\mu_n}(\mathcal U_{\mu_n}^{\frac{N+2}{N-2}})\bigr|\dd V
 \leq C\begin{cases}
 R^{-3},&N=3,\\
 (1+\log R)R^{-4},&N=4,\\
 R^{-4},&N\geq5.
 \end{cases}
\]
By Lemma \ref{lem:res-bubble-estimates}, the fixed transition region
contributes $o(1)$. Since
\[
 \mu_n^{2-N}\langle\varphi_n,Z_{\mu_n,T_n}\rangle_{T_n}
 =\mu_n^{2-N}\gamma_n\longrightarrow0,
\]
letting first $n\to\infty$ and then $R\to\infty$, we obtain
\[
 \int_{\mathbb R^N}U_{\lambda_*}^{\frac{4}{N-2}}Z_0\widetilde\varphi\dd y=0.
\]
Since $\widetilde\varphi\in\mathcal D^{1,2}(\mathbb R^N)$, the nondegeneracy of standard bubble applies.  
The imposed symmetries exclude the
translation Jacobi fields, while the last identity excludes the dilation
Jacobi field;
hence $\widetilde\varphi=0$.
For every fixed $R>1$, local convergence gives
\begin{equation}\label{eq:res-inner-boundary-vanishing}
 \sup_{\partial \mathcal B_{R\mu_n}(P)}
 \frac{|\varphi_n|}{\Psi_{\mu_n}(R\mu_n)}\longrightarrow0.
\end{equation}

\emph{Exterior estimate.}
We now propagate the core vanishing to the exterior region.
Let $G_{R,\mu,T}$ denote the Dirichlet Green kernel on
$\Sigma_T\setminus \mathcal B_{R\mu}(P)$, and let $H_{R,\mu,T}$ solve
\[
 \begin{cases}
 -L_{g_{\mathrm{cyl}}}H_{R,\mu,T}=0
 &\text{in }\Sigma_T\setminus\overline{\mathcal B_{R\mu}(P)},\\
 H_{R,\mu,T}=1
 &\text{on }\partial\mathcal B_{R\mu}(P)\cap\Sigma_T^\circ,\\
 H_{R,\mu,T}=0
 &\text{on }\partial\Sigma_T.
 \end{cases}
\]
The local Green expansion gives
$G_T(X,P)\geq c(R\mu)^{2-N}$ on
$\partial \mathcal B_{R\mu}(P)$.  Comparing $H_{R,\mu,T}$ with
$C(R\mu)^{N-2}G_T(\cdot,P)$ and using the maximum principle, we obtain
\begin{equation}\label{eq:res-harmonic-measure}
 0\leq G_{R,\mu,T}\leq G_T,\qquad
 H_{R,\mu,T}(X)\leq C
 \begin{cases}
  (R\mu/r)^{N-2},&R\mu\leq r\leq r_0,\\
  \delta(X)(R\mu)^{N-2},&r\geq r_0.
 \end{cases}
\end{equation}
By the definitions in \eqref{eq:res-solution-weight} and
\eqref{eq:res-harmonic-measure}, separately on
$R\mu\leq r\leq r_0$ and $r\geq r_0$,
\begin{equation}\label{eq:res-harmonic-weight-map}
 \Psi_\mu(R\mu)H_{R,\mu,T}(X)
 \leq C\delta(X)\Psi_\mu(r(X)).
\end{equation}
To control the potential term, use $0\leq G_{R,\mu,T}\leq G_T$ and set
\[
 \mathcal K_{\mu,R}(X)
 =\int_{\Sigma_T\setminus \mathcal B_{R\mu}(P)}
 G_T(X,Y)W_{\mu,T}^{\frac{4}{N-2}}(Y)
 \delta(Y)\Psi_\mu(r(Y))\dd V_Y.
\]
For $R\mu<r(X)<r_0$, we have
$W_{\mu,T}^{\frac{4}{N-2}}(Y)\leq C\mu^2r(Y)^{-4}$.  Split the integral with respect to $Y$
into $R\mu<s<r/2$, $r/2\leq s\leq2r$, $2r<s<r_0$, and $s\geq r_0$.
Repeating the estimates in the proof of
\eqref{eq:res-free-Green-map}, we obtain
\begin{align*}
 \frac{\mathcal K_{\mu,R}(X)}{\Psi_\mu(r(X))}
 &\leq C\frac{r^{2-N}}{\Psi_\mu(r)}
       \int_{R\mu}^{r/2}\mu^2s^{N-5}\Psi_\mu(s)\dd s
   +C\mu^2r^{-2}
 +\frac{C}{\Psi_\mu(r)}
       \int_{2r}^{r_0}\mu^2s^{-3}\Psi_\mu(s)\dd s
   +o_\mu(1)\\
 &\leq o_\mu(1)+C
 \begin{cases}
  R^{-2},&N=3,\\
  (1+\log R)R^{-2},&N=4,\\
  R^{-2},&N\geq5
 \end{cases}.
\end{align*}
Set $\varepsilon_R=C(1+\log R)R^{-2}$ for $N=4$ and
$\varepsilon_R=CR^{-2}$ otherwise; then
$\varepsilon_R\to0$ as $R\to\infty$.
If $r(X)\asymp R\mu$, the same computation applies with the first
integral omitted.  If $r(X)\geq r_0$, then
$\Psi_\mu(r(X))=\Psi_\mu(r_0)$.  We split at $r(Y)=r_0/2$.
On $R\mu<r(Y)<r_0/2$, the variables are uniformly separated, so
$G_T(X,Y)\leq C\delta(X)$ and
\[
\begin{aligned}
 \int_{R\mu<r(Y)<r_0/2}G_T(X,Y)W_{\mu,T}^{\frac{4}{N-2}}(Y)
   \delta(Y)\Psi_\mu(r(Y))\dd V_Y
 &\leq C\delta(X)\mu^2
   \int_{R\mu}^{r_0/2}s^{N-5}\Psi_\mu(s)\dd s\\
 &\leq(\varepsilon_R+o_\mu(1))
   \delta(X)\Psi_\mu(r_0),
\end{aligned}
\]
where the last inequality follows from
\[
 \frac{\mu^2}{\Psi_\mu(r_0)}
 \int_{R\mu}^{r_0/2}s^{N-5}\Psi_\mu(s)\dd s
 \leq C\begin{cases}
  R^{-2},&N=3,\\
  \displaystyle\frac{1+\log R}
  {R^2\{1+\log(r_0/\mu)\}},&N=4,\\
  o_\mu(1),&N\geq5.
 \end{cases}
\]
On $r(Y)\geq r_0/2$, Lemma \ref{lem:res-projection} and
$\Psi_\mu(r(Y))\asymp\Psi_\mu(r_0)$ give
\[
 W_{\mu,T}^{\frac{4}{N-2}}(Y)\delta(Y)\Psi_\mu(r(Y))
 \leq C\mu^2\delta(Y)\Psi_\mu(r_0).
\]
By \eqref{eq:res-free-Green-map}, we therefore have
\[
 \int_{r(Y)\geq r_0/2}G_T(X,Y)W_{\mu,T}^{\frac{4}{N-2}}(Y)
 \delta(Y)\Psi_\mu(r(Y))\dd V_Y
 \leq C\mu^2\delta(X)\Psi_\mu(r_0).
\]
Combining the inner and outer estimates, we obtain
\begin{equation}\label{eq:res-potential-absorption}
 \sup_{X\in \Sigma_T^\circ\setminus \mathcal B_{R\mu}(P)}
 \frac{\mathcal K_{\mu,R}(X)}
 {\delta(X)\Psi_\mu(r(X))}
 \leq\varepsilon_R+o_\mu(1),
 \qquad \varepsilon_R\longrightarrow0.
\end{equation}
Kato's inequality gives, in distributions on the exterior
region,
\begin{equation*}
 -L_{g_{\mathrm{cyl}}}|\varphi_n|
 \leq \frac{N+2}{N-2}W_{\mu_n,T_n}^{\frac{4}{N-2}}|\varphi_n|
 +|h_n|+|b_n|\,|(-L_{g_{\mathrm{cyl}}})Z_{\mu_n,T_n}|.
\end{equation*}
We compare this subsolution with the Dirichlet solution having the
displayed nonnegative source and the same boundary values.
The inner boundary trace is controlled by
\eqref{eq:res-inner-boundary-vanishing}, and its lateral trace is zero.
Applying the comparison principle and the Green representation, together
with \eqref{eq:res-free-Green-map},
\eqref{eq:res-inner-boundary-vanishing}--
\eqref{eq:res-harmonic-weight-map}, and
\eqref{eq:res-potential-absorption}, we obtain
\begin{equation*}
 \sup_{\Sigma_{T_n}^\circ\setminus \mathcal B_{R\mu_n}(P)}
 \frac{|\varphi_n|}{\delta\Psi_{\mu_n}}
 \leq o(1)+(\varepsilon_R+o(1))
 \sup_{\Sigma_{T_n}^\circ\setminus \mathcal B_{R\mu_n}(P)}
 \frac{|\varphi_n|}{\delta\Psi_{\mu_n}}.
\end{equation*}
Choose $R$ so that $\varepsilon_R<1/2$ and then let
$n\to\infty$.  Together with the core limit, the last estimate
contradicts \eqref{eq:res-weighted-countersequence}.  Hence
$S(\varphi)\leq CA$.  Combining this bound with
\eqref{eq:res-weighted-zero-order-reduction}, we obtain
the estimate in Lemma \ref{lem:res-weighted-inverse} for $m=0$.

\emph{Derivative with respect to $T$.}
For each fixed $\mu>0$, the coefficient $W_{\mu,T}^{\frac{4}{N-2}}$ is
continuously differentiable in $T$ with values in $C^\alpha$.
Near the boundary this follows by writing
$W_{\mu,T}^{\frac{4}{N-2}}=\delta^{\frac{4}{N-2}}(W_{\mu,T}/\delta)^{\frac{4}{N-2}}$ and using
boundary division, positivity of the quotient, and $\alpha<\frac{4}{N-2}$.
By Lemma \ref{lem:res-inverse} and elliptic regularity, the augmented
linear operator is invertible on the fixed H\"older spaces. Its inverse,
and hence $(\varphi_T,b_T)$ for the stated data, is continuously
differentiable in $T$.

For $m=1$, we work on the fixed cylinder and write
$A_T=-L_{g_T}$ and $\mathcal L_T=A_T-\frac{N+2}{N-2}W_{\mu,T}^{\frac{4}{N-2}}$.  Then
\[
 \mathcal L_T\varphi=h+bA_TZ_{\mu,T},
 \qquad \langle\varphi,Z_{\mu,T}\rangle_T=\gamma.
\]
After pullback, \eqref{eq:res-scale-mode-equation} gives
\[
 A_TZ_{\mu,T}=\chi(\mu\partial_\mu)\mathcal U_\mu^{\frac{N+2}{N-2}}.
\]
This source is independent of $T$, since it is supported where
$\Phi_T$ is the identity.  In particular,
\[
 \langle\varphi,Z_{\mu,T}\rangle_T
 =\int_{\Sigma_{T^*}}(\Phi_T^*\varphi)\,
 \chi(\mu\partial_\mu)\mathcal U_\mu^{\frac{N+2}{N-2}}\dd V,
\]
and, by differentiation, we obtain the system
\begin{equation}\label{eq:res-differentiated-weighted-system}
 \begin{cases}
 \mathcal L_T(\partial_T\varphi)
 =\widetilde h+(\partial_T b)A_TZ_{\mu,T},\\
 \langle\partial_T\varphi,Z_{\mu,T}\rangle_T=\partial_T\gamma,
 \end{cases}
\end{equation}
where $$\widetilde h
 =\partial_T h-(\partial_T A_T)\varphi
 +\frac{4(N+2)}{(N-2)^2}
 W_{\mu,T}^{\frac{6-N}{N-2}}(\partial_T W_{\mu,T})\varphi.$$
Indeed, both the derivative of $A_TZ_{\mu,T}$ and the metric derivative
in the orthogonality pairing vanish in this representation with fixed support.

The case $m=0$ gives
$\|\varphi\|\leq CA$ and $|b|\leq CA\mu^{N-2}$.
The coefficients of $\partial_T A_T$ are uniformly bounded and supported
outside $\mathcal B_{8r_0}(P)$, where
$\Psi_\mu(r)=\Theta_\mu=\Psi_\mu(r_0)$. Thus
\[
 \|(\partial_T A_T)\varphi\|_{C^\alpha(\overline{\Sigma_{T^*}})}
 \leq CA\Psi_\mu(r_0).
\]
Since the support stays away from the cutoff region of $\Theta_\mu$,
this term satisfies the required weighted data estimate.

For the last term in $\widetilde h$, we first work on $r(X)<r_0/4$.
By Lemma \ref{lem:res-projection}, the weighted bound on $\varphi$,
and the scaled product estimate, we obtain
\[
 \begin{aligned}
 |W_{\mu,T}^{\frac{6-N}{N-2}}(\partial_T W_{\mu,T})\varphi|(X)
 +r_\mu(X)^\alpha
 [W_{\mu,T}^{\frac{6-N}{N-2}}(\partial_T W_{\mu,T})\varphi]_
 {C^\alpha(\mathcal B_{cr_\mu(X)}(X))}
 &\leq CAW_{\mu,T}(X)^{\frac{4}{N-2}}\Psi_\mu(r(X)) \\
 &\leq CA\Theta_\mu(X).
 \end{aligned}
\]
On a fixed neighborhood of the boundary, $\delta$ is the smooth boundary
distance.
By boundary division and Lemma \ref{lem:res-projection},
\[
 \begin{gathered}
 c\mu^{\frac{N-2}{2}}\leq\frac{W_{\mu,T}}{\delta}\leq C\mu^{\frac{N-2}{2}},\qquad
 \left\|\frac{W_{\mu,T}}{\delta}\right\|_{C^1}
 +\left\|\frac{\partial_T W_{\mu,T}}{\delta}\right\|_{C^1}
 \leq C\mu^{\frac{N-2}{2}},\qquad
 \left\|\frac{\varphi}{\delta}\right\|_{C^1}
 \leq CA\Psi_\mu(r_0).
 \end{gathered}
\]
Using the identity
\[
 W_{\mu,T}^{\frac{6-N}{N-2}}(\partial_T W_{\mu,T})\varphi
 =\delta^{\frac{N+2}{N-2}}\left(\frac{W_{\mu,T}}{\delta}\right)^{\frac{6-N}{N-2}}
   \frac{\partial_T W_{\mu,T}}{\delta}\frac{\varphi}{\delta},
\]
we therefore obtain, for every $\frac{N+2}{N-2}>1$,
\[
 \|W_{\mu,T}^{\frac{6-N}{N-2}}(\partial_T W_{\mu,T})\varphi\|_{C^1}
 \leq CA\mu^2\Psi_\mu(r_0).
\]
The same $C^1$ bound holds on $\{r\geq r_0/8\}$ away from this boundary
neighborhood, by the positive lower bound for $W_{\mu,T}/\mu^{\frac{N-2}{2}}$ there.
Since
\[
 \mu^2\Psi_\mu(r_0)\leq C\mu^{\frac{N+2}{2}}
 \leq C\Theta_\mu(X)\quad\text{for }r(X)\geq r_0/4,
\]
the exterior local H\"older seminorms satisfy the required estimate as
well. Combining these bounds with the hypothesis on $\partial_T h$, we obtain
\[
 |\widetilde h(X)|+r_\mu(X)^\alpha
 [\widetilde h]_{C^\alpha(\mathcal B_{cr_\mu(X)}(X)\cap\Sigma_T)}
 \leq CA\Theta_\mu(X).
\]
Moreover, $|\partial_T\gamma|\leq A\mu^{N-2}$ by hypothesis.
Thus \eqref{eq:res-differentiated-weighted-system} satisfies the
hypotheses of the already proved $m=0$ estimate, and hence
\[
 \|\partial_T\varphi\|+\mu^{2-N}|\partial_T b|\leq CA.
\]
Together with the $m=0$ bound, this proves
\eqref{eq:res-weighted-inverse} for $m=1$.
\end{proof}

\par\bigskip
\noindent\begin{minipage}{\linewidth}
\setlength{\parindent}{0pt}
H.-Y. Wang\\[4pt]
School of Mathematical Sciences, Laboratory of Mathematics and
Complex Systems, MOE, Beijing Normal University,
Beijing 100875, China\\[3pt]
Email: \textsf{wanghuayang@amss.ac.cn}

\par\vspace{12pt}
J. Xiong\\[4pt]
School of Mathematical Sciences, Laboratory of Mathematics and
Complex Systems, MOE, Beijing Normal University,
Beijing 100875, China\\[3pt]
Center for Basic Mathematics, Institute for Advanced Study,
Beijing Normal University, Beijing 100875, China\\[3pt]
Email: \textsf{jx@bnu.edu.cn}
\end{minipage}


\begin{thebibliography}{41}
\normalsize
\providecommand{\url}[1]{\texttt{#1}}
\providecommand{\href}[2]{#2}
\providecommand{\path}[1]{#1}

\bibitem{Aviles1987}
P. Aviles,
Local behavior of solutions of some elliptic equations,
Comm. Math. Phys. 108 (1987), 177--192.

\bibitem{AxlerBourdonRamey2001}
S. Axler, P. Bourdon, W. Ramey,
\textit{Harmonic Function Theory}, 2nd ed., Graduate Texts in Mathematics,
vol. 137, Springer, New York, 2001.

\bibitem{BVPV2007}
M.-F. Bidaut-V\'{e}ron, A. C. Ponce, L. V\'{e}ron,
Boundary singularities of positive solutions of some nonlinear elliptic
equations, C. R. Math. Acad. Sci. Paris 344 (2007), 83--88.

\bibitem{BVPV2011}
M.-F. Bidaut-V\'{e}ron, A. C. Ponce, L. V\'{e}ron,
Isolated boundary singularities of semilinear elliptic equations,
Calc. Var. Partial Differential Equations 40 (2011), 183--221.

\bibitem{BidautVeronVivier2000}
M.-F. Bidaut-V\'{e}ron, L. Vivier,
An elliptic semilinear equation with source term involving boundary
measures: the subcritical case,
Rev. Mat. Iberoam. 16 (2000), 477--513.

\bibitem{BrezisKato1979}
H. Br\'ezis, T. Kato,
Remarks on the Schr\"odinger operator with singular complex potentials,
J. Math. Pures Appl. (9) 58 (1979), 137--151.

\bibitem{BrezisTurner1977}
H. Br\'ezis, R. E. L. Turner,
On a class of superlinear elliptic problems,
Comm. Partial Differential Equations 2 (1977), 601--614.

\bibitem{CGS1989}
L. A. Caffarelli, B. Gidas, J. Spruck,
Asymptotic symmetry and local behavior of semilinear elliptic equations
with critical Sobolev growth, Comm. Pure Appl. Math. 42 (1989), 271--297.

\bibitem{ChenLin1997}
C.-C. Chen, C.-S. Lin,
Estimates of the conformal scalar curvature equation via the method of
moving planes,
Comm. Pure Appl. Math. 50 (1997), 971--1017.

\bibitem{CorriganFairlie1977}
E. Corrigan, D. B. Fairlie,
Scalar field theory and exact solutions to a classical $SU(2)$ gauge theory,
Phys. Lett. B 67 (1977), 69--71.

\bibitem{CR1971}
M. G. Crandall, P. H. Rabinowitz,
Bifurcation from simple eigenvalues, J. Funct. Anal. 8 (1971), 321--340.

\bibitem{Dancer1992}
E. N. Dancer,
Some notes on the method of moving planes,
Bull. Aust. Math. Soc. 46 (1992), 425--434.

\bibitem{Dancer2001}
E. N. Dancer,
New solutions of equations on ${\mathbb R}^n$,
Ann. Scuola Norm. Sup. Pisa Cl. Sci. (4) 30 (2001), 535--563.

\bibitem{delPinoDolbeaultMusso2004}
M. del Pino, J. Dolbeault, M. Musso,
The Br\'ezis--Nirenberg problem near criticality in dimension 3,
J. Math. Pures Appl. (9) 83 (2004), 1405--1456.

\bibitem{DMP2007}
M. del Pino, M. Musso, F. Pacard,
Boundary singularities for weak solutions of semilinear elliptic problems,
J. Funct. Anal. 253 (2007), 241--272.

\bibitem{Gidas1980}
B. Gidas,
Euclidean Yang-Mills and Related Equations,
in: C. Bardos, D. Bessis (Eds.),
\textit{Bifurcation Phenomena in Mathematical Physics and Related Topics},
NATO Advanced Study Institutes Series C, vol. 54,
D. Reidel, Dordrecht, 1980, pp. 243--267.

\bibitem{GidasNiNirenberg1979}
B. Gidas, W.-M. Ni, L. Nirenberg,
Symmetry and related properties via the maximum principle,
Comm. Math. Phys. 68 (1979), 209--243.

\bibitem{GidasSpruck1981Apriori}
B. Gidas, J. Spruck,
A priori bounds for positive solutions of nonlinear elliptic equations,
Comm. Partial Differential Equations 6 (1981), 883--901.

\bibitem{GidasSpruck1981}
B. Gidas, J. Spruck,
Global and local behavior of positive solutions of nonlinear elliptic
equations,
Comm. Pure Appl. Math. 34 (1981), 525--598.

\bibitem{GT2001}
D. Gilbarg, N. S. Trudinger,
\textit{Elliptic Partial Differential Equations of Second Order},
2nd ed., Classics in Mathematics, Springer, Berlin, 2001.

\bibitem{GruterWidman1982}
M. Gr\"uter, K.-O. Widman,
The Green's function for uniformly elliptic equations,
Manuscripta Math. 37 (1982), 303--342.

\bibitem{Gu1980}
C. H. Gu,
Conformally flat spaces and solutions to Yang--Mills equations,
Phys. Rev. D 21 (1980), 970--971.

\bibitem{GuoYan2024}
Y. Guo and S. Yan,
\newblock \textit{An elliptic problem with periodic boundary condition involving critical growth},
\newblock Math. Ann. \textbf{388} (2024), 795--830.

\bibitem{Han1991}
Z.-C. Han,
Asymptotic approach to singular solutions for nonlinear elliptic equations
involving critical Sobolev exponent,
Ann. Inst. H. Poincar\'e Anal. Non Lin\'eaire 8 (1991), 159--174.

\bibitem{HanLiLi}
Q. Han, X. Li, Y. Li,
Asymptotic expansions of solutions of the Yamabe equation and the
$\sigma_k$-Yamabe equation near isolated singular points,
Comm. Pure Appl. Math. 74 (2021), 1915--1970.

\bibitem{KMPS1999}
N. Korevaar, R. Mazzeo, F. Pacard, R. Schoen,
Refined asymptotics for constant scalar curvature metrics with isolated
singularities,
Invent. Math. 135 (1999), 233--272.

\bibitem{LiVairaWeiWu2025}
F. Li, G. Vaira, J. Wei, Y. Wu,
Construction of bubbling solutions of the Br\'ezis--Nirenberg problem in
general bounded domains (I): The dimensions 4 and 5,
J. Lond. Math. Soc. 112 (2025), Art. e70246.

\bibitem{Li1995}
Y. Y. Li,
Prescribing scalar curvature on $S^n$ and related problems, Part I,
J. Differential Equations 120 (1995), 319--410.

\bibitem{LierlSaloffCoste2014}
J. Lierl, L. Saloff-Coste,
The Dirichlet heat kernel in inner uniform domains: local results,
compact domains and non-symmetric forms,
J. Funct. Anal. 266 (2014), 4189--4235.

\bibitem{Lions1980}
P.-L. Lions,
Isolated singularities in semilinear problems,
J. Differential Equations 38 (1980), 441--450.

\bibitem{Nirenberg2001}
L. Nirenberg,
\textit{Topics in Nonlinear Functional Analysis}, Courant Lecture Notes in
Mathematics, vol. 6, Courant Institute of Mathematical Sciences, New York,
and American Mathematical Society, Providence, RI, 2001.

\bibitem{Nussbaum1975}
R. D. Nussbaum,
A global bifurcation theorem with applications to functional differential
equations, J. Funct. Anal. 19 (1975), 319--338.


\bibitem{Pohozaev1965}
S. I. Pohozaev,
On the eigenfunctions of the equation $\Delta u+\lambda f(u)=0$,
Soviet Math. Dokl. 6 (1965), 1408--1411;
translated from Dokl. Akad. Nauk SSSR 165 (1965), 36--39.

\bibitem{PremoselliVetois2022}
B. Premoselli, J. V\'etois,
Sign-changing blow-up for the Yamabe equation at the lowest energy level,
Adv. Math. 410 (2022), Part B, Art. 108769.

\bibitem{Rabinowitz1971}
P. H. Rabinowitz,
Some global results for nonlinear eigenvalue problems,
J. Funct. Anal. 7 (1971), 487--513.

\bibitem{Rey1990}
O. Rey,
The role of the Green's function in a non-linear elliptic equation involving
the critical Sobolev exponent, J. Funct. Anal. 89 (1990), 1--52.

\bibitem{Schoen1991}
R. M. Schoen,
On the number of constant scalar curvature metrics in a conformal class,
in H. B. Lawson, Jr.\ and K. Tenenblat (Eds.),
\textit{Differential Geometry: A Symposium in Honor of Manfredo do Carmo},
Pitman Monographs and Surveys in Pure and Applied Mathematics, vol. 52,
Longman Scientific \& Technical, Harlow, 1991, pp. 311--320.

\bibitem{SW2016}
N. Shioji, K. Watanabe,
Uniqueness and nondegeneracy of positive radial solutions of
${\rm div}(\rho\nabla u)+\rho(-gu+hu^{\frac{N+2}{N-2}})=0$,
Calc. Var. Partial Differential Equations 55 (2016), Art. 32.

\bibitem{WeiYan2011CriticalGrowth}
J. Wei, S. Yan,
Infinitely many positive solutions for an elliptic problem with critical or
supercritical growth,
J. Math. Pures Appl. (9) 96 (2011), 307--333.

\bibitem{Xiong2017}
J. Xiong,
The critical semilinear elliptic equation with isolated boundary
singularities, J. Differential Equations 263 (2017), 1907--1930.

\bibitem{Zettl2005}
A. Zettl,
\textit{Sturm--Liouville Theory},
Mathematical Surveys and Monographs, vol. 121,
American Mathematical Society, Providence, RI, 2005.

\end{thebibliography}
\end{document}